\PassOptionsToPackage{usenames, dvipsnames}{xcolor}
\documentclass[a4paper, 11pt]{preprint}
\pdfoutput=1

\usepackage{xcolor}

\usepackage{tikz}

\usepackage{tikz-cd}

\usepackage{enumitem}
\usepackage{amsfonts}
\usepackage{amsmath}

\usepackage{caption}
\usepackage{subcaption}
\crefname{subsection}{subsection}{subsections}

\usepackage{url}

\usepackage{algpseudocode}
\usepackage{algorithm}

\algrenewcommand{\algorithmiccomment}[1]{\hfill\textcolor{gray}{\texttt{\#} #1}}

\usepackage{soul}
\usepackage{nicematrix}

\newcommand{\PD}{\RM{PD}}

\newcommand{\edge}{\mathrel{-}}
\newcommand{\ot}{\leftarrow}
\newcommand{\pa}{\RM{pa}}
\newcommand{\nd}{\RM{nd}}
\newcommand{\de}{\RM{de}}
\newcommand{\an}{\RM{an}}
\newcommand{\ch}{\RM{ch}}
\renewcommand{\top}{\RM{top}}

\newcommand{\cir}{\RM{cir}}
\newcommand{\vcr}{\RM{vcr}}
\newcommand{\ecr}{\RM{ecr}}
\newcommand{\vc}{\RM{vc}}
\newcommand{\ec}{\RM{ec}}

\newcommand{\Real}{\mathbb{R}}
\newcommand{\PosReal}{\mathbb{R}_{>0}}

\DeclareMathOperator{\im}{im}
\DeclareMathOperator{\diag}{diag}
\DeclareMathOperator{\adj}{adj}
\DeclareMathOperator{\id}{id}

\DeclareMathOperator{\Var}{Var}
\newcommand{\fa}{\RM{fam}}

\TodoColor{Tobias}{cyan}
\TodoColor{Kaie}{orange}
\TodoColor{Pratik}{purple}
\TodoColor{Liam}{blue}

\title{Colored Gaussian directed acyclic graphical models}
\author{Tobias Boege}
\author{Kaie Kubjas}
\author{Pratik Misra}
\author{Liam Solus}

\address[T.~Boege \& L.~Solus]{Department of Mathematics, KTH Royal Institute of Technology, Sweden}
\email{post@taboege.de, solus@kth.se}

\address[K.~Kubjas]{Department of Mathematics and Systems Analysis, Aalto University, Finland}
\email{kaie.kubjas@aalto.fi}

\address[P.~Misra]{TUM School of Computation, Information and Technology, Technical University of Munich, Germany}
\email{pratik.misra@tum.de}

\date{\today}


\begin{document}

\begin{abstract}
We study submodels of Gaussian DAG models defined by partial homogeneity constraints imposed on the model error variances and structural coefficients.  We represent these models with colored DAGs and investigate their properties for use in statistical and causal inference.  Local and global Markov properties are provided and shown to characterize the colored DAG model.  Additional properties relevant to causal discovery are studied, including the existence and non-existence of faithful distributions and structural identifiability.  Extending prior work of Peters and Bühlmann and Wu and Drton, we prove structural identifiability under the assumption of homogeneous structural coefficients, as well as for a family of models with partially homogeneous structural coefficients.  The latter models, termed BPEC-DAGs, capture additional causal insights by clustering the direct causes of each node into communities according to their effect on their common target.  An analogue of the GES algorithm for learning BPEC-DAGs is given and evaluated on real and synthetic data.  Regarding model geometry, we provide a proof of a conjecture of Sullivant which generalizes to colored DAG models, colored undirected graphical models and directed ancestral graph models.  The proof yields a tool for identification of Markov properties for any rationally parameterized model with globally, rationally identifiable parameters.
\end{abstract}

\maketitle

\section{Introduction}
Directed acyclic graphs (DAGs) and their associated DAG models are fundamental to the field of causal inference; see for instance \cite{drton2018algebraic, koller2009probabilistic, maathuis2018handbook, Causality, peters2017elements}.
Although defined nonparametrically, a substantial amount of research focuses on DAG models for specified families of parametric distributions, with a popular choice being the Gaussian family.
A Gaussian DAG model is a linear structural equation model with independent, normally distributed errors where the structural equations are specified by the edges of the associated DAG.
Each edge $i \rightarrow j$ of the DAG is assigned a real-valued parameter $\lambda_{ij}$ serving as its associated structural coefficient.
When unconcerned with first moments, we may assume that the error variables associated to each node $i$ in the DAG have mean $0$ and therefore contribute a single additional parameter $\omega_i > 0$, i.e., its error variance.

When interpreted causally, the edges $i \rightarrow j$ of the DAG are taken to represent direct causal relations with the structural coefficient $\lambda_{ij}$ interpreted as the direct causal effect of the variable $X_i$ on $X_j$.
From the perspective of inference, two natural tasks arise: the first being to identify the direct causal relations from a random sample from the joint distribution contained in the DAG model, and the second being to identify the causal effects $\lambda_{ij}$.
In the case of Gaussian DAG models, the latter problem has a well-known solution, as the model is known to satisfy global rational identifiability \cite{drton2018algebraic}.
On the other hand, it is also well-known that different DAGs can represent the same collection of linear SEMs, which implies that the underlying causal DAG may not be identifiable from observational data alone.
This is a phenomenon termed \emph{Markov equivalence} of DAGs, which occurs more generally in the non-parametric setting.

Among the many advantages of DAG models is their admittance of local and global Markov properties, which, in the parametric setting, amount to a family of polynomial constraints that define the model.
By way of these constraints, several characterizations of Markov equivalence have been obtained \cite{AMPEquivalence, chickering2013transformational, VermaPearl}, and the corresponding graphical constraints have been utilized in the development of \emph{causal discovery algorithms} which are used to estimate the Markov equivalence class of the data-generating DAG from a random sample; see, for example, \cite{chickering2002optimal, solus2021consistency, spirtes1991algorithm}.

A natural consideration for these causal discovery algorithms is to identify the conditions under which they are consistent.
The conventional assumption under which we expect such an algorithm to be consistent is known as \emph{faithfulness}, which assumes that the data-generating distribution satisfies precisely the set of constraints associated to the global Markov property of the causal DAG.
An important feature of the Gaussian DAG models is that a generic distribution in the model is faithful to its DAG \cite{spirtes2000causation}, meaning that standard causal discovery algorithms will consistently estimate its Markov equivalence class (MEC).

However, from the causal perspective, an accurate estimate of the MEC of the causal DAG remains less than ideal, as it often leaves the direction of multiple edges in the graph undetermined; meaning that we cannot recover the direction of causation.
Hence, a substantial amount of research has focused on methods for determining the true causal DAG from within its MEC.
The gold standard approach, is to use available experimental data, which, depending on the nodes targeted, may completely or only partially refine the MEC \cite{hauser2012characterization, yang2018characterizing}.
Such methods are only applicable in situations where experimental data is available, which is often expensive or even unethical to obtain.

Alternatively, a growing body of research has focused on using additional parametric assumptions on the distribution to achieve \emph{structural identifiability}; i.e., the identification of the true DAG defining the data-generating SEM.
Methods observed to yield structural identifiability include, among others, using linear models with non-Gaussian errors \cite{shimizu2006linear, shimizu2014lingam}, using nonlinear models with additive noise \cite{hoyer2008nonlinear}, and in the Gaussian linear SEM setting, imposing homoscedasticity \cite{peters2014identifiability} or partial homoscedasticity \cite{wu2023partial} constraints.

The methods of \cite{peters2014identifiability, wu2023partial} obtain structural identifiability for Gaussian DAG models by imposing equality constraints $\omega_i = \omega_j$ on the error variances in the model.
Wu and Drton consider specifically \emph{partial} homoscedasticity constraints, in which nodes in the DAG are partitioned into classes in which their associated error variances are all equal.
To represent these constraints graphically, they color vertices in the DAG the same whenever their error variances are assumed to be equal.
This utilizes a special case of the more general representation of a \emph{colored DAG}, recently introduced in \cite{makam2022symmetries}, in which nodes are colored the same whenever $\omega_i = \omega_j$ and edges are colored the same whenever $\lambda_{ij} = \lambda_{k\ell}$.
In \cite{peters2014identifiability}, the authors motivated their homoscedasticity assumption as being applicable in situations where the variables are derived from similar domains; analogously one may interpret edge colors as representing similar causal effects.

\subsection*{Our contributions}
Based on the recent activity around colored Gaussian DAG models, we establish in this paper some basic properties of these models that will be of general use for their emerging applications. Our aims are three-fold:
The first is to establish extensions of the fundamental properties for (uncolored) Gaussian DAG models to the recently introduced colored DAG models of \cite{makam2022symmetries}.
In \Cref{sec:markovproperties}, we derive local and global Markov properties for the colored Gaussian DAG models.
In direct analogy to the uncolored Gaussian DAG models, we prove that these Markov properties each provide an alternative definition of the model via a collection of polynomial constraints satisfied by every distribution in the model.
We additionally provide some fundamental geometric properties of colored DAG models, deriving the model dimension and observing that each colored DAG model is a smooth submanifold of the positive definite cone (\Cref{thm:Smoothness}).
These properties allow for the use of standard techniques from large-sample asymptotic theory when performing likelihood ratio tests \cite{DrtonSmooth}. We also define the colored analogue of the conditional independence ideal for colored DAGs and use it to obtain an algebraic connection between the conditional independence ideal and the vanishing ideal. In the process, we prove a conjecture posed by Sullivant in \cite{Sullivant} and provide a generalized version of the same for undirected graphs and directed ancestral graphs.

Second, in \Cref{sec: model equivalence}, we investigate the existence of faithful distributions in colored DAG models and structural identifiability.
In \Cref{subsec: faithfulness}, we make precise the notion of faithfulness to a colored DAG and show that when a model has only colored edges or colored vertices then a generic distribution in the model will be faithful to its DAG.
In contrast to the uncolored setting, we observe that there exist colored DAG models that do not contain faithful distributions.
In \Cref{subsec: model equivalence edge}, we prove structural identifiability results providing edge-colored analogues to the identifiability results of both \cite{peters2014identifiability} and \cite{wu2023partial}.
Namely, we show that structural identifiability is obtained when all structural coefficients $\lambda_{ij}$ are equal, as well as in a special case where edges with equal structural coefficients are assumed target the same node.
The latter condition provides us with a family of colored DAG models which can be interpreted as clustering the direct causes of each node in the graph into \emph{causal communities} based on similar causal effects.
In \Cref{sec:causaldiscovery}, we provide an analogue of the GES causal discovery algorithm \cite{chickering2002optimal}.
Our algorithm provides estimates of the causal DAG together with additional information on how the causes are clustered into causal communities.
The method is evaluated on both real and synthetic data, where our algorithm generally appears to outperform GES, especially when learning dense models, while also offering the additional benefits of causal community identification and structurally identifiable DAG estimates.

Our third main contribution arises from our proof of the conjecture of Sullivant \cite{Sullivant} regarding the structure of uncolored Gaussian DAG models for several families of graphical models.
The conjecture was motivated by an effort to understand how the set of polynomial constraints defining the model via its Markov properties relates to the parametric definition of the model. %
Specifically, a \emph{Markov property} is a set of constraints characterizing a parametrically defined model that may be used to test data against the model (e.g., as done in basic constraint-based causal discovery algorithms).
In the Gaussian setting, both the parametric and constraint-based definitions of the DAG model have polynomial interpretations, and understanding their relation can provide insights into testable model constraints for rationally parametrized statistical models with no known Markov property.
The original conjecture of Sullivant for uncolored Gaussian DAG models was addressed in~\cite{SaturationProof}, however, there are inaccuracies in the details of the proof.
We provide a short proof of this conjecture that proves the result in the more general context of Gaussian colored DAG models, colored undirected graphical models \cite{UndirectedColored, lauritzen1996graphical}, as well as directed ancestral graph models \cite{richardson2002ancestral}.
Our proof method further yields a tool (\Cref{lemma:Saturation}) for identifying a Markov property (e.g., a set of testable, model-characterizing constraints) for rationally parameterized statistical models with globally, rationally identifiable parameters (see \Cref{rem:MarkovProperty}).
In \Cref{sec:discussion}, we end with a brief summary of future directions for further development of the family of colored Gaussian DAG models as motivated by the results derived in this paper.

\section{Preliminaries}
\label{sec: preliminaries}
This section summarizes basic notation and results for Gaussian DAG models that will be used throughout the paper.
For readers familiar with the basic theory of graphical models, this section serves mainly as a reference for notation.

\subsection{Graph theory}
\label{subsec: graphs}

Let $G=(V,E)$ be a directed acyclic graph (DAG) on vertex set $V$ with edges~$E$. We usually assume $V=[p] \defas\{1,\ldots, p\}$. An edge $i \to j$ is given by an ordered pair $(i,j)$ which we often denote as $ij$ for brevity. %
A \emph{topological ordering} of the DAG $G$ is a linear ordering $\pi = \pi_1\ldots \pi_p$ of its vertices such that $i$ precedes $j$ in $\pi$ whenever $ij$ is an edge in $G$. The DAG $G$ is \emph{naturally ordered} if $\pi = 1\ldots p$ is a topological ordering.
If $ij$ is an edge in $G$, then $i$ is a \textit{parent} of $j$ and $j$ is a \textit{child} of $i$ in $G$. If there exists an edge between $i$ and $j$ in either direction, they are \emph{adjacent}.
A sequence of distinct vertices $(i_1,\ldots, i_m)$ such that $i_k$ and $i_{k+1}$ are adjacent for all $k \in [m-1]$ is called a \textit{path} in $G$.
A path is called \textit{directed} if the edges are directed as $i_k\rightarrow i_{k+1}$ for all $k \in [m-1]$. If there exists a directed path from $i$ to $j$ in $G$, then $i$ is called an \textit{ancestor} of $j$ and $j$ is called a \textit{descendant} of $i$. If $j \neq i$ and $j$ is not a descendant of $i$, then $j$ is called a \textit{nondescendant} of $i$. A vertex $i$ is a \textit{source} node in $G$ if there is no incoming edge from any other vertex $j$ to $i$ (i.e., edge of the form $ji$) in the edge set. Similarly, $i$ is called a \textit{sink} node if there is no outgoing edge from $i$ to any other vertex. As all these definitions are specific to a given DAG $G$, we use the notation $\pa_G(i), \ch_G(i), \an_G(i), \de_G(i)$, and  $\nd_G(i)$ to denote the parents, children, ancestors, descendants and nondescendants of $i$ in $G$. We further use $\overline{\de}_G(i)=\{i\}\cup \de_G(i)$ to denote the \textit{closure} of descendants.

The \textit{skeleton} of a DAG $G$ is defined as the undirected version of the DAG, i.e., if $ij$ is an edge in $G$, then $i \edge j$ is an edge in the skeleton of $G$. The edges $ij$ and $kj$ are said to form a \textit{v-structure} in $G$ if $i$ and $k$ are not adjacent in $G$. An edge $ij$ is said to be \textit{covered} in $G$ if $\pa_G(j) = \pa_G(i)\cup \{i\}$.
If a path contains edges of the form $ij$ and $kj$, then $j$ is said to be a \textit{collider} vertex within that path. A \textit{trek} in $G$ from a vertex $i$ to a vertex $j$ is a pair $\tau = (P_L,P_R)$, where $P_L$ is a directed path from some vertex $s$ to $i$ and $P_R$ is a directed path from the same vertex $s$ to $j$. Here, $s$ is called the \textit{top-most} vertex of the trek and we write $s = \top(\tau)$.
If $A$, $B$, and $C$ are disjoint subsets of $V$, then $C$ \emph{d-separates} $A$ and $B$ if every path in $G$ connecting a vertex $i \in A$ to a vertex $j \in B$ contains a vertex $k$ that is either a non-collider that belongs to $C$ or a collider that does not belong to $C$ and has no descendants that belong to~$C$.

\subsection{Matrix notation and Gaussian conditional independence}
\label{subsec: Gaussian CI}

We work with symmetric matrices $\Sigma$ whose rows and columns are indexed by the set of vertices $V$ of a graph, and their submatrices. The entries of $\Sigma$ are denoted by lowercase $\sigma_{ij}$ with $i, j \in V$. Subsets of the index set $V$ are usually denoted by $A, B, C, K, L, \ldots$ and elements by $i, j, k, \ell, \ldots$. We use two notational conventions: (1) the union $A \cup B$ of index sets is often abbreviated to $AB$, and (2) an index $i$ is used interchangeably with the singleton subset $\Set{i}$. For example, $iA$ and $A \setminus i$ should implicitly be understood as the sets $\Set{i} \cup A$ and $A \setminus \Set{i}$, respectively, where $i \in V$ and $A \subseteq V$.

Given a real symmetric $V \times V$ matrix $\Sigma = (\sigma_{ij})$ and sets $A,B\subseteq V$, we let $\Sigma_{A,B} = (\sigma_{ij})_{i\in A, j\in B}$ denote the $A \times B$ submatrix of $\Sigma$ with rows indexed by $A$ and columns indexed by~$B$. The determinant of $\Sigma$ is $|\Sigma|$.
When $|A| = |B| = m$, the submatrix determinant $|\Sigma_{A,B}|$ is called an \emph{$m$-minor} of $\Sigma$. If $A = B$, we use the shorthand notation $\Sigma_A = \Sigma_{A,A}$ and the determinant $|\Sigma_A|$ is a \emph{principal minor}. A matrix is positive definite if all of its principal minors are positive. We denote the set of positive definite $V \times V$ matrices by~$\PD^V$.

Let $X \sim \mathrm{N}(\mu, \Sigma)$ be a multivariate Gaussian random vector with mean $\mu \in \BB R^V$ and covariance matrix $\Sigma \in \PD^V$. The conditional independence relations among the components of $X$ are completely determined by certain minors of~$\Sigma$. Namely, the following equivalence holds for all pairwise disjoint $I, J, K \subseteq V$:
\begin{align*}
  \label{eq:CI} \tag{$\CIperp$}
  \CI{X_I,X_J|X_K} \quad\Leftrightarrow\quad \text{$\Sigma_{IK,JK}$ has rank $|K|$}.
\end{align*}
See for example \cite[Proposition~4.1.9]{Sullivant} for a proof. Of special interest is the situation where $|I| = |J| = 1$. By \cite[Lemma~2.2]{Studeny}, these \emph{elementary} conditional independence relations completely characterize all conditional independences in~$X$. In this case the rank condition above simplifies to the vanishing of a single determinant of the form $|\Sigma_{ij|K}| \defas |\Sigma_{iK,jK}|$ which we call an \emph{almost-principal minor}.
We frequently use the Schur complement expansion of a determinant; cf.~\cite{Zhang}. In particular, for an almost-principal minor, this expansion is
$|\Sigma_{ij|K}| = |\Sigma_K|\left( \sigma_{ij} - \Sigma_{i,K} \, \Sigma_K^{-1} \, \Sigma_{K,j}\right)$.

\subsection{Gaussian DAG models}
\label{subsec: SEMs}
A \emph{linear structural equation model (SEM)} consists of a random vector $X = (X_i)_{i \in V}$ satisfying the relation
\[
X = \Lambda^T X + \eps,
\]
where $\eps = (\eps_i)_{i\in V}$ is a noise vector and $\Lambda = (\lambda_{ij})$ is a $V\times V$ real matrix. We work in the Gaussian setting where the error $\eps$ has mutually independent and normally distributed components, i.e., $\eps \sim \mathrm{N}(0, \Omega)$ where $\Omega = \diag((\omega_i)_{i \in V})$ is a diagonal matrix with positive diagonal.
The nonzero entries in the matrix $\Lambda$ determine the dependence relations (sometimes called the causal structure) amongst the variables in the system.
In this paper, we assume that this causal structure is representable by a DAG $G$; that is, we assume the matrix $\Lambda$ can be brought into strictly upper triangular form by a simultaneous permutation of its rows and columns (corresponding to a topological ordering of the vertices of $G$). In this case, an entry $\lambda_{ij} \neq 0$ corresponds to a directed edge $i\rightarrow j$ in~$G$. The value $\lambda_{ij}$ is sometimes called the \emph{causal effect} of $i$ on $j$, as it encodes the direct influence of $X_i$ on $X_j$ via the structural equations.

If $G$ is a DAG, the covariance matrix of a random vector in the linear structural equation model associated to $G$ is easily derived. Namely, for any choice of $\Omega \in \PosReal^V$ and $\Lambda$ in
\[
\Real^E = \{\Lambda = (\lambda_{ij}) \in\Real^{V\times V} : \lambda_{ij} = 0 \textrm{ whenever } ij\notin E\}
\]
the covariance matrix $\Sigma$ of $X$ is given by the image under the parametrization map
\begin{equation*}
    \begin{split}
        \phi_G: \PosReal^V \times \Real^E &\longrightarrow \PD^V, \\
                (\Omega, \Lambda) &\longmapsto (\BBm1_V - \Lambda)^{-T} \, \Omega \, (\BBm1_V - \Lambda)^{-1}.
    \end{split}
\end{equation*}
The \emph{Gaussian DAG model}, denoted $\CC M(G)$, is defined to be the collection of multivariate normal distributions with covariance matrix lying in the image of the map $\phi_G$.
Since we assume all errors have mean $0$, each covariance matrix corresponds to a unique distribution in the model, so we identify the set of distributions $\CC M(G)$ with the set of covariance matrices $\im(\phi_G)$; that is,  $\CC M(G) = \im(\phi_G)$.

Gaussian DAG models are well-understood, and known to admit several properties that are useful for inference.
For instance, the dimension of the model $\CC M(G)$ for $G = (V,E)$ is known to be $|V| + |E|$, and %
$\CC M(G)$ is a smooth submanifold of the positive definite cone.
This ensures that standard asymptotic theory for hypothesis testing can be used when performing statistical inference with these models \cite{DrtonSmooth}.

Gaussian DAG models also enjoy several properties relevant to the field of causality.
For instance, $\CC M(G)$ can be characterized as the set of distributions satisfying a collection of conditional independence relations specified by a so-called \emph{Markov property} with respect to $G = (V,E)$.
We say that a distribution $X \sim \textrm{N}(0,\Sigma)$ satisfies the
\begin{enumerate}
    \item \emph{local Markov property} with respect to $G$ if $\CI{X_i,X_{\nd_G(i)\setminus\pa_G(i)}|X_{\pa_G(i)}}$ for all $i\in V$;
    \item \emph{global Markov property} with respect to $G$ if $\CI{X_I,X_J|X_K}$ whenever $I$ and $J$ are d-separated given $K$ in $G$;
    \item \emph{ordered pairwise Markov property} with respect to $G$ if $\CI{X_i,X_j|X_{1,\ldots, j-1}}$ for all $i < j$ where $ij\notin E$.
\end{enumerate}

A fundamental result states that a distribution lies in $\CC M(G)$ if and only if it satisfies any one of the above Markov properties \cite{lauritzen1996graphical}.
\begin{theorem}
    \label{thm:uncoloredMP}
    Let $G = (V, E)$ be a DAG and $\Sigma \in \PD^V$. The following are equivalent:
    \begin{enumerate}[noitemsep, itemsep=0.3em]
        \item\label{thm:uncoloredMP:model} $\Sigma\in \CC M(G)$,
        \item\label{thm:uncoloredMP:local} $\Sigma$ satisfies the local Markov property with respect to $G$,
        \item\label{thm:uncoloredMP:global} $\Sigma$ satisfies the global Markov property with respect to $G$, and
        \item\label{thm:uncoloredMP:ordered} $\Sigma$ satisfies the ordered pairwise Markov property with respect to $G$.
    \end{enumerate}
\end{theorem}
The equivalence of \eqref{thm:uncoloredMP:model}--\eqref{thm:uncoloredMP:global} in \Cref{thm:uncoloredMP} in fact does not require Gaussianity; e.g., it holds in the general (nonparametric) setting. %
The global Markov property is referred to as such since it describes all conditional independence relations that a distribution is required to satisfy if it belongs to the model $\CC M(G)$.
Namely, it is \emph{complete}, meaning that any conditional independence relation $\CI{X_I,X_J|X_K}$ satisfied by all $\Sigma\in \CC M(G)$ is represented by a d-separation relation in $G$; i.e., $I$ and $J$ are d-separated given $K$ in $G$.
The fact that the global Markov property is complete for Gaussian DAG models can be seen from the existence of distributions in the model $\CC M(G)$ that are faithful to $G$.
Namely, a distribution $\Sigma\in \CC M(G)$ is said to be \emph{faithful} to $G$ if $I$ and $J$ are d-separated given $K$ in $G$ whenever $\CI{X_I,X_J|X_K}$.

The completeness of the global Markov property allows us to recover a characterization of the DAGs~$G$ and $H$ that satisfy $\CC M(G) = \CC M(H)$ via purely graph-theoretic means.
We say that $G$ and $H$ are \emph{Markov- (or model-) equivalent} if $\CC M(G) = \CC M(H)$, and we call the set of all DAGs Markov equivalent to $G$ its \emph{Markov equivalence class (MEC)}.
\begin{theorem}[\cite{VermaPearl}]
    \label{thm:vp}
    Two graphs $G$ and $H$ are Markov equivalent if and only if they have the same skeleton and v-structures.
\end{theorem}
The Markov properties and characterizations of Markov equivalence play a fundamental role in the problem of causal discovery, where one aims to learn the DAG structure $G$ from data.

By the \emph{trek rule} \cite[Proposition~14.2.13]{Sullivant}, a matrix $\Sigma = (\sigma_{ij})$ satisfies $\Sigma = \phi_G(\Omega, \Lambda)$ for some $(\Omega, \Lambda)\in \PosReal^V\times\Real^E$ (i.e., $\Sigma$ belongs to the model $\CC M(G)$) if and only if
\begin{align*}
  \label{eq:TrekRule} \tag{T}
  \sigma_{ij} = \sum_{\substack{\text{$\tau$ : trek} \\ \text{from $i$ to $j$}}} \omega_{\top(\tau)} \lambda^\tau,
\end{align*}
for all $i,j \in V$, where $\omega_{\top(\tau)} \lambda^\tau$ is the \emph{trek monomial} of~$\tau$: namely, if $\tau$ consists of two directed paths $s = i_n \to i_{n-1} \to \dots \to i_0 = i$ and $s = j_m \to j_{m-1} \to \dots \to j_0 = j$, then $\top(\tau) = s$ and $\lambda^\tau = \prod_{\ell=1}^n \lambda_{i_{\ell-1} i_{\ell}} \cdot \prod_{\ell=1}^m \lambda_{j_{\ell-1} j_{\ell}}$.

The DAG model $\CC M(G)$ also admits parameter identifiability results that are useful in causal inference.
It is well-known that the parameters $(\Omega, \Lambda)$ for which $\Sigma = \phi_G(\Omega, \Lambda)\in \CC M(G)$ are \emph{rationally identifiable}, meaning that there are rational functions of the $\sigma_{ij}$ coordinates yielding the $\omega_i$ and $\lambda_{ij}$ parameter values defining $\Sigma$. These rational functions are shown in \Cref{lemma:Ident}.

\subsection{Algebra preliminaries}
\label{subsec: algebra preliminaries}
Several results in this paper, including the structural identifiability results, do not yield to standard proof techniques from probability theory.
However, proofs may be obtained using techniques from algebra, for which we now recall the basics relevant to statistical problems.
Recall that Gaussian DAG models may be defined via conditional independence constraints coming from d-separations.
A conditional independence statement $\CI{X_i,X_j|X_K}$ for a Gaussian distribution is equivalent to a polynomial equation in its covariance matrix, namely the vanishing of the subdeterminant $|\Sigma_{ij|K}|$ introduced in \Cref{subsec: Gaussian CI}.
It is well-known that the parameters $\omega_i$ and $\lambda_{ij}$ of the structural equations can be recovered from the covariance matrix using rational functions (see also \Cref{subsec: parameter idenfitication} for more details).
This suggests to study Gaussian DAG~models via functions from the covariance matrices to the field of real numbers.
We now review some concepts from algebra which are required for our geometric~approach.
More details can be found in any book on commutative algebra such as \cite{CommAlgebraBook} and \cite{Kemper}.
\subsubsection{Polynomial rings}
We denote by $\BB R[x]:=\BB R[x_1,x_2,\ldots,x_n]$ the \emph{polynomial ring} with variables $x_1,x_2,\ldots,x_n$ and real coefficients.
For a given DAG $G = (V,E)$, we have two polynomial rings. The first is $\BB R[\Sigma]$ with variables $\sigma_{ij}$, where $i,j \in V$ and $\sigma_{ij} = \sigma_{ji}$ reflecting the symmetry of the covariance matrix. The second ring is $\BB R[\Omega,\Lambda]$ in variables $\omega_i$, where $i \in V$, and $\lambda_{ij}$, where $ij \in E$.
Given two affine spaces $\BB R^n$ and $\BB R^m$ with corresponding polynomial rings $\BB R[x] = \BB R[x_1, \ldots, x_n]$ and $\BB R[t] = \BB R[t_1, \ldots, t_m]$ and a map $f\colon \BB R^m \to \BB R^n$, its \emph{pullback} is the map $f^*\colon \BB R[x] \to \BB R[t]$ defined by $f^*(\alpha) = \alpha \circ f$ for $\alpha \in \BB R[x]$. This map is a \emph{ring homomorphism}, meaning that it satisfies $f^*(\alpha + \beta) = f^*(\alpha) + f^*(\beta)$ and $f^*(\alpha\beta) = f^*(\alpha) f^*(\beta)$ for all $\alpha, \beta \in \BB R[x]$.  For a Gaussian DAG model the parametrization $\phi_G\colon \mathbb{R}_{>0}^V \times \mathbb{R}^E \to \PD^V$ naturally has a pullback $\phi_G^*\colon \mathbb{R}[\Sigma] \to \mathbb{R}[\Omega, \Lambda]$. It sends any function $\alpha(\Sigma)$ whose input is a covariance matrix to the corresponding function $\alpha(\phi_G(\Omega, \Lambda))$ whose inputs are the parameters of the model.
\subsubsection{Ideals}
A subset $I \subseteq \BB R[x]$ is an \emph{ideal} if (1) $0 \in I$, (2) $f,g \in I$ imply $f+g \in I$, and (3) $f \in I$ and $h \in \BB R[x]$ imply $hf \in I$. To any set of polynomials $f_1,\ldots,f_s \in \BB R[x]$ we associate the set consisting of all finite sums of the form $\sum_{i=1}^s h_i f_i$ with $h_1,\ldots,h_s \in \BB R[x]$. This set is denoted $\langle f_1, \ldots, f_s\rangle$ and it is the smallest ideal containing $f_1, \ldots, f_s$. If $I = \langle f_1, \ldots, f_s\rangle$, then $\{f_1, \ldots, f_s\}$ is called a \emph{generating set} for~$I$.
The \emph{sum} of two ideals $I,J \subseteq \BB R[x]$ is the set $I+J=\{f+g:f \in I, g \in J\}$, which can be shown to be an ideal.
Lastly, an ideal $I \subseteq \BB R[x]$ is \emph{prime} if $fg \in I$ for some $f,g \in \BB R[x]$ necessitates that $f \in I$ or $g \in I$.

By rewriting the global Markov property via \eqref{eq:CI} we get the following algebraic description of the covariance matrices of distributions in a Gaussian DAG model via an ideal:
\begin{lemma}
    \label{lem: CI polynomials}
    The Gaussian DAG model $\CC M(G)$ is the collection of all Gaussian distributions with covariance matrices $\Sigma\in \PD^V$ on which all the polynomials in
    \[
    J_G = \ideal{\,
    \mbox{$(|K| + 1)$-minors of } \Sigma_{IK,JK} : \mbox{$I$ and $J$ are d-separated given $K$ in $G$}\,
}
    \]
    evaluate to zero.
\end{lemma}
The collection of polynomials $J_G$ is sometimes called the \emph{(global) conditional independence ideal} of the model $\CC M(G)$ \cite{AlgGeoOfGBN}.
As hinted in \Cref{lem: CI polynomials}, ideals~capture algebraic properties of collections of functions which evaluate to zero on a set of points.
For example, assume that $\alpha_1, \ldots, \alpha_s \in \mathbb{R}[\Sigma]$ all satisfy $\alpha_i(\Sigma_0) = 0$ for some covariance matrix~$\Sigma_0$.
Then every polynomial in the ideal $\langle \alpha_1, \ldots, \alpha_s\rangle$ will also evaluate to zero on~$\Sigma_0$.
Given a set of polynomials $F \subseteq \BB R[x]$, the \emph{vanishing locus} of $F$ or the \emph{affine variety} defined by $F$ is the set $\{a \in \BB R^n: f(a)=0 \quad \forall f \in F\}$. Conversely, given any set $\mathcal{M} \subseteq \BB R^n$, the set $\{ f \in \BB R[x] : f(a) = 0 \quad \forall a \in \mathcal{M}\}$ is an ideal and is called the \emph{vanishing ideal} of~$\mathcal{M}$.
In the context of statistics, the set $\CC M$ could be, for example, the set of covariance matrices in $\mathbb{R}^{V\times V}$ of the distributions in a Gaussian DAG model.
The ideal $J_G$ is generally not the same as the vanishing ideal of $\CC M(G)$, since the vanishing ideal contains \emph{all} polynomials vanishing on $\CC M$.
The ideal $J_G$ may only contain some of these polynomials, and hence may be missing polynomial constraints needed in certain statistical analyses.

A set is \emph{irreducible} if its vanishing ideal is prime.
Given an irreducible set $\mathcal{M} \subseteq \BB R^n$, we say that a property $P(a)$ holds \emph{generically} (or \emph{for a generic point}) on $\mathcal{M}$ if there exists a polynomial $f \in \BB R[x]$ which does not vanish on every point of $\mathcal{M}$ and such $\text{$P(a)$ fails}$ implies that $f(a) = 0$, so that $P(a)$ holds for all points of $\mathcal{M}$ outside of the vanishing locus of~$f$. By irreducbility and a dimension argument, this implies that $P(a)$ holds for almost all points in $\mathcal{M}$ with respect to the Lebesgue measure.
\subsubsection{Multiplicatively closed sets}
A set $S \subseteq \BB R[x]$ is \emph{multiplicatively closed} if (1) $1 \in S$ and (2) $f,g \in S$ implies $fg \in S$. With such a set one can define the \emph{localization} $S^{-1}\BB R[x]$ of $\BB R[x]$ at $S$ which consists of all rational functions $\frac{f(x)}{g(x)}$, where $f(x) \in \BB R[x]$ and $g(x) \in S$. This is a version of the polynomial ring where division by elements of~$S$ is allowed.
We usually apply this technique in cases where the elements of $S$ are known to never vanish on the set of points $\CC M \subseteq \BB R^n$.
For example, our set $\CC M$ may be the set of covariance matrices corresponding to the distributions in a Gaussian DAG model.
Since a covariance matrix $\Sigma$ in a Gaussian DAG model is positive definite, we may localize at the multiplicatively closed set generated by all the principal minors $|\Sigma_A|$ for $A \subseteq V$.
Relating this to ideals and their use in describing polynomials that vanish on a set of points $\CC M$,
if a product $fs$ is contained in an ideal $I$ (so $fs$ vanishes on all points in $\CC M$) and $s \in S$ is known to never vanish, then this means that $f$ should vanish. To incorporate this knowledge that $s$ never vanishes into the ideal, we have the following construction:

Given an ideal $I$ and a multiplicatively closed set $S$ in $\BB R[x]$, the \emph{saturation} of $I$ by $S$ is the ideal $I:S=\{f \in \BB R[x]: fg \in I \text{ for some } g \in S\}$.
One of our main results in \Cref{subsec:conj} gives a description of the vanishing ideal of a colored Gaussian DAG model as a saturation of a \emph{colored conditional independence ideal}.
This means that we know all vanishing polynomials on these models and, as a consequence, are able to tell when two polynomial functions of the covariance matrices are equal on a model.
This information can be used to prove structure identifiability results as we do in \Cref{subsec: model equivalence edge}.

\section{Markov properties for colored Gaussian DAG models} \label{sec:markovproperties}
\label{sec:coloredDAGmodels}

A \emph{colored DAG} is a triple $(V, E, c)$ consisting of a directed acyclic graph $G = (V, E)$ together with a \emph{coloring map} $c: V \sqcup E \to C$, where $C$ is a finite set and $c(V)\cap c(E) = \emptyset$.
We will often denote the colored DAG $(V, E, c)$ simply by $(G,c)$.
The image of a vertex or edge under this map is referred to as its \emph{color}. The coloring induces a partition on vertices and edges into \emph{color classes}. We denote the number of vertex and edge color classes by $\vc$ and $\ec$, respectively.
\begin{convention}
\label{conv: coloring}
When drawing colored DAGs, white nodes and black edges are ``uncolored,'' i.e., they belong to a color class containing exactly one element (themselves).  That is, a white node implicitly has a color different from every other node (including other white nodes) and similarly for black edges.
\end{convention}

For $\Sigma$ in the DAG model $\CC M(G)$, with corresponding structural equations
\begin{equation*}
\begin{split}
    X_i = \sum_{k\in\pa_G(i)}\lambda_{ki}X_k + \eps_i, \qquad i\in V
\end{split}
\end{equation*}
where $\eps_i\sim \textrm{N}(0,\omega_i)$, we say that $\Sigma$ is \emph{Markov} to the colored DAG $(G,c)$ if $\omega_i = \omega_k$ whenever $c(i) = c(k)$ and $\lambda_{ij} = \lambda_{k\ell}$ whenever $c(ij) = c(k\ell)$.

\begin{definition}
    \label{def:coloredDAGmodel}
    The \emph{colored Gaussian DAG model} $\CC M(G,c)$ for the colored DAG $(G,c)$ is the set of all $\Sigma\in \PD^V$ that are Markov to $(G,c)$.
    That is,
    \[
    \CC M(G,c) = \Set{\phi_{G,c}(\Omega, \Lambda) \in \PD^V : (\Omega, \Lambda)\in A(G,c)},
    \]
    where
    \begin{equation*}
        \begin{split}
            A(G,c) = \{\,
    (\Omega, \Lambda) \in \PosReal^V \times \Real^{V\times V} :{}&
    \lambda_{ij} = \lambda_{c(ij)} \, \,  \forall ij\in E, \\
    &\omega_i = \omega_{c(i)} \, \,  \forall i \in V, \\
    &\lambda_{ij} = 0 \, \,  \forall ij\notin E, \\
    &\mbox{and } \lambda_{c(ij)}\in \Real, \omega_{c(i)}\in \PosReal
    \,\}
        \end{split}
    \end{equation*}
    is the colored parameter space and
    \begin{equation*}
    \begin{split}
        \phi_{G,c}: A(G,c) &\longrightarrow \PD^V \\
                    (\Omega, \Lambda) &\longmapsto (\BBm1_V - \Lambda)^{-T} \, \Omega \, (\BBm1_V - \Lambda)^{-1}.
    \end{split}
\end{equation*}
\end{definition}

The colored Gaussian DAG model $\CC M(G,c)$ is therefore the image of the parametrization $\phi_{G,c}$ arising from $\phi_G$ (see \Cref{subsec: SEMs}) by replacing all occurrences of parameters in the same color class with a single parameter corresponding to the color class.

We introduce some further terminology that will be used in the paper.
Let $(G,c)$ be a colored DAG and $\mathbf{c}$ a color class for $(G,c)$.
The color class $\mathbf{c}$ is either a set of nodes or a set of edges of $G$.
If $i$ indexes the smallest vertex or $ij$ indexes the smallest edge in $\mathbf{c}$ with respect to the lexicographic ordering  from the right on the elements of $\mathbf{c}$, then $\sigma_{ii}$ and $\sigma_{ij}$, respectively, are called the \emph{base variables} for $\mathbf{c}$, and correspondingly $\omega_i$ and $\lambda_{ij}$ are the \emph{base parameters}.

The uncolored model $\CC M(G)$ is obtained as a special case of the colored model by choosing $C = V \sqcup E$ and $c$ to be the identity map, i.e., each vertex and edge has its own, unique color.
In this case, we say that the DAG is \emph{uncolored}.
Analogously, we say that $(G,c)$ is \emph{vertex-colored} if $c|_E$ is the identity and \emph{edge-colored} if $c|_V$ is the identity.
Note that $\CC M(G,c)\subseteq \CC M(G)$ for all colorings $c$ of $G$; that is, all colored DAG models are submodels of the corresponding uncolored DAG model.

\begin{example}
Consider the following colored DAGs:

\begin{minipage}{\linewidth}
\begin{minipage}{.3\linewidth}
\begin{center}
\begin{tikzpicture}[thick,scale=1]

 	 \node[circle, draw, fill=white, inner sep=1pt, minimum width=1pt] (1) at (0,0)  {$1$};
	   \node[circle, draw, fill=white, inner sep=1pt, minimum width=1pt] (2) at (1,0) {$2$};
          \node[circle, draw, fill=white, inner sep=1pt, minimum width=1pt] (3) at (2,0) {$3$};
          \node[circle, draw, fill=white, inner sep=1pt, minimum width=1pt] (4) at (3,0) {$4$};

          \draw[->,black]   (1) -- (2) ;
          \draw[->,black]   (2) -- (3) ;
          \draw[->,black]   (3) -- (4) ;
    \end{tikzpicture}
\end{center}
\end{minipage}
\begin{minipage}{.3\linewidth}
\begin{center}
\begin{tikzpicture}[thick,scale=1]

 	 \node[circle, draw, fill=red!50, inner sep=1pt, minimum width=1pt] (1) at (0,0)  {$1$};
	   \node[circle, draw, fill=white, inner sep=1pt, minimum width=1pt] (2) at (1,0) {$2$};
          \node[circle, draw, fill=red!50, inner sep=1pt, minimum width=1pt] (3) at (2,0) {$3$};
          \node[circle, draw, fill=white, inner sep=1pt, minimum width=1pt] (4) at (3,0) {$4$};

          \draw[->,blue!70!white]   (1) -- (2) ;
          \draw[->,black]  (2) -- (3) ;
          \draw[->,blue!70!white]   (3) -- (4) ;
    \end{tikzpicture}
\end{center}
\end{minipage}%
\begin{minipage}{.3\linewidth}
\begin{center}
\begin{tikzpicture}[thick,scale=1]

 	 \node[circle, draw, fill=red!50, inner sep=1pt, minimum width=1pt] (1) at (0,0)  {$1$};
	   \node[circle, draw, fill=green!50, inner sep=1pt, minimum width=1pt] (2) at (1,0) {$2$};
          \node[circle, draw, fill=red!50, inner sep=1pt, minimum width=1pt] (3) at (2,0) {$3$};
          \node[circle, draw, fill=yellow!50, inner sep=1pt, minimum width=1pt] (4) at (3,0) {$4$};

          \draw[->,blue!70!white]   (1) -- (2) ;
          \draw[->,orange]  (2) -- (3) ;
          \draw[->,blue!70!white]   (3) -- (4) ;
    \end{tikzpicture}
\end{center}
\end{minipage}
\end{minipage}

By the \Cref{conv: coloring}, the left DAG is uncolored, i.e., it has seven color classes which are all singletons. The corresponding colored Gaussian DAG model has seven independent parameters $\omega_1, \omega_2, \omega_3, \omega_4, \lambda_{12}, \lambda_{23}, \lambda_{34}$.
The colored DAG on the right imposes the additional constraints that $\omega_1 = \omega_3$ and $\lambda_{12} = \lambda_{34}$. By \Cref{conv: coloring}, it defines the same model as the colored DAG in the middle since nodes $2$ and $4$ and edge $2 \to 3$ on the right have singleton color classes.
\end{example}

In this section, we establish a local and a global Markov property for colored Gaussian DAG models analogous to the classic (uncolored) case in~\Cref{subsec: SEMs}.
To do so, we use a characterization of the rational functions of the covariance parameters $\sigma_{ij}$ that may be used to identify the error variances $\omega_i$ and structural coefficients $\lambda_{ij}$.

\subsection{Parameter identification}
\label{subsec: parameter idenfitication}

Let $G = (V,E)$ be a DAG and let $\Sigma \in \CC M(G)$. Then there exist $(\Omega, \Lambda)\in \PosReal^{V} \times \Real^E$ such that $\phi_G(\Omega, \Lambda) = \Sigma$. Since $G$ is acyclic, the parametrization map is injective or, in other words, the parameters $\Omega$ and $\Lambda$ are \emph{globally identifiable} from the covariance matrix $\Sigma$ \cite[Theorem~16.2.1]{Sullivant}, \cite[Section~3]{NestedDeterminants}. Moreover, each parameter can be expressed as a rational function in~$\Sigma$ whose denominator is a principal minor and therefore nonzero on~$\PD^V$. The next lemma shows the explicit formulas for parameter recovery.

\begin{lemma} \label{lemma:Ident}
Let $G$ be a DAG and $\phi_G$ its trek rule parametrization. For $\Sigma = \phi_G(\Omega, \Lambda) \in \CC M(G)$, the $\Omega$ and $\Lambda$ parameters are recovered via
\begin{align}
\label{eq:omega} \tag{$\omega$} \omega_i &= \RM{Var}(X_i \mid X_{\pa(i)}) = \frac{|\Sigma_{\pa(i) \cup i}|}{|\Sigma_{\pa(i)}|}, \\
\label{eq:lambda} \tag{$\lambda$} \lambda_{ij} &= \frac{\RM{Cov}(X_i, X_j \mid X_{\pa(j) \setminus i})}{\RM{Var}(X_i \mid X_{\pa(j) \setminus i})} = \frac{|\Sigma_{ij|\pa(j) \setminus i}|}{|\Sigma_{\pa(j)}|}.
\end{align}
\end{lemma}

These formulas are well-known: see \cite[Theorem~3.1]{wu2023partial} for \eqref{eq:omega} and \cite[Section~4.4]{PathDiagrams} for \eqref{eq:lambda} in the more general context of path diagrams.
Since a homogeneity constraint amounts to setting $\omega_i = \omega_j$ or $\lambda_{ij} = \lambda_{k\ell}$, the above rational functions provide constraints that capture when two nodes, or respectively two edges, have the same color; for instance, the function
\[
f(\Sigma) = \frac{|\Sigma_{\pa(i) \cup i}|}{|\Sigma_{\pa(i)}|} - \frac{|\Sigma_{\pa(j) \cup j}|}{|\Sigma_{\pa(j)}|}
\]
will evaluate to $0$ on all $\Sigma \in \CC M(G,c)$ where $c(i) = c(j)$ since this equality of colors corresponds to the equality of parameters $\omega_i = \omega_j$. To formulate a global Markov property for colored DAG models, we would like a description of all such constraints on the model that arise from the chosen coloring.
To do so, we utilize a characterization of the ways in which we may identify the parameter values $\omega_i$ and $\lambda_{ij}$ using rational functions of the same form as \eqref{eq:omega} and \eqref{eq:lambda}.

\begin{definition} \label{def:RatFun}
We define two families of rational functions of covariance matrices $\Sigma$:
\begin{align*}
  \omega_{i|A}(\Sigma) \defas \frac{|\Sigma_{iA}|}{|\Sigma_{A}|} \quad\text{and}\quad
  \lambda_{ij|A}(\Sigma) \defas \frac{|\Sigma_{ij|A\setminus i}|}{|\Sigma_A|}.
\end{align*}
Fix a DAG $G = (V,E)$. A set $A \subseteq V$ is \emph{identifying for the vertex $i \in V$} if $\omega_{i|A}(\phi_G(\Omega, \Lambda)) = \omega_i$ for all $(\Omega, \Lambda)\in \PosReal^V \times \Real^E$. It is \emph{identifying for the edge $ij \in E$} if $\lambda_{ij|A}(\phi_G(\Omega, \Lambda)) = \lambda_{ij}$  for all $(\Omega, \Lambda)\in \PosReal^V \times \Real^E$. Let $\CC A_G(i)$ and $\CC A_G(ij)$ denote the sets of $i$- respectively $ij$-identifying sets.
\end{definition}

The vertex-identifying sets are characterized in recent work of Drton and Wu \cite[Theorem~3.3]{wu2023partial}. %
Characterizations of edge-identifying sets are introduced in~\cite{PathDiagrams} and~\cite{pearl1998graphs}.
We give an independent proof of the characterization of edge-identifying sets in Supplementary material \Cref{section:IdentifyingSets}.
Together, these results yield the following theorem.
Recall that $\overline{\de}_G(i)=\{i\}\cup \de_G(i)$.

\begin{theorem}
\label{thm:IdentifyingSets}
Let $G = (V,E)$ be a DAG. Then:
\begin{enumerate}
\item\label{thm:IdentifyingSets:Vertices}%
$\omega_i = \omega_{i|A}(\Sigma)$ for every $\Sigma \in \CC M(G)$ if and only if $\pa(i) \subseteq A \subseteq V \setminus \ol{\de}(i)$.
\item\label{thm:IdentifyingSets:NonEdges}%
If $ij \not\in E$, then $\lambda_{ij} = 0 = \lambda_{ij|A}(\Sigma)$ for every $\Sigma \in \CC M(G)$ if and only if $A \setminus i$ d-separates $i$ and $j$ in~$G$.
\item\label{thm:IdentifyingSets:Edges}%
If $ij \in E$, then $\lambda_{ij} = \lambda_{ij|A}(\Sigma)$ for every $\Sigma \in \CC M(G)$ if and only if $i \in A \subseteq V \setminus \ol{\de}(j)$ and $A \setminus i$~d-separates $i$ and $j$ in the graph $G_{ij}$ which arises from $G$ by deleting the edge $ij$ and the vertices $\de(j)$.
\end{enumerate}
\end{theorem}

The condition~\eqref{thm:IdentifyingSets:Edges} for edge identification resembles the \emph{backdoor criterion} of Pearl \cite{Causality} but it ensures recovery of the \emph{direct (causal) effect}, as opposed to the \emph{total effect}.
In~\cite[Theorem~5.3.1]{Causality}, the edge-identifying sets are termed the \emph{single-door criterion for direct effects}.

\subsection{Markov properties}
\label{subsec: MPs}
\Cref{thm:IdentifyingSets} characterizes the subsets of nodes that may be used to identify a given model parameter in a DAG model $\CC M(G)$.
We define the \emph{vertex-coloring constraint} for the quadruple $(i,j, A, B)$ to be
\[
\textrm{vcc}(i,j ; A, B) = \omega_{i | A}(\Sigma) - \omega_{j | B}(\Sigma).
\]
Similarly, we define an \emph{edge-coloring constraint} for the quadruple $(ij, k\ell, A, B)$ to be
\[
\textrm{ecc}(ij, k\ell; A, B) = \lambda_{ij |A}(\Sigma) - \lambda_{k\ell | B}(\Sigma).
\]
We say that $\Sigma$ satisfies the coloring constraint $\textrm{vcc}(i,j ; A, B)$ (resp. $\textrm{ecc}(ij, k\ell; A, B)$) whenever the given function evaluates to zero on $\Sigma$.
Recall from \Cref{lem: CI polynomials} that conditional independence for multivariate Gaussian distributions is characterized by the vanishing of certain functions of its covariance matrix (namely, the partial correlations). The functions $\textrm{vcc}$ and $\textrm{ecc}$ play the analogous role for coloring constraints: their vanishing on a covariance matrix detects the coloring constraints with which it is compatible.

Just as a Markov property for $\CC M(G)$ corresponds a collection of rational functions capturing conditional independence relations, we may define a Markov property for $\CC M(G,c)$ as a collection of rational functions capturing conditional independence and coloring relations.
\begin{definition}
    \label{def: local colored MP}
    We say that $\Sigma$ satisfies the \emph{local Markov property} with respect to the colored DAG $(G,c)$ if
    \begin{enumerate}[noitemsep, itemsep=0.3em]
        \item\label{def: local colored MP:1} $\Sigma$ satisfies the local Markov property with respect to $G$,
        \item\label{def: local colored MP:2} $\Sigma$ satisfies $\textrm{vcc}(i,j ; \pa_G(i), \pa_G(j))$ if $c(i) = c(j)$, and
        \item\label{def: local colored MP:3} $\Sigma$ satisfies $\textrm{ecc}(ij, k\ell; \pa_G(j), \pa_G(\ell))$ if $c(ij) = c(k\ell)$.
    \end{enumerate}
\end{definition}

Condition~\eqref{def: local colored MP:1} in \Cref{def: local colored MP} states that $\CC M(G,c)$ is a submodel of $\CC M(G)$, while conditions~\eqref{def: local colored MP:2} and~\eqref{def: local colored MP:3} encode the assumptions that we have equal error variances whenever $c(i) = c(j)$ and equal structural coefficients whenever $c(ij) = c(k\ell)$.
Note also that condition~\eqref{def: local colored MP:2} is equivalent to $\Var(X_i \mid X_{\pa_G(i)}) = \Var(X_j \mid X_{\pa_G(j)})$.
Hence, the local Markov property for a colored DAG is precisely what we would expect.
\Cref{def: global colored MP} provides the corresponding global Markov property.
\begin{definition}
    \label{def: global colored MP}
    We say that $\Sigma$ satisfies the \emph{global Markov property} with respect to the colored DAG $(G,c)$ if
    \begin{enumerate}[noitemsep, itemsep=0.3em]
        \item\label{def: global colored MP:ci} $\Sigma$ satisfies the global Markov property with respect to $G$,
        \item\label{def: global colored MP:vcc} $\Sigma$ satisfies $\textrm{vcc}(i,j ; A, B)$ for all $(A,B) \in \mathcal{A}_G(i)\times \mathcal{A}_G(j)$ if $c(i) = c(j)$, and
        \item\label{def: global colored MP:ecc} $\Sigma$ satisfies $\textrm{ecc}(ij, k\ell; A, B)$ for all $(A,B) \in \mathcal{A}_G(ij)\times \mathcal{A}_G(k\ell)$ if $c(ij) = c(k\ell)$.
    \end{enumerate}
\end{definition}

\begin{remark}
    \label{rmk: global MP}
    Note that a distribution $(X_1,\ldots, X_d)^T \sim \textrm{N}(0,\Sigma)$ satisfies the coloring constraint $\textrm{vcc}(i,j ; A, B)$ if and only if $\Var(X_i \mid X_A ) = \Var(X_j \mid X_B)$.
    In other words, the model $\CC M(G,c)$ is defined as the submodel of $\CC M(G)$ specified by a collection of invariance constraints among conditional distributions corresponding to the coloring.
    The local Markov property considers the local invariance constraints corresponding to families (i.e., $i \cup \pa_G(i)$).
    The global Markov property then includes all additional invariance constraints that are implied by these local constraints.
\end{remark}

\begin{remark}
    \label{rmk: invariances}
    This definition of the global Markov property based on the addition of invariance constraints is in direct analogy to the definition of the $\mathcal{I}$-Markov property for general interventional DAG models \cite[Definition 3.6]{yang2018characterizing}.
    In the interventional context, the relevant additional invariance constraints satisfied by the model are placed on conditional distributions in different experimental settings.
    In this context, the relevant invariance constraints are on conditional distributions derived from a single (observational) distribution and fix the variables of interest ($X_i$ and $X_j$ or $X_{\{i,j\}}$ and $X_{\{k,\ell\}}$) but vary the conditioning sets.
\end{remark}

It follows from global rational identifiability and the definition of the coloring constraints that $\Sigma\in \CC M(G,c)$ if and only if $\Sigma$ satisfies the local Markov property with respect to $(G,c)$.
It can further be shown that these conditions are also equivalent to $\Sigma$ satisfying the global Markov property with respect to $(G,c)$.
\begin{theorem}
    \label{thm: MP}
    Let $(G,c)$ be a colored DAG and $\Sigma \in \PD^V$. The following are equivalent:
    \begin{enumerate}[noitemsep, itemsep=0.3em]
        \item\label{thm: MP:1} $\Sigma \in \CC M(G,c)$,
        \item\label{thm: MP:2} $\Sigma$ satisfies the local Markov property with respect to $(G,c)$, and
        \item\label{thm: MP:3} $\Sigma$ satisfies the global Markov property with respect to $(G,c)$.
    \end{enumerate}
\end{theorem}

In \Cref{sec: model equivalence}, we will make use of these Markov properties to derive structural identifiability results which are then applied to yield causal discovery methods in \Cref{sec:causaldiscovery}.

\begin{example}\label{example:colored path}
Let $(G,c)$ be the colored DAG as show below.
\begin{center}
\begin{tikzpicture}[thick,scale=1]

 	 \node[circle, draw, fill=red!50, inner sep=1pt, minimum width=1pt] (1) at (0,0)  {$1$};
	   \node[circle, draw, fill=white, inner sep=1pt, minimum width=1pt] (2) at (1,0) {$2$};
          \node[circle, draw, fill=red!50, inner sep=1pt, minimum width=1pt] (3) at (2,0) {$3$};
          \node[circle, draw, fill=white, inner sep=1pt, minimum width=1pt] (4) at (3,0) {$4$};

          \draw[->,blue!70!white] (1) -- (2) ;
          \draw[->,black]         (2) -- (3) ;
          \draw[->,blue!70!white] (3) -- (4) ;
    \end{tikzpicture}
\end{center}
The corresponding structural equations are
\begin{eqnarray*}
X_1=\epsilon_1, \hspace{.5cm}   X_2= \lambda_{12}X_1 + \epsilon_2, \hspace{.5cm} X_3= \lambda_{23}X_2 + \epsilon_3, \hspace{.5cm} X_4= \lambda_{12}X_3 + \epsilon_4.
\end{eqnarray*}
Recall that by \Cref{conv: coloring}, the vertices $2$ and $4$ and the edge $2\rightarrow 3$ lie in singleton color classes. This gives us the colored parametrization $\phi_{G,c}$ as
\begin{eqnarray*}
(\Omega, \Lambda) \mapsto \begin{pmatrix}
1 & -\lambda_{12} & 0 & 0 \\
0 & 1 & -\lambda_{23} & 0 \\
0 & 0 & 1 & -\lambda_{12} \\
0 & 0 & 0 & 1
\end{pmatrix}^{-T}
\begin{pmatrix}
\omega_1 & 0 & 0 & 0 \\
0 & \omega_2 & 0 & 0 \\
0 & 0 & \omega_1 & 0 \\
0 & 0 & 0 &  \omega_4
\end{pmatrix}
\begin{pmatrix}
1 & -\lambda_{12} & 0 & 0 \\
0 & 1 & -\lambda_{23} & 0 \\
0 & 0 & 1 & -\lambda_{12} \\
0 & 0 & 0 & 1
\end{pmatrix}^{-1}.
\end{eqnarray*}
Note that $\omega_1$ and $\lambda_{12}$ are the base parameters for the respective color classes and we set $\omega_3=\omega_1$ and $\lambda_{34}=\lambda_{12}$. The local Markov properties with respect to $G$ are $\CI{X_3,X_1|X_2}$ and $\CI{X_4,{X_1,X_2}| X_3}$. Now, as $c(1)=c(3)$, the vertex coloring constraint $\textrm{vcc}(1,3;\emptyset, 2)$ is given by
\begin{eqnarray}\label{eqnarray:vcc(1,3)}
    \textrm{vcc}(1,3;\emptyset,2) = \frac{|\Sigma_{1}|}{|\Sigma_{\emptyset}|} -\frac{|\Sigma_{2,3}|}{|\Sigma_{2}|}. %
\end{eqnarray}
Similarly, the edge coloring constraint $\textrm{ecc}(12,34;1,3)$ for the equality $c(12)=c(34)$ is given by
\begin{eqnarray}\label{eqnarray:ecc(12,34)}
    \textrm{ecc}(12,34;1,3)= \frac{|\Sigma_{12|\emptyset}|}{|\Sigma_{1}|} - \frac{|\Sigma_{34|\emptyset}|}{|\Sigma_{3}|}. %
\end{eqnarray}
Thus, $\Sigma$ satisfies the local Markov property with respect to $(G,c)$ if $\Sigma$ satisfies the local Markov properties with respect to $G$ along with equations \ref{eqnarray:vcc(1,3)} and \ref{eqnarray:ecc(12,34)}. We can similarly obtain the vertex and edge identifying sets by Theorem \ref{thm:IdentifyingSets} to obtain the global Markov property of $\Sigma$ with respect to $(G,c)$.
\end{example}

\section{Model geometry}
\label{sec:ModelGeometry}

The map identified in \Cref{lemma:Ident}, which plays a fundamental role in recovering the global Markov property in \Cref{def: global colored MP}, may also be used to deduce geometric properties of the model and determine its dimension.
In~\Cref{subsec:smooth}, we use this map to show that all colored DAG models are smooth submanifolds of the positive definite cone.
This observation implies that standard large-sample asymptotic theory may be applied when performing statistical inference with colored Gaussian DAG models, while such techniques may fail for models containing singularities \cite{DrtonSmooth, DrtonXiao}.
For instance, likelihood ratio tests with null and alternative contained in a colored DAG model will have test variables that are asymptotically $\chi^2$-distributed.

In \Cref{subsec:conj}, we use the map in \Cref{lemma:Ident} to give a concise proof of a conjecture from applied algebra regarding the geometry of (uncolored) Gaussian DAG models.
We prove this conjecture more generally for the family of colored Gaussian DAG models, and we observe that our techniques also apply to prove analogous results for other well-studied graphical models, including undirected Gaussian graphical models \cite{lauritzen1996graphical}, RCON models \cite{UndirectedColored} and Gaussian ancestral graph models \cite{richardson2002ancestral}.
The former observation provides us with a computational method for deducing when a colored DAG model does not admit faithful distributions (see \Cref{subsec: faithfulness}).
The main tool in our proof (Lemma~\ref{lemma:Saturation}) also provides a general tool for deriving Markov properties for rationally parameterized statistical models with globally, rationally identifiable parameters.
Hence, one may, more generally, apply the results of this section in the identification of constraints that may be used to test data against such statistical models (see \Cref{rem:MarkovProperty}).

\subsection{Smoothness, dimension and topology}\label{subsec:smooth}

Global identifiability shows that the parametrization $\phi_{G,c}: \PosReal^{\vc} \times \BB R^{\ec} \to \CC M(G,c)$ is a homeomorphism: it is bijective and in both directions it is given by well-defined rational functions, which are continuous. We further note that $\CC M(G,c)$ is a smooth submanifold of $\PD^V$ and that the parametrization and its inverse establish a diffeomorphism, i.e., an isomorphism of differentiable manifolds.

\begin{theorem} \label{thm:Smoothness}
The parametrization $\phi_{G,c}: \PosReal^{\vc} \times \BB R^{\ec} \to \CC M(G,c) \subseteq \PD^V$ of a colored Gaussian DAG model is a diffeomorphism from its domain onto its image equipped with the induced smooth structure from~$\PD^V$. Hence, $\CC M(G,c)$ is a smooth submanifold of~$\PD^V$ diffeomorphic to an open ball of dimension $\vc + \ec$.
\end{theorem}

\subsection{Algebraic description}\label{subsec:conj}
In this subsection we give a proof of a conjecture of Sullivant \cite{Sullivant} on Gaussian DAG models by proving the result more generally for colored Gaussian DAG models.
The proof follows from a lemma (Lemma~\ref{lemma:Saturation}) that allows us to additionally prove the same result for colored (and uncolored) undirected Gaussian graphical models as well as ancestral Gaussian graphical models.

We give natural generators of an ideal $I_{G,c} \subseteq \BB R[\Sigma]$ whose saturation at certain principal minors gives the vanishing ideal $P_{G,c}$ of the colored Gaussian DAG model $\CC M(G,c)$. For algebraists, the ideal $I_{G,c}$ is generally not prime but the proof shows that $P_{G,c}$ is the unique prime above $I_{G,c}$ which intersects the cone of positive definite matrices.

To this end, let us first reexamine the results of \Cref{lemma:Ident} from an algebraic angle. The trek rule parametrization $\phi_G$ for any DAG is a polynomial map. By global rational identifiability, its inverse $\psi_G$ is a rational map whose denominators are principal minors with respect to parent~sets.

\begin{definition}
Let $G$ be a DAG. For any vertex $i$, the polynomial $|\Sigma_{\pa(i)}| \in \BB R[\Sigma]$ is a \emph{parental principal minor}. Denote by $S_G$ the multiplicatively closed subset generated by all parental principal minors in~$G$, i.e., $S_G = \Set{ \prod_{i \in V} |\Sigma_{\pa(i)}|^{k_i} : k_i \in \BB N }$.
\end{definition}

\begin{definition}
    \label{def:vanishingideal}
    The kernel of the map $\phi^\ast_{G,c}: \BB R[\Sigma]\to\BB R[\Lambda,\Omega]$ is called the \emph{vanishing ideal} of the model $\CC M(G,c)$, and it is denoted $P_{G,c}$.
    When the coloring $c$ is an injection (i.e., $\CC M(G,c) = \CC M(G)$), we denote this ideal by $P_G$.
\end{definition}

The map $\phi^\ast_{G,c}$ in \Cref{def:vanishingideal} is the pullback of the parametrization map $\phi_{G,c}$ of the colored DAG model $\CC M(G)$ as defined in \Cref{subsec: algebra preliminaries}.
It follows from standard results in algebra that the vanishing ideal of $\CC M(G)$ defined in \Cref{subsec: algebra preliminaries} is the same as the ideal $P_{G,c}$ in \Cref{def:vanishingideal}.
Namely, the ideal $P_{G,c}$ is the set of all polynomials in variables $\sigma_{ij}$ with real coefficients that evaluate to zero on every $\Sigma\in \CC M(G,c)$.

\begin{definition} \label{def:Relations}
Let $(G, c)$ be a colored DAG. We define the \emph{vertex coloring relations}, \emph{edge coloring relations} and the \emph{conditional independence relations} as the following polynomials in~$\BB R[\Sigma]$:
\begin{align*}
\vcr_c(i,j) &\defas |\Sigma_{\pa(i)}| \cdot |\Sigma_{\pa(j)}| \cdot (\omega_{i|\pa(i)}(\Sigma) - \omega_{j|\pa(j)}(\Sigma)) \text{ for $c(i) = c(j)$}, \\
\ecr_c(ij,kl) &\defas |\Sigma_{\pa(j)}| \cdot |\Sigma_{\pa(l)}| \cdot (\lambda_{ij|\pa(j)}(\Sigma) - \lambda_{kl|\pa(l)}(\Sigma)) \text{ for $c(ij) = c(kl)$}, \\
\cir_G(i,j) &\defas |\Sigma_{\pa(j)}| \cdot \lambda_{ij|\pa(j)}(\Sigma) \text{ for $ij \not\in E$}.
\end{align*}
Let $I_G$ denote the ideal generated by all $\cir_G$ relations, $I_c$ the ideal generated by all $\vcr_c$ and $\ecr_c$ relations.  We call the ideal $I_{G,c} = I_G + I_c$ the \emph{local colored conditional independence ideal} of~$(G,c)$.
\end{definition}

\begin{remark}
The vanishing of $\lambda_{ij|\pa(j)}(\Sigma)$ for $ij \not\in E$ encodes the conditional independence statement $\CI{i,j|\pa(j)}$. Hence, $I_G$ is the ideal of the \emph{directed (pairwise) local Markov property} of~$G$.
\end{remark}

Recall the global conditional independence ideal $J_G$ from \Cref{lem: CI polynomials}. By definition, the polynomials generating $I_G$ are contained in $J_G$ and therefore $I_G\subseteq J_G$.
The following is a conjecture of Sullivant.
\begin{conjecture}[{\cite{AlgGeoOfGBN}}]
    \label{conj:sullivant}
    Let $S_V = \Set{ \prod_{A \subseteq V} |\Sigma_A|^{k_A}: k_A\in \mathbb{N}}$.  Then $P_G = J_G : S_V$.
\end{conjecture}

Since $I_G \subseteq J_G \subseteq P_G$, showing that $P_G = I_G : S_V$ is sufficient to prove \Cref{conj:sullivant}.
Moreover, it is further sufficient to show that $I_G : S^\prime = P_G$ for some $S^\prime\subseteq S_V$.

In the following, we give a short proof of \Cref{conj:sullivant} based on the observations made in \Cref{sec:markovproperties}.
We obtain our proof by proving a more general result for colored Gaussian DAG models.
Hence, before the proof, we note that one can analogously define a \emph{global colored conditional independence ideal} for $(G,c)$ as $J_{G,c} = J_G + J_c$, where $J_c$ is the ideal generated by the $\vcr$ and $\ecr$ relations
\begin{equation*}
    \begin{split}
        \vcr_c(i, j : A, B) &= |\Sigma_A|\, |\Sigma_B|\, (\omega_{i | A}(\Sigma) - \omega_{j | B}(\Sigma)),\\
        \ecr_c(ij,kl; A, B) &= |\Sigma_A|\, |\Sigma_B|\, (\lambda_{ij |A}(\Sigma) -\lambda_{k\ell |B}(\Sigma)),
    \end{split}
\end{equation*}
for all pairs $(A,B)\in \mathcal{A}(ij)\times \mathcal{A}(k\ell)$ when $c(ij) = c(k\ell)$ and $(A,B)\in \mathcal{A}(i)\times \mathcal{A}(j)$ when $c(i) = c(j)$.
Hence, $J_{G,c}$ is the ideal associated to the global colored Markov property, and we have that $I_{G,c}\subseteq J_{G,c}\subseteq P_{G,c}$.
Given this set-up, the proof of \Cref{conj:sullivant} is quick, depending only on \Cref{lemma:Saturation} below.
We also note %
that this technique can be used to prove the analogous result for other families of models for which this question has been studied:
\begin{paraenum}
\item
An \emph{undirected colored Gaussian graphical model} consists of all covariance matrices $\Sigma$ such that entries in the concentration matrix $K = \Sigma^{-1}$ are zero or equal, as specified by an undirected graph $U = (V, E)$ with coloring $c: V\sqcup E \to [k]$; see~\cite{UndirectedColored}.
The vanishing ideal $P_{U,c}$ for the model is the kernel of the pullback of the parameterizing map
\begin{equation*}
    \begin{split}
        \phi_{(U,c)}:{}& \Theta \longrightarrow \Real^{V\times V};\\
         & K \longmapsto K^{-1}
    \end{split}
\end{equation*}
where $\Theta$ is the space of invertible symmetric matrices $K = (\kappa_{ij})_{i,j \in V}$ subject to the linear constraints $\kappa_{ij} = 0$ whenever $ij \not\in E$, $\kappa_{ii} = \kappa_{jj}$ whenever $c(i) = c(j)$ and $\kappa_{ij} = \kappa_{kl}$ whenever $c(ij) = c(kl)$.
The colored conditional independence ideal $I_{U,c}$ for this model is the ideal
\[
I_{U,c} = \ideal{|\Sigma_{ij|V \setminus ij}| : ij \not\in E} + \ideal{|\Sigma_{V \setminus i}| - |\Sigma_{V \setminus j}| : c(i) = c(j)} + \ideal{|\Sigma_{ij|V \setminus ij}| - |\Sigma_{kl|V \setminus kl}| : c(ij) = c(kl)}.
\]

\item
A \emph{mixed graph} is an ordered triple $A = (V, D, B)$, where $V$ is the node set of the graph, $D$ is the set of directed edges $i\rightarrow j$ in the graph and $B$ is the set of bidirected edges $i\leftrightarrow j$.
The mixed graph $A$ is called \emph{directed ancestral} if it does not contain any directed cycles or bidirected edges $i \leftrightarrow j$ for which $i \in \de_A(j)$.
Consider the linear structural equation model $X = (X_i)_{i \in V}$
\begin{equation*}
    \begin{split}
        X_i = \sum_{k\in\pa_A(i)}\lambda_{ki}X_k + \eps_i
    \end{split}
\end{equation*}
with normally distributed errors $\eps_i\sim \textrm{N}(0,\omega_i)$ where $\eps_i$ and $\eps_j$ are correlated if and only if $i\leftrightarrow j\in B$.
In this case we let $\omega_{ij}$ denote the covariance of $\eps_i$ and $\eps_j$.
Letting $\Omega$ denote the covariance matrix of the noise vector $\eps = (\eps_i)_{i \in V}$, we obtain that $X$ has covariance matrix
\[
\Sigma = \phi_A(\Omega, \Lambda) = (\BBm1_V - \Lambda)^{-T} \, \Omega \, (\BBm1_V - \Lambda)^{-1}.
\]
The \emph{Gaussian directed ancestral graph model} for $A$ is then
\[
\CC M(A) = \{\phi_A(\Omega, \Lambda)\in \PD^V : \Lambda\in\Real^{D}, \Omega\in \PD^B\},
\]
where $\PD^B$ is the set of $V\times V$ positive definite matrices with zero pattern specified by $B$.
As~noted in \cite[Section~7]{drton2018algebraic}, the parameters $(\Omega, \Lambda)$ for $\Sigma\in \CC M(A)$ are globally identifiable and given by an extension of the map in \Cref{lemma:Ident} to include bidirected edge parameters; namely,
\begin{align*}
\lambda_{ij} &= \lambda_{ij | \pa_A(j)}(\Sigma) \mbox{ for $i\rightarrow j\in D$}, \\
\omega_{ij} &= \lambda_{ij | \pa_A(i)\cup \pa_A(j)}(\Sigma) \mbox{ for $i\leftrightarrow j\in B$} \; \mbox{ and } \\
\omega_{ii} &= \omega_{i | \pa_A(i)}(\Sigma) \mbox{ for $i \in V$}.
\end{align*}
These formulas yield the following conditional independence ideal:
\begin{align*}
I_A = & \ideal{\sigma_{ij} : \an_A(i) \cup \an_A(j) = \emptyset} + {} \\
&\ideal{|\Sigma_{ij|\an_A(j)\setminus i}| : i \in \an_A(j) \setminus \pa_A(j)} + {} \\
&\ideal{|\Sigma_{ij|\an_A(i)\cup\an_A(j)}| : i < j, \; i \notin \an_A(j) \neq \emptyset}.
\end{align*}
The details of this derivation are provided in \Cref{app: ancestral}. The vanishing ideal of the model $\CC M(A)$ is denoted by $P_A$.
\end{paraenum}

Since Gaussian colored DAG models, colored undirected graphical models and directed ancestral graph models are all rationally identifiable (as described above), we obtain the following general answer to the conjecture of Sullivant:

\begin{theorem}
    \label{thm: all the graphs}
    Let $S_V = \{ \prod_{A \subseteq V} |\Sigma_A|^{k_A} : k_A \in \BB N \}\subseteq\BB R[\Sigma]$ and $S_V^\prime = \{ |\Sigma|^{k} : k \in \BB N \}\subseteq\BB R[\Sigma]$.
    \begin{enumerate}
        \item $P_{G,c} = I_{G,c} : S_V$ for any colored DAG $(G,c)$. \label{cor:coloredDAGideals}
        \item $P_{U,c} = I_{U,c} : S_V^\prime$ for any colored undirected graph $(U,c)$. \label{cor:RCONideals}
        \item $P_A = I_A : S_V$ for any directed ancestral graph $A$. \label{cor:ancestralideals}
    \end{enumerate}
\end{theorem}

The main idea behind the proof of \Cref{thm: all the graphs} is that a Gaussian graphical model is described via linear relations on its parameters ($\Omega$ and $\Lambda$ in the DAG setting).
By rational identifiability, these linear relations correspond to rational function equations in $\Sigma$, as shown in \Cref{subsec: MPs}. The ideal $I_{G,c}$ is generated by all the numerators of these equations and $S_V$ is generated by the denominators. Hence, saturating $I_{G,c}$ at $S_V$ should produce an ideal in $\BB R[\Sigma]$ which contains all the polynomial equations entailed by the linear relations among the parameters under the parametrization map $\phi_{G,c}$ --- and this is the vanishing ideal~$P_{G,c}$. The following \namecref{lemma:Saturation} makes this idea precise:

\begin{lemma} \label{lemma:Saturation}
Let $R, R'$ be commutative rings, $S \subseteq R$ a multiplicatively closed set and $\phi^*: R \to R'$ and $\psi^*: R' \to S^{-1} R$ ring homomorphisms with $\psi^* \circ \phi^* = \id_R$. Let $I' \subseteq R'$ be a prime ideal with generators $f'_1, \dots, f'_k$. Denote $\psi^*(f_i') = \sfrac{g_i}{h_i}$ with $h_i \in S$, $I = {\phi^*}^{-1}(I')$ and $J = \ideal{g_1, \dots, g_k}$. If~$J \subseteq I$ and $I \cap S = \emptyset$, then $I = J : S$.
\end{lemma}

\begin{remark} \label{rem:MarkovProperty}
Note that \Cref{lemma:Saturation} may be applied in the same way to any rationally parametrized statistical model satisfying global rational identifiability.
This is of statistical interest since it provides a natural method for extracting a Markov property (e.g.~a set of model-characterizing constraints) for any rationally identifiable statistical model with rationally identifiable parameters directly from its parametrization.
Specifically, the Markov property corresponds to the polynomial constraints produced by applying the (rational) parameter identification map to the polynomial constraints defining the model parameter space and then clearing denominators.
When the denominators correspond to boundary conditions on the ambient model space (for instance, the boundary of the PD cone) the necessary conditions on the set $S$ in \Cref{lemma:Saturation} will naturally be satisfied and the result of the lemma as applied to the model indicates that these polynomial constraints constitute a Markov property for the model.
The~polynomials are then testable constraints defining the parametrized model~\cite{sturma2022testing}.
\Cref{thm: all the graphs} demonstrates this phenomenon for well-studied families of Gaussian models where these constraints are already known (as they are given by the Markov properties associated to the graphs).
\end{remark}

\begin{remark} \label{rem:ApplyVanishingIdeal}
Having a parametrization for a model as well as an implicit description via the vanishing ideal (up to saturation) as in \Cref{thm: all the graphs} allows to test model equivalence effectively. Suppose $(G,c)$ and $(H,c')$ are given colored DAGs. They induce the same model if and only if their vanishing ideals coincide. The inclusion $P_{G,c} \subseteq P_{H,c'}$ can be tested by plugging the generators of $I_{G,c}$ into the parametrization~$\phi_{H,c'}$. Analogously plugging generators of $I_{H,c'}$ into $\phi_{G,c}$ gives a complete model equivalence test which consists only of evaluating polynomial expressions and is therefore quite cheap.
\end{remark}

\Cref{thm: all the graphs}~\eqref{cor:coloredDAGideals} and~\eqref{cor:RCONideals} prove \Cref{conj:sullivant} for uncolored Gaussian DAG models and uncolored undirected Gaussian graphical models by simply considering an injective coloring.

While \Cref{thm: all the graphs}~\eqref{cor:coloredDAGideals} is sufficient to prove the original conjecture of Sullivant, \Cref{thm:Saturation} below gives a second proof, by observing a stronger property; namely, that it is sufficient to saturate only at the multiplicatively closed set generated by the principal minors indexed by the parent sets of nodes in~$G$.
Proving this stronger result requires some additional steps, which are carried out in \Cref{section:UncoloredSaturation}.

\begin{theorem} \label{thm:Saturation}
Let $(G, c)$ be a colored DAG and let $S_G = \{\prod_{i\in V} |\Sigma_{\pa_G(i)}|^{k_i} : k_i \in \BB N\}$. The vanishing ideal $P_{G,c}$ of the colored Gaussian DAG model $\CC M(G,c)$ is the saturation $I_{G,c} : S_G$.
\end{theorem}

\begin{example}
 Let $(G,c)$ be the colored DAG as seen in Example \ref{example:colored path}. As there are three missing edges in $G$, $I_G$ is the ideal generated by the $\cir_G(i,j)$ relations
 \[
 I_G=\langle \sigma_{13}\sigma_{22}-\sigma_{12}\sigma_{23}, \sigma_{14}\sigma_{33}-\sigma_{13}\sigma_{34}, \sigma_{24}\sigma_{33}-\sigma_{23}\sigma_{34} \rangle.
 \]
Similarly, $I_c$ is the ideal generated by the two coloring relations $\vcr_c(1,3)$ and $\ecr_c(12,34)$ and is given by
\[
I_c=\langle \sigma_{11}\sigma_{22}- \sigma_{22}\sigma_{33} +\sigma_{23}^2, \sigma_{12}\sigma_{33}-\sigma_{11}\sigma_{34} \rangle.
\]
Therefore, $I_{G,c} = I_G +I_c$ is the local colored conditional independence ideal of $(G,c)$. The vanishing ideal $P_{G,c}$ of $(G,c)$ is generated by $I_{G,c}$ and 8 additional relations
\begin{eqnarray*}
P_{G,c}&=& I_{G,c} + \langle \sigma_{14}\sigma_{23}-\sigma_{13}\sigma_{24}, \, \sigma_{12}\sigma_{23}-\sigma_{11}\sigma_{24}, \, \sigma_{14}\sigma_{22}-\sigma_{12}\sigma_{24},  \\
&&\sigma_{13}\sigma_{22}-\sigma_{11}\sigma_{24}, \, \sigma_{12}\sigma_{22}+\sigma_{23}\sigma_{24}-\sigma_{22}\sigma_{34}, \, \sigma_{12}\sigma_{13}-\sigma_{11}\sigma_{14},  \\
&&\sigma_{12}^2 + \sigma_{13}\sigma_{24} -\sigma_{12}\sigma_{34}, \, \sigma_{11}\sigma_{12}+\sigma_{13}\sigma_{23}-\sigma_{11}\sigma_{34} \rangle \\
&=& I_{G,c}:S_G.
\end{eqnarray*}
The 8 additional generators appear after saturation of $I_{G,c}$ at $S_G = \{\, \sigma_{11}^k \sigma_{22}^\ell \sigma_{33}^m : k, \ell, m \in \BB N \}$.
\end{example}

\section{Faithfulness Considerations and Model equivalence}
\label{sec: model equivalence}

When investigating consistency of a causal discovery algorithm it is common practice to first check if the algorithm is consistent under the assumption that the data-generating distribution is faithful to the causal graph.
For example, methods such as Greedy Equivalence Search (GES) \cite{chickering2002optimal}, PC algorithm \cite{spirtes1991algorithm} or hybrid algorithms such as GreedySP \cite{solus2021consistency}, were first shown to be consistent under faithfulness before this assumption was relaxed to more general consistency guarantees.
Distributions within $\CC M(G)$ that are faithful to $G$ are known to exist (see for instance \cite{meek1995strong}), and hence this standard assumption for guaranteeing consistency is non-vacuous.

A second important consideration is whether or not two distinct graphs define the same DAG model.
In the case of uncolored DAGs, there are several well-known combinatorial characterizations of this condition, including the classic result due to Verma and Pearl \cite{VPequivalence}:
two DAGs define the same model if and only if they have the same skeleton and v-structures.
There is also the characterization of Andersson et al.~\cite{AMPEquivalence} who characterize model equivalence as two DAGs having the same \emph{essential graph} (or \emph{CPDAG}), and the transformational characterization of model equivalence due to Chickering \cite{chickering2013transformational}.
These characterizations are fundamental to the process of causal discovery, as they describe the structure that is possible to estimate from observational data alone (without further modeling assumptions).

In causality, learning only an equivalence class of DAGs may not achieve the desired outcome, as we typically do not want the directions of our cause-effect relations to be interchangeable.
Hence, there are many works characterizing model equivalence under additional assumptions on the distribution \cite{hoyer2008nonlinear, peters2012identifiability, peters2014identifiability, wu2023partial} as well as with the help of interventional data \cite{hauser2012characterization, yang2018characterizing}.
Under several of these conditions \cite{hoyer2008nonlinear, peters2012identifiability, peters2014identifiability} the assumptions result in equivalence classes of size $1$; in which case the graph is called \emph{structurally identifiable}.

The general colored DAG models $\CC M(G,c)$ fall into the former of the two categories discussed in the preceding paragraph; i.e., where one aims to refine the model equivalence classes of uncolored DAGs with the help of additional parametric constraints, specified in this case by the coloring.
In \Cref{subsec: faithfulness}, we first address the question of existence of distributions that are faithful to the colored DAG $(G,c)$.
In \Cref{subsec: model equivalence edge}, we derive some necessary conditions for model equivalence and obtain as a corollary some structural identifiability results in the case of edge-colored DAGs.

\subsection{Faithfulness}
\label{subsec: faithfulness}
In this subsection we %
show that distributions faithful to colored DAGs exist in some cases but not in general.
Proofs of results in this subsection may be found in the Supplementary material \Cref{section: faithfulness and MP proofs}.

\begin{definition}
Let $(G, c)$ be a colored DAG and $\Sigma \in \CC M(G,c)$. The covariance matrix $\Sigma$ is \emph{faithful to $G$} if it satisfies no more conditional independence statements than those implied by d-separation in~$G$. It is \emph{$G$-faithful to $c$} if it satisfies no more coloring constraints than those in \Cref{def: global colored MP}~\eqref{def: global colored MP:vcc} and~\eqref{def: global colored MP:ecc}; i.e., $\Sigma$ is $G$-faithful to $c$ if
\begin{itemize}
    \item for every pair of nodes $i, j$ in $G$ and sets $A \subseteq V\setminus i$, $B\subseteq V\setminus j$ we have that $\vcr_c(i, j; A, B)$ evaluates to $0$ on $\Sigma$ if and only if $c(i) = c(j)$ and $A\in\mathcal{A}_G(i)$ and $B\in \mathcal{A}_G(j)$; and
    \item for every pair of edges $ij, k\ell$ in $G$ and sets $A\subseteq V \setminus j$, $B\subseteq V\setminus \ell$ we have that $\ecr_c(ij, k\ell; A, B)$ evaluates to $0$ on $\Sigma$ if and only if $c(ij) = c(k\ell)$ and $A\in\mathcal{A}_G(ij)$ and $B\in \mathcal{A}_G(k\ell)$.
\end{itemize}
We say that $\Sigma$ is $G$-faithful to $(G,c)$ if it is \emph{faithful} to $G$ and $G$-faithful to $c$.
\end{definition}

    Note that the definition of $G$-faithful depends on the choice of DAG $G$.  We use $G$-faithful to distinguish from a more general notion of faithfulness to $c$; namely, that the only constraints of the form
    $
    \ecr_c(ij, k\ell; A, B)
    $
    and
    $
    \vcr_c(i, j; A, B)
    $
    are those implied by the assumption that $\Sigma$ is Markov to $(G,c)$.
    In principle, it is possible that $\CC M(G,c) = \CC M(H,c')$ where $(H,c')$ contains two edges of the same color that do not appear in $(G,c)$.
    While $G$-faithfulness does not capture relations possibly arising from model equivalence in which graphs change both edge structure and coloring, it does allow us to distinguish between models in which the DAGs are identical but the colorings are distinct.

\begin{proposition}
\label{prop: faithful to c}
Let $(G,c)$ be a colored DAG containing edges $ij$ and $k\ell$.
Let $A\in \mathcal{A}(ij)$ and $B\in \mathcal{A}(k\ell)$.
For a generic $\Sigma\in \CC M(G,c)$, we have that $c(ij) = c(k\ell)$ if and only $\ecr_c(ij, k\ell; A, B)$ evaluates to $0$ on $\Sigma$.
\end{proposition}

Wu and Drton recently showed that vertex colors alone do not imply additional CI~statements:

\begin{proposition}[{\cite[Proposition~3.1]{wu2023partial}}] \label{prop:vertex coloring faithful to G}
Let $(G, c)$ be a vertex-colored DAG. A generic $\Sigma \in \CC M(G,c)$ is faithful to~$G$.
\end{proposition}

The analogous statement for edge colors is also easy to prove: %

\begin{proposition}
\label{prop: edge faithful}
Let $(G, c)$ be an edge-colored DAG. A generic $\Sigma \in \CC M(G,c)$ is faithful to~$G$.
\end{proposition}

Model equivalence for colored DAGs is defined similarly to the uncolored case.

\begin{definition}
We say that two colored DAGs $(G,c)$ and $(H,d)$ are \textit{model equivalent} if $\CC M(G,c) = \CC M(H,d)$. The \textit{model equivalence class} $[G,c]$ of a colored DAG $(G,c)$ is a set that consists of all the colored DAGs that are model equivalent to $(G,c)$.
\end{definition}

\begin{corollary} \label{thm:MarkovEqv}
Let $(G,c)$ and $(H,d)$ be model equivalent colored DAGs. Suppose that $c$ is a vertex- (edge-) coloring and that $d$ is a vertex- (edge-) coloring. Then $G$ and $H$ are Markov equivalent DAGs. Hence, they have the same skeleton and v-structures.
\end{corollary}

However, if vertex and edge colors occur simultaneously, additional CI~statements not represented by d-separations may be implied on the model.
The following is an example of a colored DAG $(G, c)$ on five vertices whose model does not contain any distribution which is faithful to~$G$.

\begin{example}\label{example:unfaithful CI}
Consider the following colored DAG $(G, c)$:
\begin{center}
\begin{tikzpicture}[thick,scale=0.6]

 	 \node[circle, draw, fill=green!40, inner sep=1pt, minimum width=1pt] (1) at (0,0)  {$1$};
	   \node[circle, draw, fill=red!50, inner sep=1pt, minimum width=1pt] (2) at (0,-4) {$2$};
          \node[circle, draw, fill=red!50, inner sep=1pt, minimum width=1pt] (3) at (2,0) {$3$};
          \node[circle, draw, fill=blue!40, inner sep=1pt, minimum width=1pt] (4) at (2,-2) {$4$};
          \node[circle, draw, fill=blue!40, inner sep=1pt, minimum width=1pt] (5) at (0,-2) {$5$};

  	 \draw[->,cyan]   (1) -- (5) ;
          \draw[->,cyan]   (1) -- (3) ;
          \draw[->, orange]   (2) -- (5) ;
          \draw[->,orange]   (3) -- (4) ;
          \draw[->,orange]   (4) -- (5) ;
    \end{tikzpicture}
\end{center}
In the DAG $G$, $1$ is d-connected to $4$ given $5$ and therefore the generic matrix in the uncolored model does not satisfy $\CI{1,4|5}$. To witness this, it suffices to see that the submatrix $\Sigma_{14|5}$ is generically invertible:
\begin{align*}
  |\Sigma_{14|5}| &= \lambda_{13} \lambda_{25}^2 \lambda_{34} \omega_{1} \omega_{2} - \lambda_{15} \lambda_{34}^2 \lambda_{45} \omega_{1} \omega_{3} - \lambda_{15} \lambda_{45} \omega_{1} \omega_{4} + \lambda_{13} \lambda_{34} \omega_{1} \omega_{5}.
\end{align*}
However, the given coloring identifies parameters and, as one easily verifies, causes all terms in the determinant to cancel.
This in turn means that $\CI{1,4|5}$ holds for all points $\Sigma$ in the colored model. Hence, there does not exist a distribution in $\CC M(G,c)$ which is faithful to~$G$.
\end{example}

The existence of colored DAG models that do not admit faithful distributions complicates efforts to provide easy, general characterizations of model equivalence for colored DAGs. The procedure described in \Cref{rem:ApplyVanishingIdeal} gives an effective model equivalence test based on \Cref{thm: all the graphs}. However, we are still lacking a combinatorial criterion similar to the d-separation criterion for uncolored DAGs.
In the following subsection, we expand upon the results for vertex-colored DAGs due to \cite{peters2014identifiability, wu2023partial} by providing structural identifiability results for edge-colored DAGs, which are known to admit faithful distributions by \Cref{prop: edge faithful}. The following important question is left for future work:

\begin{question}\label{quest: faithfulness}
Which colored DAGs admit distributions that are faithful to~$G$?
\end{question}

\begin{remark}
A complete solution to Question~\ref{quest: faithfulness} amounts to characterizing which minors $|\Sigma_{A,B|C}|$ vanish only after imposing the coloring constraints.
The entries of $\Sigma_{A,B|C}$ are each sums of trek monomials, and some of these monomials will cancel in the computation of the minor $|\Sigma_{A,B|C}|$.
In \cite{Positivity}, a cancellation-free expression of $\Sigma_{A,B|C}$ via trek systems with \emph{no-sided intersection} is given, which for the minor in Example~\ref{example:unfaithful CI} is
\[
\Sigma_{14|5} =
\begin{pmatrix}
  \omega_1\lambda_{34}\lambda_{13}  & \omega_1\lambda_{15}\\
  \omega_4\lambda_{45} + \omega_3\lambda_{34}^2\lambda_{45}  & \omega_2\lambda_{25}^2 + \omega_5
\end{pmatrix}.
\]
Once $\Sigma_{14|5}$ is expressed in this simplified form we see that the coloring constraints impose the condition that the first column is equal to the second scaled by $\lambda_{25}$.
Question~\ref{quest: faithfulness} asks for a complete characterization of the coloring constraints that impose such linear dependences amongst the columns of the submatrices $\Sigma_{A,B|C}$ when expressed via treks with no-sided intersection.
\end{remark}

\subsection{Structural identifiability results for edge-colored DAGs}
\label{subsec: model equivalence edge}

In this subsection, we present structural identifiability results for edge-colored graphs.
We show first that the edge structure of a Gaussian DAG model is identifiable under the assumption of homogeneous structural coefficients.
We then prove structural identifiability for a family of edge-colored DAG models, called BPEC-DAGs, whose coloring clusters the direct causes of each variable in the system according to similar causal effects on their target node.

\subsubsection{Homogeneous structural coefficients}
\label{subsubsec: single edge color}
To prove structural identifiability under the assumption of homogeneous structural coefficients, we use the following lemma pertaining to covered edges in colored graphs.
Recall that an edge $ij$ in a DAG $G$ is called \emph{covered} if $\pa_G(j) = \pa_G(i) \cup \{i\}$.

\begin{lemma}
    \label{lem: no_rev_cov}
    Suppose that $ij$ is a covered edge in $G$, and let $G_{i\leftarrow j}$ be the DAG that differs from $G$ only by the reversal of the edge $ij$ to $ji$.
    Let $c$ be an edge-coloring of $G$ and suppose that there is an edge
    $k\ell\in c(ij)$ such that $\ell \neq i, j$.
    Define the coloring of $G_{i\leftarrow j}$ by
    \[
    c_{i\leftarrow j}(k\ell) =
    \begin{cases}
        c(k\ell) &  k\ell \neq ji\\
        c(ij) & k\ell = ji.
    \end{cases}
    \]
    Then $\CC M(G_{i\leftarrow j}, c_{i\leftarrow j}) \neq \CC M(G,c)$.
\end{lemma}

Let $\mathcal{E}_{1}$ denote the collection of all edge-colored DAGs having exactly one color class (e.g., all edges are the same color). The following theorem is an extension of the technique in \Cref{lem: no_rev_cov}.
Its proof technique (Supplementary material~\Cref{section: faithfulness and MP proofs}) is novel and related to the geometry described in Subsection~\ref{subsec:conj}.

\begin{theorem}
    \label{thm: ident single edge color}
    Let $[G,c]$ denote the model equivalence class of an edge-colored DAG $(G,c)$ with constant edge coloring where $G$ contains at least two edges.
    Then $|[G,c]\cap \mathcal{E}_{1}| = 1$; i.e., $(G,c)$ is structurally identifiable in the class of all edge-colored DAGs having exactly one color class.
\end{theorem}

\begin{remark}
    \label{rmk: oneedge}
    The only case excluded by \Cref{thm: ident single edge color} is the case when $G$ contains exactly one edge.
    In this case, the edge-coloring induces no additional relations on the model (as there is only one edge), so the equivalence class is size two by classic Markov equivalence.
\end{remark}

\Cref{thm: ident single edge color} provides a structural coefficient analogue to the result of Peters and B\"uhlmann \cite{peters2014identifiability} who showed that Gaussian DAG models are structurally identifiable under the assumption of equal error variances.
That is, they showed that each vertex-colored DAG is in a model equivalence class of size one in the class of vertex-colored DAGs with a single color.
The result of \cite{peters2014identifiability} admits applications in settings where all variables in the system arise from similar domains.
\Cref{thm: ident single edge color} analogously applies when all effects in a system are similar.

\subsubsection{Structural identifiability for BPEC-DAGs}
\label{subsubsec: BPEC}
The structural identifiability result given in \Cref{thm: ident single edge color}, would mainly be applicable in situations where one may assume the causal system is such that the causal effects between any pair of variables are similar.
This is a somewhat specific assumption.
However, using the same techniques as in the proof of \Cref{thm: ident single edge color}, we can prove a second structural identifiability result with a broader potential for applications.
To do this we define the following subfamily of edge-colored DAGs.

\begin{definition}
    \label{def:properlycolored}
    A \emph{properly} edge-colored DAG is an edge-colored DAG $(G,c)$ for which there is no color class of size less than two.
\end{definition}

\begin{definition}
    \label{def:blockedcoloring}
    A \emph{blocked} edge-colored DAG is an edge-colored DAG $(G,c)$ for which any two edges $ij$ and $k\ell$ belong to the same color class only if $j = \ell$.
    If an edge-colored DAG is both blocked and properly colored we call it a \emph{BPEC-DAG}.
    We let $\mathcal{BP}$ denote the family of all BPEC-DAGs.
\end{definition}

The blocked edge-colored DAGs are the edge-colored DAGs in which any two edges of the same color have the same head node.
The reason to consider BPEC-DAG models in practice is motivated from the perspective of community detection as studied in network modeling; see, for instance, \cite{almendra2024irreducible, holland1983stochastic, karwa2024monte, fienberg1985statistical}.
In our context, for each node $i$, we consider its set of parents $\pa_G(i)$.
The parents of $i$ constitute the population of individuals that have a direct causal effect on~$i$.
It is possible that several individuals within this population have very similar (or the same) causal effect on~$i$.
BPEC-DAGs let us model these communities of causes, where each community is comprised of the direct causes of $i$ that have the same causal effect on~$i$.

\begin{remark}
\label{rmk: decomposability}
Regarding inference, the assumption of a blocked edge-coloring also ensures that the likelihood function for the model is \emph{decomposable} in the sense of \cite[Section~2.3]{chickering2002optimal}. In particular, information-theoretic scoring criterion such as the Bayesian Information Criterion (BIC) are decomposable for BPEC-DAGs.
\end{remark}

\begin{remark}
    \label{rmk:compatible}
    The BPEC-DAGs are a subfamily of the compatibly colored DAGs introduced in \cite{makam2022symmetries}.
    A colored DAG $(G,c)$ is \emph{compatibly colored} if whenever $c(ij) =c(k\ell)$ for edges $ij,k\ell$ in $G$ we have that $c(j) = c(\ell)$.
    Compatibly colored DAGs are the family of colored DAGs for which the Maximum Likelihood Estimator (MLE) can be calculated as the solution to a family of least squares problems, one for each vertex color.
    BPEC-DAGs are compatibly colored since the blocked condition means that when two edges are the same color they point to the same node.
    BPEC-DAGs are the compatibly, properly edge-colored DAGs.
    See Appendix~\ref{app: BPEC additional details} for more details.
\end{remark}

The following theorem says that an edge-colored DAG is structurally identifiable from within the family of BPEC-DAGs.
Hence, learning a BPEC-DAG $(G,c)$ from data amounts to learning a causal structure, given by $G$, as well as a \emph{causal community} structure on the direct causes of each node, given by the coloring~$c$.
The proof uses the same novel techniques introduced in \Cref{thm: ident single edge color}.
In \Cref{sec:causaldiscovery}, we give a causal discovery algorithm for learning BPEC-DAG models.

\begin{theorem}\label{theorem:properlyBlocked}
Let $[G, c]$ be the model equivalence class of a BPEC-DAG $(G, c)$.
Then $|\mathcal{BP}\cap [G, c]| = 1$; i.e., an edge-colored DAG is structurally identifiable in the set of BPEC-DAGs.
\end{theorem}

\Cref{theorem:properlyBlocked} is a structural identifiability result for partially homogeneous structural coefficients in analogy to results obtained in the partially homoscedastic setting by  \cite{wu2023partial}.
The model equivalence results obtained in \cite{wu2023partial} yield structural identifiability whenever every color class in a vertex-colored DAG has size at least two.
So a proof of structural identifiability for general properly edge-colored DAGs would be an exact analogue to the structural identifiability results obtained in \cite{wu2023partial}.
Similar to the more general observations on vertex-colored model equivalence made in \cite{wu2023partial}, there exist edge-colored DAGs that define the same model.

\begin{example}
    \label{ex: edge colored nonident}
    The following two edge-colored DAGs, denoted $(G, c)$ and $(G', c')$ respectively, are model equivalent; i.e., they satisfy $\CC M(G, c) = \CC M(G', c')$. This may be verified using \Cref{rem:ApplyVanishingIdeal} with the help of computer algebra software such as \texttt{Macaulay2}.
    \begin{center}
    \begin{tikzpicture}[thick,scale=0.5]

 	 \node[circle, draw, fill=black!0, inner sep=1pt, minimum width=1pt] (1) at (0,0)  {$1$};
	   \node[circle, draw, fill=black!0, inner sep=1pt, minimum width=1pt] (2) at (2,1) {$2$};
          \node[circle, draw, fill=black!0, inner sep=1pt, minimum width=1pt] (3) at (2.5,3) {$3$};
          \node[circle, draw, fill=black!0, inner sep=1pt, minimum width=1pt] (4) at (-2,0) {$4$};
          \node[circle, draw, fill=black!0, inner sep=1pt, minimum width=1pt] (5) at (-2,2) {$5$};
          \node[circle, draw, fill=black!0, inner sep=1pt, minimum width=1pt] (6) at (-3,1) {$6$};

  	 \draw[->,cyan]   (1) -- (2) ;
          \draw[->]   (1) -- (3) ;
          \draw[->]   (2) -- (3) ;
          \draw[->]   (1) -- (4) ;
          \draw[->,cyan]   (4) -- (5) ;
          \draw[->]   (4) -- (6) ;
          \draw[->]   (5) -- (6) ;
    \end{tikzpicture}
    \hspace{1in}
    \begin{tikzpicture}[thick,scale=0.5]

 	 \node[circle, draw, fill=black!0, inner sep=1pt, minimum width=1pt] (1) at (0,0)  {$1$};
	   \node[circle, draw, fill=black!0, inner sep=1pt, minimum width=1pt] (2) at (2,1) {$2$};
          \node[circle, draw, fill=black!0, inner sep=1pt, minimum width=1pt] (3) at (2.5,3) {$3$};
          \node[circle, draw, fill=black!0, inner sep=1pt, minimum width=1pt] (4) at (-2,0) {$4$};
          \node[circle, draw, fill=black!0, inner sep=1pt, minimum width=1pt] (5) at (-2,2) {$5$};
          \node[circle, draw, fill=black!0, inner sep=1pt, minimum width=1pt] (6) at (-3,1) {$6$};

  	 \draw[->,cyan]   (1) -- (2) ;
          \draw[->]   (1) -- (3) ;
          \draw[->]   (2) -- (3) ;
          \draw[<-]   (1) -- (4) ;
          \draw[->,cyan]   (4) -- (5) ;
          \draw[->]   (4) -- (6) ;
          \draw[->]   (5) -- (6) ;
    \end{tikzpicture}
\end{center}
\end{example}

It would be interesting to characterize model equivalence for edge-colored DAGs when some edges are allowed to be uncolored, yielding a full analogue to the result in \cite{wu2023partial}.
This is achievable under certain restrictions on the coloring, but requires significantly more work, which is perhaps more fitting for a follow-up article.
Instead, we pose the following general question.

\begin{question}
    \label{quest: edge-color equivalence}
    What is a graphical characterization of model equivalence for edge-colored DAGs?
\end{question}

A complete answer to \Cref{quest: edge-color equivalence} would be similar to the classical characterization of Markov equivalence of (uncolored) DAGs; i.e., two DAGs are Markov equivalent if and only if they have the same skeleton and v-structures.
However, as seen in the model equivalence characterization of vertex-colored DAGs due to \cite{wu2023partial}, the presence of colors fixes additional substructures in the graphs.
The main challenge is characterizing which substructures are fixed by the presence of colored edges.
For instance, it is not even clear if colored edges $ij$ must always be present (and colored the same) in all model equivalent edge-colored DAGs, or if some can be reversed $ji$ and colored differently to produce the same model.
These challenges are even more unwieldily in the case of general colored DAGs, where faithfulness to $G$ is not guaranteed.
A characterization of model equivalence for $\CC M(G,c)$ admitting faithful distributions would already be significant.

\begin{question}
    \label{quest: equivalence}
    What is a graphical characterization of model equivalence for the family of colored DAGs admitting faithful distributions?
\end{question}
From the discussions above we have that any two model equivalent colored DAGs admitting faithful distributions need to have the same skeleton and v-structures. Thus, depending on how Question \ref{quest: faithfulness} is resolved, it may be possible to impose additional graphical constraints to fully characterize such equivalence. Answering \Cref{quest: equivalence} would be crucial for developing constraint-based causal discovery algorithms (similar to the PC algorithm) for colored DAGs having faithful distributions. It could also yield a transformational characterization of model equivalence, generalizing \cite{chickering2013transformational}, as well provide a notion of essential graphs for colored DAGs.

\section{Causal Discovery}\label{sec:causaldiscovery}
In this section, we apply the structural identifiability results of \Cref{sec: model equivalence} to give a causal discovery algorithm for learning DAG models using partial homogeneity constraints on structural coefficients.
\Cref{theorem:properlyBlocked} provides a family of structurally identifiable edge-colored DAGs; namely, the family of BPEC-DAGs (\Cref{def:blockedcoloring}).
Moreover, the coloring of a BPEC-DAG provides information on how the direct causes of each node in the graph are clustered according to similar causal effects on their target.
Hence, by searching over the family of BPEC-DAGs for an optimal model, we obtain a single DAG estimate of the causal structure (by \Cref{theorem:properlyBlocked}) as well as an estimate of the causal communities around each node in the system.
In \Cref{alg:GECS}, we present the \emph{Greedy Edge-Colored Search} (GECS) for estimating a BIC-optimal BPEC-DAG from a random sample.
\begin{algorithm}
  \caption{Greedy Edge-Colored Search (GECS)}
  \label{alg:GECS}
  \raggedright
  \hspace*{\algorithmicindent} \textbf{Input:} A random sample $\mathbb{D}$ of size $n$ from a distribution of over $p$ variables.\\
  \hspace*{\algorithmicindent} \textbf{Output:} A BPEC-DAG $(G, c)$.
  \begin{algorithmic}[1]
    \State $G \gets$ {$([p], \emptyset)$} \Comment{Initialize at the empty DAG}
    \State $c \gets \{\}$ \Comment{Initialize at the empty coloring}

    \Loop
        \State{$\textrm{score} \gets \textrm{BIC}(G, c ; \mathbb{D})$}
        \State{$G,c \gets \TT{ecDAGmodify}(G,c; \BB D; \TT{addColor}, \TT{splitColor})$} \Comment{Phase 1}
        \State{$G,c \gets \TT{ecDAGmodify}(G,c; \BB D; \TT{addEdge}, \TT{moveEdge}, \TT{reverseEdge}, \TT{removeEdge})$} \Comment{Phase 2}
        \State{$G,c \gets \TT{ecDAGmodify}(G,c; \BB D; \TT{mergeColors}, \TT{removeColor})$} \Comment{Phase 3}
        \If{$\RM{score} = \RM{BIC}(G, c; \BB D)$} \Comment{Stop if local maximum is reached}
        \State \textbf{break}
      \EndIf
    \EndLoop
    \State \Return $(G, c)$
  \end{algorithmic}
\end{algorithm}

Let $\mathbf{x} = (x_{i,j})\in \mathbb{R}^{V\times n}$ be a data matrix in which the columns form a random sample from a joint multivariate Gaussian distribution $P$ on $X = (X_i)_{i\in V}$ with density $f$ and covariance matrix $\Sigma$ belonging to some $\CC M(G, c)$ with $G = (V, E)$.
GECS performs a loop over the updating function $\texttt{ecDAGmodify}$ presented in~\Cref{section:pseudocode}. %
The loop operates on a BPEC-DAG $(G,c)$ and uses the random sample $\mathbf{x}$.
It consists of three modification phases, which are broken down as follows: The first phase loops over transformations of $(G,c)$ that increase the total number of model parameters, the second loops over transformations of $(G,c)$ that keep the parameter count the same, and the third phase loops over transformations that decrease parameter count.
The first and last phase are analogous to the edge addition and edge removal phases, respectively, of the original GES algorithm in \cite{chickering2002optimal}.
The second phase is analogous to the edge reversal phase which was found to improve the performance of GES in \cite{hauser2012characterization}.

Unlike GES, each of the three phases of GECS considers multiple transformations of $(G,c)$ to account for the fact that an edge-colored DAG may change not only by adding, reversing or removing an edge, but also by transforming the colors of edges.
Each move used by GECS returns the BIC-optimal transformation of the input BPEC-DAG, or it returns the input model when no transformation considered by the move improves the BIC score.
The moves in each phase of GECS are broken down as follows:
\begin{enumerate}[leftmargin=3em, rightmargin=2em]
    \item \textbf{Phase 1 (parameter addition phase)} moves:
    \begin{enumerate}[leftmargin=2em, rightmargin=0em]
        \item $\texttt{addColor}(G, c; \mathbf{x})$ considers all possible ways to add a new color class with two edges to $(G,c)$.
        \item $\texttt{splitColor}(G, c; \mathbf{x})$ considers all possible ways to partition a color class into two classes, one with exactly two edges and the other consisting of the remaining edges in the original color class.
    \end{enumerate}
    \item \textbf{Phase 2 (parameter exchange phase)} moves:
    \begin{enumerate}[leftmargin=2em, rightmargin=0em]
        \item $\texttt{addEdge}(G, c; \mathbf{x})$ considers all possible ways to add a single new edge (not present in $G$) to an existing color class.
        \item $\texttt{moveEdge}(G, c; \mathbf{x})$ considers all possible ways to move a single edge to another color class consisting of edges having the same head node as the considered edge.
        \item $\texttt{reverseEdge}(G, c; \mathbf{x})$ considers all possible ways to reverse a single edge, for each edge considering all possible (existing) color classes to which it can be assigned.
        \item $\texttt{removeEdge}(G, c; \mathbf{x})$ considers all possible ways to remove a single edge from a color class in the graph.
    \end{enumerate}
    \item \textbf{Phase 3 (parameter removal phase)} moves:
    \begin{enumerate}[leftmargin=2em, rightmargin=0em]
        \item $\texttt{mergeColors}(G, c; \mathbf{x})$ considers all possible ways to merge two color classes consisting of edges having a common head node.
        \item $\texttt{removeColor}(G, c; \mathbf{x})$ considers all possible ways to remove a single color class from the graph.
    \end{enumerate}
\end{enumerate}

As seen in \Cref{alg:GECS}, GECS loops over each of these phases repeatedly until it completes a loop through all phases without improving the BIC score.
Hence, GECS operates analogous to GES (as implemented in \cite{hauser2012characterization}) while searching over all BPEC-DAGs.
Unlike GES, GECS returns a single DAG, not a Markov equivalence class, since BPEC-DAGs are identifiable (\Cref{theorem:properlyBlocked}).

The moves used by GECS are both specific to the assumption that it considers properly edge-colored DAGs and that the considered edge-colored DAGs are blocked.
For instance, when a new color class is created by the $\texttt{addColor}$ move, it creates a color class containing exactly two edges to ensure the coloring is proper.
It also requires that the two edges in the new color class have the same head node, so as to ensure the result is a blocked edge-colored DAG.
The same principles apply to the formulation of the other seven moves used by GECS.

The moves used by GECS to search the space of BPEC-DAGs were chosen for their relative simplicity.
Similar to GES, phase 1 (the parameter addition phase) should intuitively add sufficiently many parameters to produce a BPEC-DAG $(G,c)$ so that the data-generating distribution is contained in $\CC M(G, c)$; and phase 3 (the parameter removal phase) should remove extraneous parameters to account for overfitting.

Our implementation of GECS uses the Bayesian Information Criterion (BIC) as a score; i.e.,
\[
\textrm{score}(G,c; \mathbf{x}) = \log L(\hat\Lambda, \hat\Omega \mid \mathbb{D}) - \frac{\ln(n)k}{2}
\]
where $(\hat\Lambda, \hat\Omega)$ is the MLE of the parameters $(\Lambda, \Omega)$ for the graph $(G,c)$, $L(\Lambda, \Omega \mid \mathbb{D})$ is the likelihood function and $k$ is the number of free parameters in the model.
For a BPEC-DAG $(G,c)$, we have $k = |V| + |c(E)|$.
As derived in Appendix~\ref{app: BPEC additional details}, the MLE of the model parameters is given by
\[
\begin{split}
    (\hat\lambda_{e})_{e \in \pa_{(G,c)}(k)} &= \textrm{argmin}_{\beta\in\mathbb{R}^{|\pa_{(G,c)}(k)|}}\left|\left|\mathbf{x}_{k:} - \Lambda_{\pa_G(k), k}^T\mathbf{x}_{\pa_G(k):}\right|\right|^2, \\
    \hat\omega_k &= \frac{1}{n}\left|\left|\mathbf{x}_{k:} - \hat\Lambda_{\pa_G(k), k}^T\mathbf{x}_{\pa_G(k):}\right|\right|^2
\end{split}
\]
for $k\in V$
where $\pa_{(G,c)}(k) = \{ e\in c(E) : e= c(jk) \textrm{ for } j\in \pa_G(k)\}$.
We note that the algorithm works for any choice of decomposable score function (see \Cref{rmk: decomposability}).
In particular, BIC is a decomposable score for the family of BPEC-DAG models, as the edge parameters for such a $(G,c)$ are each isolated to a single \emph{family} $\fa_G(i) = \pa_G(i) \cup\{i\}$.
Hence, the BIC of a BPEC-DAG is computed via local regression computations, as for classic (uncolored) Gaussian DAG models (see Appendix~\ref{app: BPEC additional details}).
If desired, one could naturally extend GECS to search over all properly edge-colored DAGs by appropriately augmenting the above list of moves.
However, one may then require scoring models in which the BIC is not necessarily decomposable.

In \cite{chickering2002optimal}, Chickering proved that such a parameter addition phase followed by a parameter deletion phase is sufficient to guarantee consistency of GES when the data-generating distribution is assumed to be faithful to a DAG.
It seems reasonable that GECS would have a similar consistency guarantee.
However, the added complexity of the moves needed to account for coloring considerations likely makes proving such guarantees more challenging.
We leave these questions for consideration in future work, and instead empirically evaluate the performance of GECS on synthetic and real data.
Our implementation of GECS, as well as all necessaries to reproduce the results of the following experiments are available at
\url{https://github.com/soluslab/coloredDAGs}.

\subsection{Synthetic data experiments}
\label{subsec:simulations}
To get a sense of the performance of GECS, we generated synthetic data from random BPEC-DAG models and then tasked both GES and GECS with estimating the data-generating DAG from a random sample drawn from these models.

To generate the random BPEC-DAGs, we first generated an Erd\H{o}s-R\'{e}nyi random DAG $G = ([p], E)$ with its natural topological ordering $\pi = 1\ldots p$ where each possible edge $ij$, with $i < j$, appears with fixed probability $\rho \in (0, 1)$.
The random DAGs were then adjusted to ensure that $|\pa_G(i)| > 1$ for all $i\in [p]$.
Specifically, if the random DAG $G$ contained a node $i$ with $|\pa_G(i)| = 1$, an additional node $j$ was drawn uniformly at random from those $j < i$ not in $\pa_G(i)$.
The edge $ji$ was then added to $G$ to ensure that the edges with head node $i$ may be colored as a BPEC-DAG.

A random (proper) edge-coloring $c$ was then assigned to the resulting DAG $G$.
Here, we introduced a parameter $\texttt{nc}$ which takes a positive integer value indicating how many color classes into which each parent set should be partitioned.
In the case that the pre-specified parameter value $\texttt{nc}$ is larger than the number of parents of $i$ divided by $2$, the value $\texttt{nc}$ was adjusted to $\lfloor|\pa_G(i)| / 2\rfloor$ for the node $i$ to ensure that a BPEC coloring could be assigned to $G$.
The nodes in $\pa_G(i)$ were then assigned to one of $\texttt{nc}$ classes for each node $i$ to produce a BPEC-DAG $(G,c)$.
The BPEC-DAG $(G,c)$ was then parametrized by assigning each color class a random parameter value from $(-1, -0.25] \cup [0.25, 1)$ and a random error variance for each node.

For $p \in\{6, 10\}$ nodes, each $\rho\in \{0.2, 0.3, 0.4, 0.5, 0.6, 0.7, 0.8, 0.9\}$, and each $\texttt{nc}\in\{2,\ldots,p - 1\}$ we generated $25$ random BPEC-DAGs according to the above scheme and drew $n\in \{250, 1000\}$ samples from the resulting BPEC-DAG model.
For each data set, we tasked GECS and GES with recovering the data-generating DAG $G$, and compared the structural Hamming distance (SHD) between the learned DAGs and the ground truth.
For each BPEC-DAG learned by GECS we also recorded the true positive rate of pairs of edges that were assigned the same color as in the data-generating BPEC-DAG; e.g., the \emph{coloring sensitivity}.
The SHD and coloring sensitivity results for $p = 6$ and $n = 1000$ are presented in \Cref{fig:p6n1000} and the SHD results for $p = 6$ and $n = 1000$ are presented in \Cref{fig:p10n1000}.
Plots for all other results can be found in Supplementary material \Cref{section:pseudocode}.

\begin{figure}[t]
    \begin{subfigure}[b]{0.24\textwidth}
    \centering
    \includegraphics[width=\textwidth]{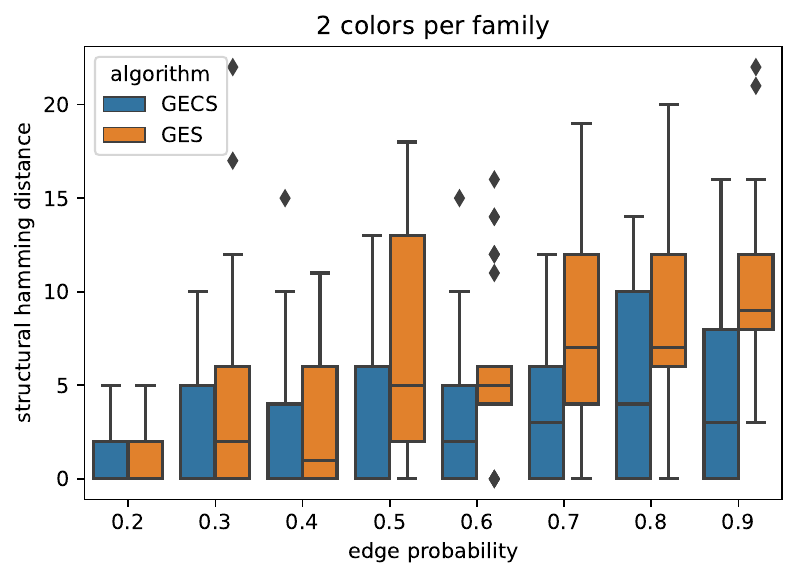}
    \caption{$\texttt{nc} = 2$}
    \label{fig:p6n1000c2}
    \end{subfigure}
    \hfill
    \begin{subfigure}[b]{0.24\textwidth}
    \centering
    \includegraphics[width=\textwidth]{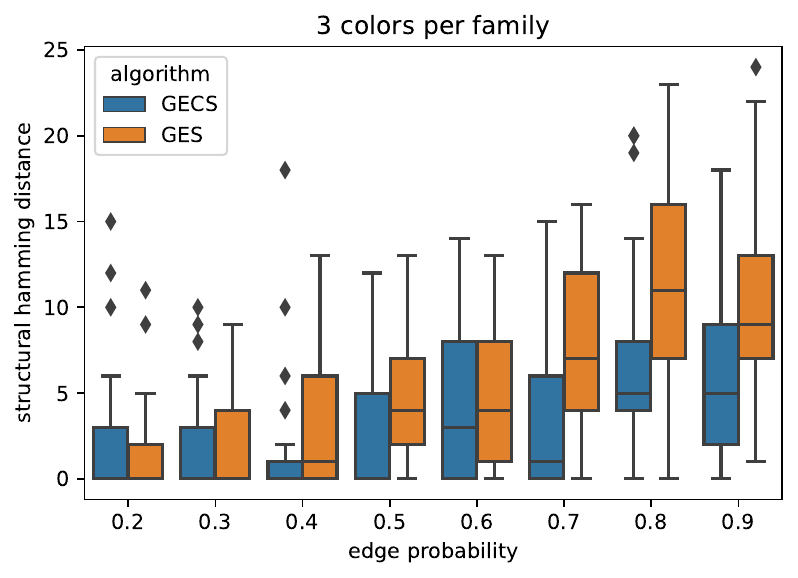}
    \caption{$\texttt{nc} = 3$}
    \label{fig:p6n1000c3}
    \end{subfigure}
    \hfill
    \begin{subfigure}[b]{0.24\textwidth}
    \centering
    \includegraphics[width=\textwidth]{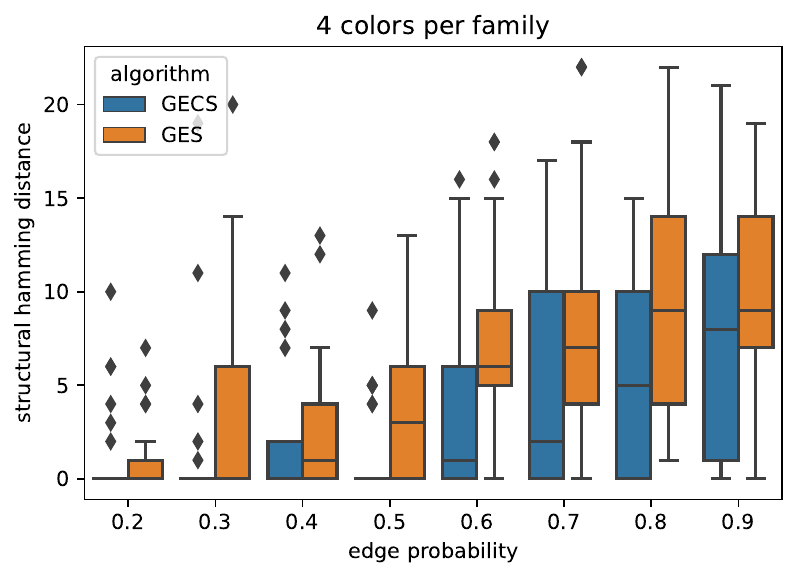}
    \caption{$\texttt{nc} = 4$}
    \label{fig:p6n1000c4}
    \end{subfigure}
    \hfill
    \begin{subfigure}[b]{0.24\textwidth}
    \centering
    \includegraphics[width=\textwidth]{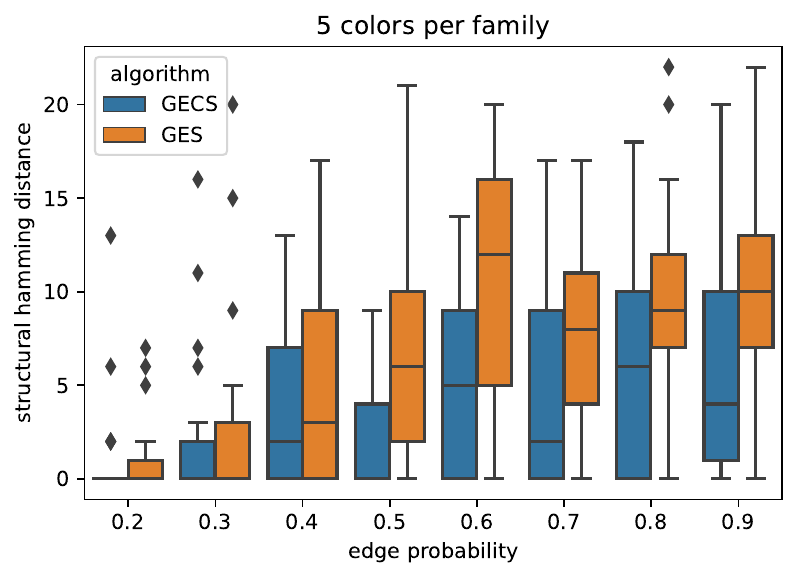}
    \caption{$\texttt{nc} = 5$}
    \label{fig:p6n1000c5}
    \end{subfigure}

    \begin{subfigure}[b]{0.24\textwidth}
    \centering
    \includegraphics[width=\textwidth]{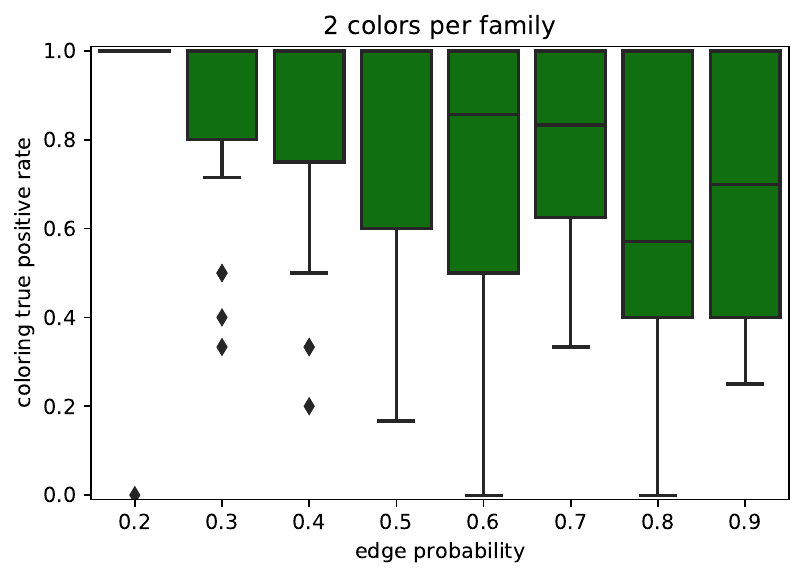}
    \caption{$\texttt{nc} = 2$}
    \label{fig:p6n1000c2cTPR}
    \end{subfigure}
    \hfill
    \begin{subfigure}[b]{0.24\textwidth}
    \centering
    \includegraphics[width=\textwidth]{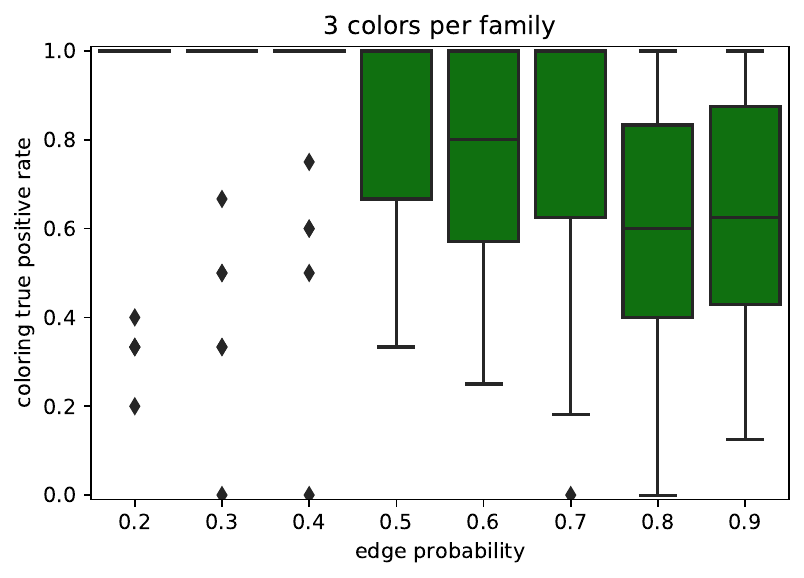}
    \caption{$\texttt{nc} = 3$}
    \label{fig:p6n1000c3cTPR}
    \end{subfigure}
    \hfill
    \begin{subfigure}[b]{0.24\textwidth}
    \centering
    \includegraphics[width=\textwidth]{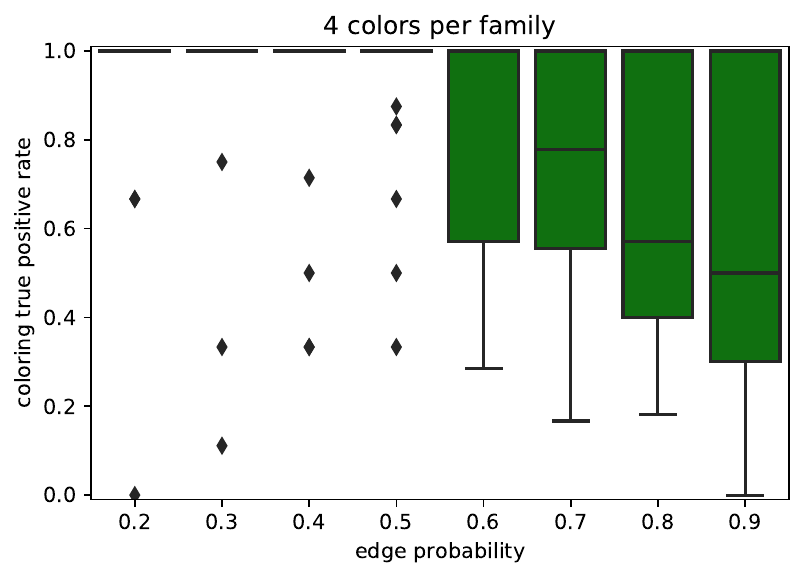}
    \caption{$\texttt{nc} = 4$}
    \label{fig:p6n1000c4cTPR}
    \end{subfigure}
    \hfill
    \begin{subfigure}[b]{0.24\textwidth}
    \centering
    \includegraphics[width=\textwidth]{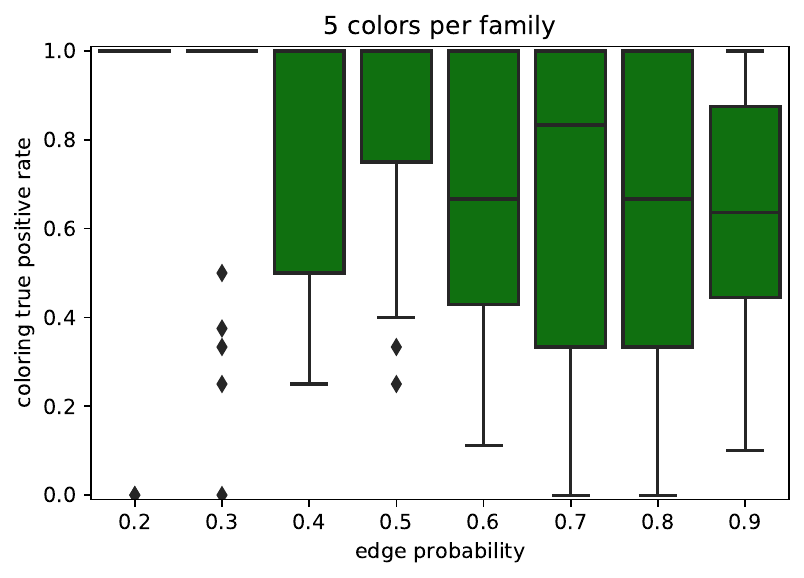}
    \caption{$\texttt{nc} = 5$}
    \label{fig:p6n1000c5cTPR}
    \end{subfigure}
\caption{(a)--(d): Structural Hamming distance results for $p = 6$ and sample size $n = 1000$. (e)--(h): Coloring sensitivity of BPEC-DAGs learned by GECS.}
\label{fig:p6n1000}
\end{figure}

We see from the results presented in \Cref{fig:p6n1000,fig:p10n1000} that GECS generally tends to outperform GES in regards to structural Hamming distance except for a few choices of $(\texttt{nc}, \rho)$.
This~is especially true as the density of the true graph (specified by the edge probability) increases.
This~is likely due to the fact that GECS can include more edges at a lower penalty than GES when modeling dense systems.
Specifically, the penalization term for the BIC used in GES includes the sum of $p$ and the number of edges, whereas the penalization term for GECS is only the sum of $p$ and the number of colors (e.g., the number of free parameters in the edge-colored DAG model).

On the other hand, for very sparse graphs (e.g., $\rho=0.2$) GECS does not appear to exhibit better performance over GES.
This is likely a consequence of a couple of factors:  In these simulations, all graphs are BPEC-DAGs, which means that every edge in the graph has at least one ``twin,'' i.e., if there exists $ij$ in $G$ then there exists $kj$ in $G$ with $k\neq j$.
Extremely sparse graphs with this property will naturally contain sufficiently many v-structures to notably increase the likelihood that the MEC of the uncolored graph $G$ is size one.
Combining this observation with the facts that GES searches over MECs and performs very well in the very sparse setting, we ascertain that GES is very likely to achieve a good structural Hamming distance when $\rho$ is very small.
In this same very sparse setting, GECS is searching over colored BPEC-DAGs (as opposed to uncolored MECs), so the search space for GECS is much larger than for GES, thereby creating more opportunity for error.
However, as the density of the graphs grows these specific advantages in favor of GES disappear and we see GECS emerging with better performance.

As noted above, we leave the question of consistency of GECS under faithfulness open, but the empirical results here provide some supporting evidence when considered together with the corresponding figures for $n = 250$ in Supplementary material \Cref{section:pseudocode}.
Namely, we typically see a decrease in the corresponding SHD  quartile values as we go from $n = 250$ (\Cref{fig:p6n250} for $p = 6$ and \Cref{fig:p10n250} for $p = 10$) to $n = 1000$ (\Cref{fig:p6n1000} for $p = 6$ and \Cref{fig:p10n1000} for $p = 10$).

Regarding coloring sensitivity for GECS, we see from \Cref{fig:p6n1000}~(e)--(h) that GECS tends to learn the correct coloring with a high level of accuracy for sparse graphs, with decreasing performance as graph density increases.
For $\texttt{nc} = 4$ we even observe GECS achieving a color sensitivity of $1$ for edge probabilities as high as $0.5$ for all but a few outliers.
Comparing these plots with the additional plots in Supplementary material \Cref{section:pseudocode} suggests that color sensitivity also improves as sample size increases, providing additional evidence for the consistency of GECS.

\begin{figure}[t]
    \begin{subfigure}[b]{0.24\textwidth}
    \centering
    \includegraphics[width=\textwidth]{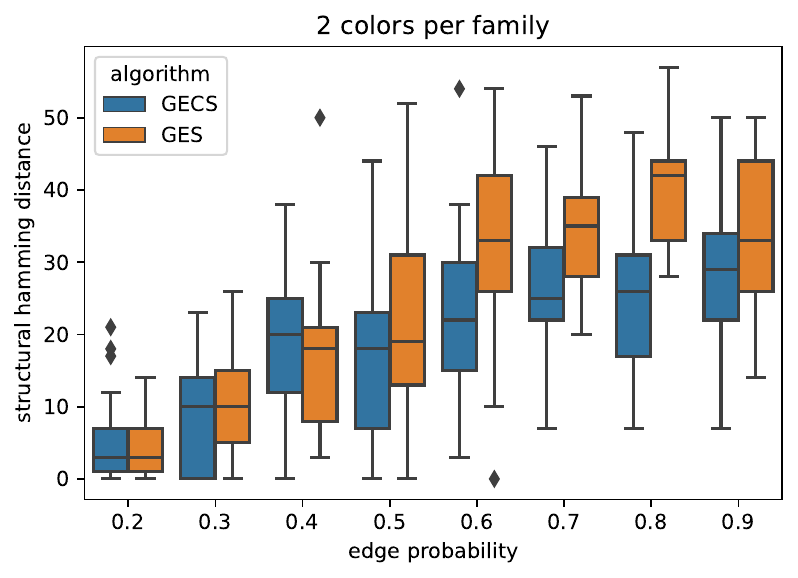}
    \caption{$\texttt{nc} = 2$}
    \label{fig:p10n1000c2}
    \end{subfigure}
    \hfill
    \begin{subfigure}[b]{0.24\textwidth}
    \centering
    \includegraphics[width=\textwidth]{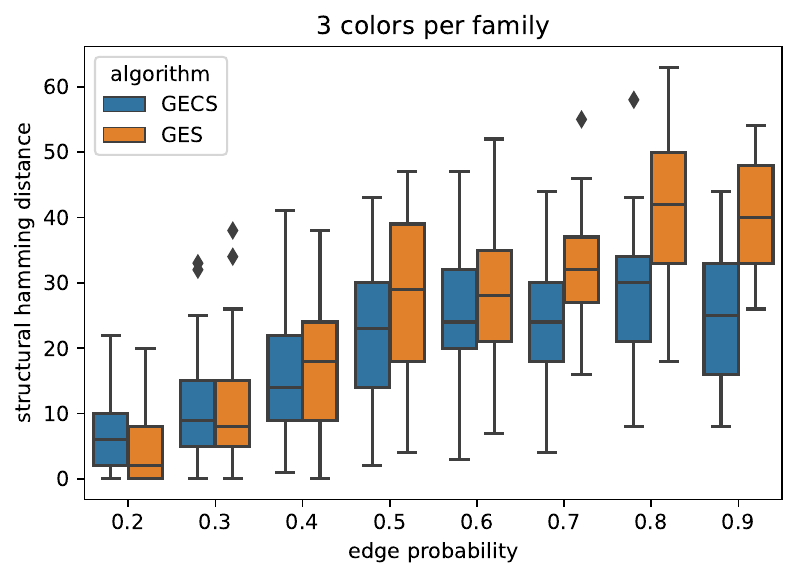}
    \caption{$\texttt{nc} = 3$}
    \label{fig:p10n1000c3}
    \end{subfigure}
    \hfill
    \begin{subfigure}[b]{0.24\textwidth}
    \centering
    \includegraphics[width=\textwidth]{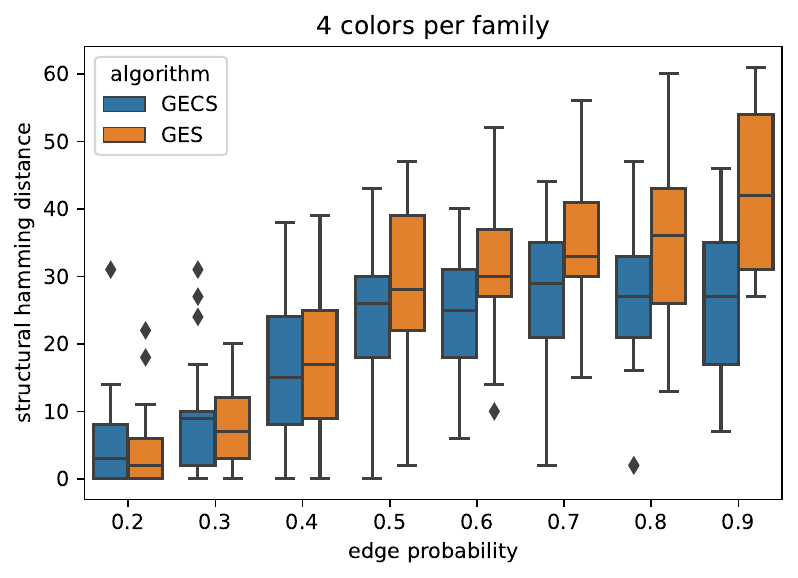}
    \caption{$\texttt{nc} = 4$}
    \label{fig:p10n1000c4}
    \end{subfigure}
    \hfill
    \begin{subfigure}[b]{0.24\textwidth}
    \centering
    \includegraphics[width=\textwidth]{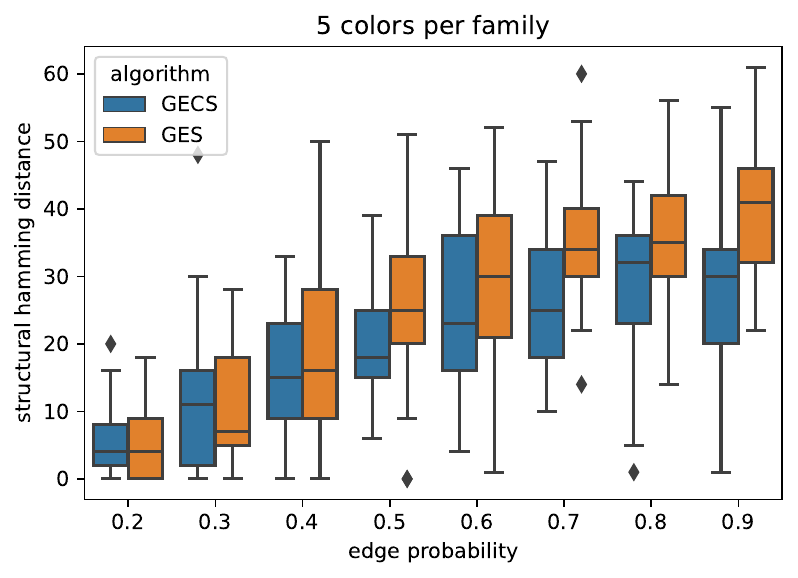}
    \caption{$\texttt{nc} = 5$}
    \label{fig:p10n1000c5}
    \end{subfigure}

    \begin{subfigure}[b]{0.24\textwidth}
    \centering
    \includegraphics[width=\textwidth]{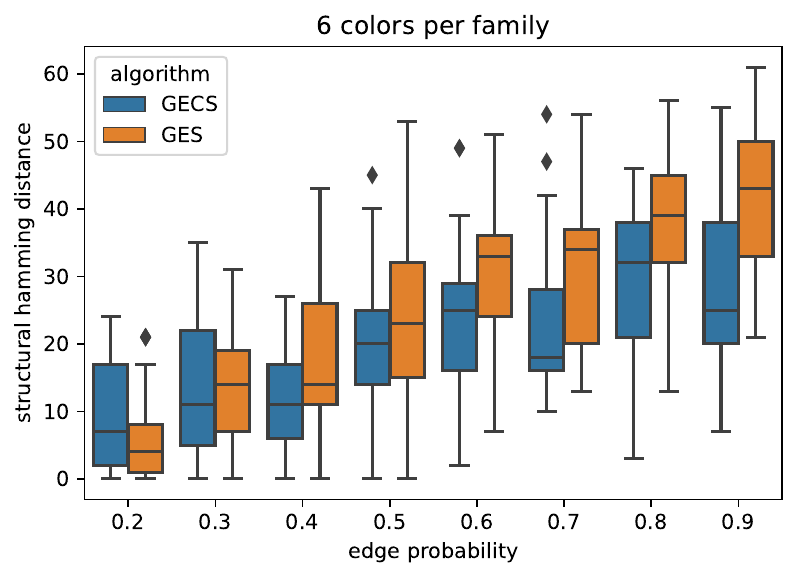}
    \caption{$\texttt{nc} = 6$}
    \label{fig:p10n1000c6}
    \end{subfigure}
    \hfill
    \begin{subfigure}[b]{0.24\textwidth}
    \centering
    \includegraphics[width=\textwidth]{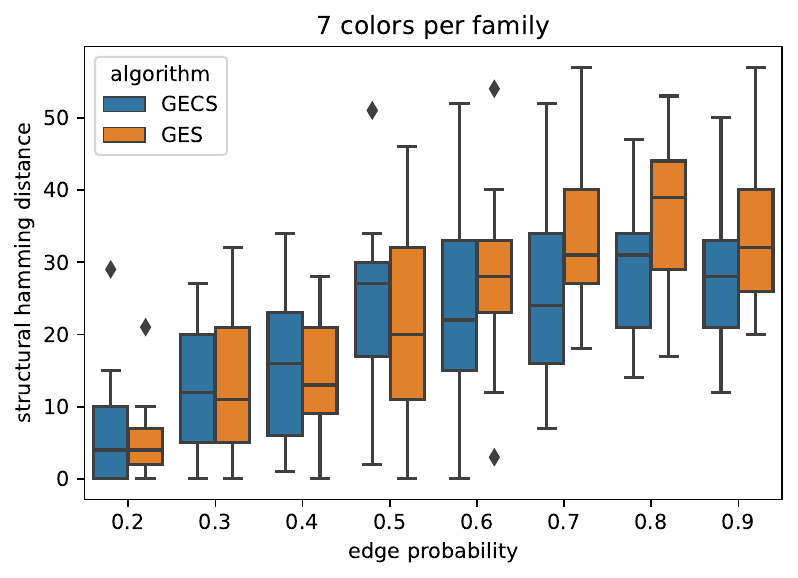}
    \caption{$\texttt{nc} = 7$}
    \label{fig:p10n1000c7}
    \end{subfigure}
    \hfill
    \begin{subfigure}[b]{0.24\textwidth}
    \centering
    \includegraphics[width=\textwidth]{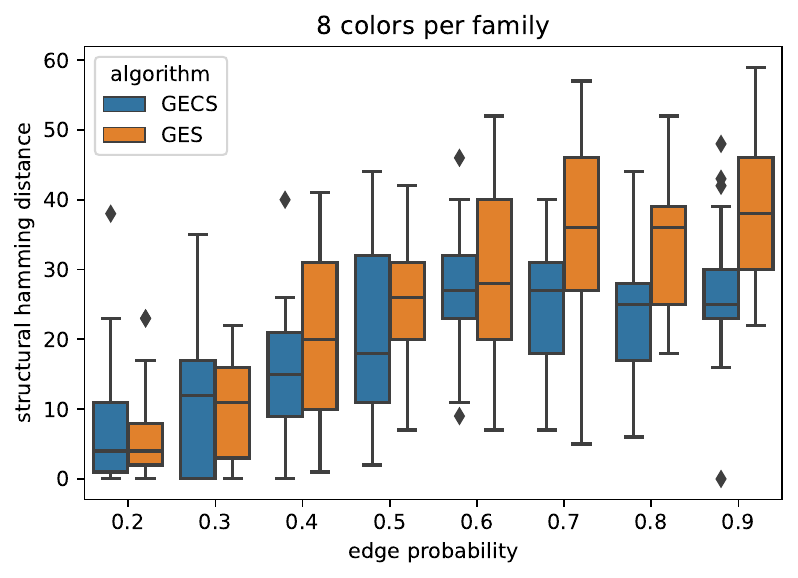}
    \caption{$\texttt{nc} = 8$}
    \label{fig:p10n1000c8}
    \end{subfigure}
    \hfill
    \begin{subfigure}[b]{0.24\textwidth}
    \centering
    \includegraphics[width=\textwidth]{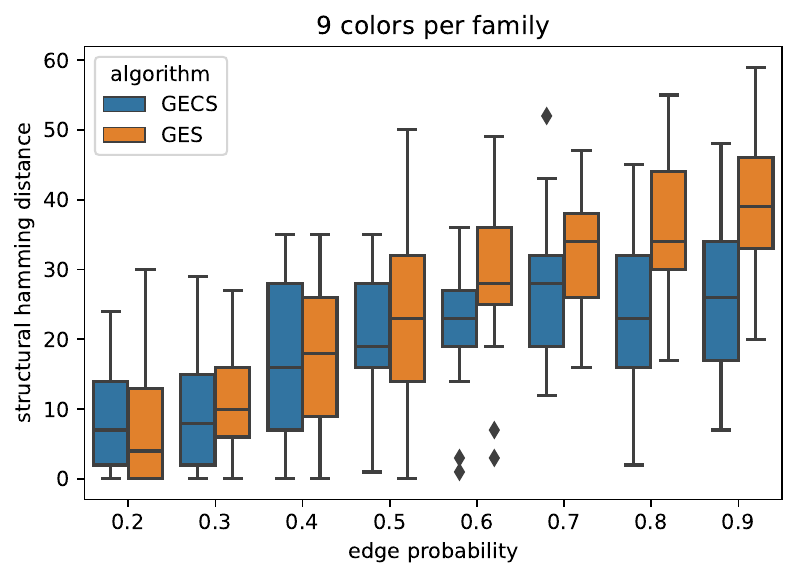}
    \caption{$\texttt{nc} = 9$}
    \label{fig:p10n1000c9}
    \end{subfigure}
\caption{Structural Hamming distance results for $p = 10$ and $n = 1000$ samples. $\texttt{nc}$ is the pre-specified number of colors per family for the data-generating models.}
\label{fig:p10n1000}
\end{figure}

\subsection{Real data experiments}
\label{subsec:realdata}
We ran GECS on three real data sets.
The first two data sets are the Red Wine Quality and White Wine Quality data sets available in the Wine Quality Data Set at the UCI Machine Learning Repository \url{https://archive.ics.uci.edu/dataset/186/wine+quality}.
The third data set is the protein-signaling network data set of Sachs et al.~\cite{sachs2005causal}.

\subsubsection{Wine quality data}\label{subsubsec:winedata}
Each of the two data sets contains samples of $11$ physiochemical properties shared by red and white wine variants of the Portugese Vinho Verde wine which are intended for use as features in the prediction of wine quality.
The mapping of node labels to these physiochemical properties is given in Supplementary material \Cref{tab:WineLabels} in \Cref{section:pseudocode}.
All variables assume numerical values on a continuous (real) domain.
The red wine data set consists of 1599 joint samples of the $11$ physiochemical properties, and the white wine data set has 4898.
More details on the data can be found in \cite{CorCer09}.
We use GECS to give a model of the (causal) dependence structure amongst these $11$ physiochemical properties for both the red wine and white wine data sets, yielding Gaussian hierarchical models that may be used in the prediction of the overall quality of the wine.

The learned BPEC-DAGs for both wine types are presented in \Cref{fig:wine}.
Since BPEC-DAGs are structurally identifiable, we interpret their edges causally.
Indeed, we see that examples of learned arrows have believable causal meaning.
For instance, in the white wine DAG, the arrow from $4$ (chlorides) to $7$ (density) indicates a causal effect of chlorides on the density of the wine.
Some of the causal communities captured by the coloring also appear to suggest that related chemical compounds are exhibiting similar causal effects on their common targets.
For example, the parent set of node $4$ (chlorides) contains the community
\[
\{5 \mbox{ (free sulfur dioxide)}, 6 \mbox{ (total sulfur dioxide)}, 9 \mbox{ (sulphates)}\}.
\]

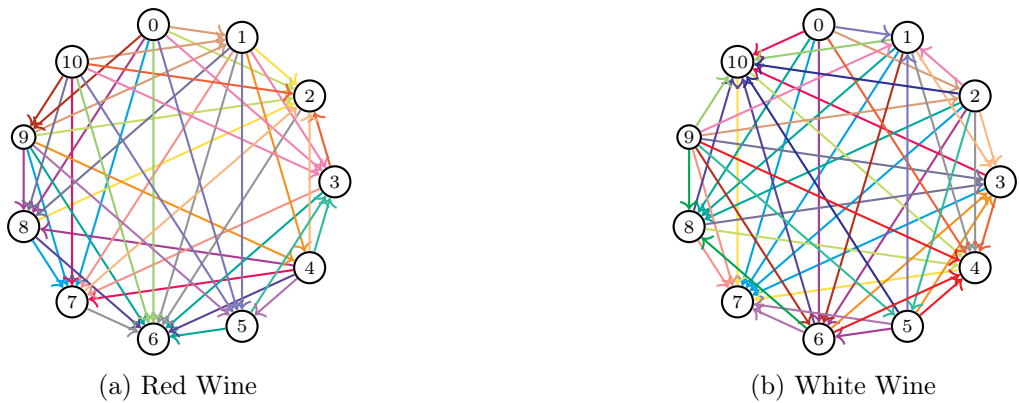
\begin{figure}[t]
    \centering
    \begin{subfigure}[b]{0.45\textwidth}
    \centering
    \begin{tikzpicture}[thick, scale=0.175]

    \def \n {11}
    \def \radius {12cm}
    \def \margin {3} %

  \node[circle, draw, inner sep=2pt, minimum width=2pt] (0) at ({360/\n * (3)}:\radius) {\tiny 0};
  \node[circle, draw, inner sep=2pt, minimum width=2pt] (1) at ({360/\n * (2)}:\radius) {\tiny 1};
  \node[circle, draw, inner sep=2pt, minimum width=2pt] (2) at ({360/\n * (1)}:\radius) {\tiny 2};
  \node[circle, draw, inner sep=2pt, minimum width=2pt] (3) at ({360/\n * (11)}:\radius) {\tiny 3};
  \node[circle, draw, inner sep=2pt, minimum width=2pt] (4) at ({360/\n * (10)}:\radius) {\tiny 4};
  \node[circle, draw, inner sep=2pt, minimum width=2pt] (5) at ({360/\n * (9)}:\radius) {\tiny 5};
  \node[circle, draw, inner sep=2pt, minimum width=2pt] (6) at ({360/\n * (8)}:\radius) {\tiny 6};
  \node[circle, draw, inner sep=2pt, minimum width=2pt] (7) at ({360/\n * (7)}:\radius) {\tiny 7};
  \node[circle, draw, inner sep=2pt, minimum width=2pt] (8) at ({360/\n * (6)}:\radius) {\tiny 8};
  \node[circle, draw, inner sep=1pt, minimum width=1pt] (9) at ({360/\n * (5)}:\radius) {\tiny 9};
  \node[circle, draw, inner sep=1pt, minimum width=1pt] (10) at ({360/\n * (4)}:\radius) {\tiny 10};

\draw[->,Cerulean] (0) -- (7) ;
\draw[->,Cerulean] (8) -- (7) ;
\draw[->,Cerulean] (9) -- (7) ;
\draw[->,Goldenrod] (8) -- (2) ;
\draw[->,Goldenrod] (1) -- (2) ;
\draw[->,Mulberry] (0) -- (8) ;
\draw[->,Mulberry] (4) -- (8) ;
\draw[->,Mulberry] (9) -- (8) ;
\draw[->,OrangeRed] (10) -- (7) ;
\draw[->,OrangeRed] (4) -- (7) ;
\draw[->,JungleGreen] (3) -- (6) ;
\draw[->,JungleGreen] (5) -- (6) ;
\draw[->,JungleGreen] (9) -- (6) ;
\draw[->,Salmon] (1) -- (7) ;
\draw[->,Salmon] (3) -- (7) ;
\draw[->,SpringGreen] (9) -- (2) ;
\draw[->,SpringGreen] (0) -- (2) ;
\draw[->,BurntOrange] (9) -- (4) ;
\draw[->,BurntOrange] (1) -- (4) ;
\draw[->,Tan] (0) -- (1) ;
\draw[->,Tan] (9) -- (1) ;
\draw[->,Tan] (10) -- (1) ;
\draw[->,CadetBlue] (10) -- (8) ;
\draw[->,CadetBlue] (1) -- (8) ;
\draw[->,Violet] (8) -- (6) ;
\draw[->,Violet] (4) -- (6) ;
\draw[->,Gray] (2) -- (6) ;
\draw[->,Gray] (1) -- (6) ;
\draw[->,Gray] (7) -- (6) ;
\draw[->,BrickRed] (10) -- (9) ;
\draw[->,BrickRed] (0) -- (9) ;
\draw[->,CarnationPink] (0) -- (3) ;
\draw[->,CarnationPink] (1) -- (3) ;
\draw[->,CarnationPink] (10) -- (3) ;
\draw[->,RedOrange] (10) -- (2) ;
\draw[->,RedOrange] (3) -- (2) ;
\draw[->,YellowGreen] (0) -- (6) ;
\draw[->,YellowGreen] (10) -- (6) ;
\draw[->,Apricot] (7) -- (2) ;
\draw[->,Apricot] (4) -- (2) ;
\draw[->,SeaGreen] (5) -- (3) ;
\draw[->,SeaGreen] (4) -- (3) ;
\draw[->,Periwinkle] (10) -- (5) ;
\draw[->,Periwinkle] (0) -- (5) ;
\draw[->,Periwinkle] (1) -- (5) ;
\draw[->,Orchid] (9) -- (5) ;
\draw[->,Orchid] (4) -- (5) ;

    \end{tikzpicture}
    \caption{Red Wine}
    \label{fig:redwine}
    \end{subfigure}
    \hfill
    \begin{subfigure}[b]{0.45\textwidth}
    \centering
    \begin{tikzpicture}[thick, scale=0.175]

    \def \n {11}
    \def \radius {12cm}
    \def \margin {3} %

  \node[circle, draw, inner sep=2pt, minimum width=2pt] (0) at ({360/\n * (3)}:\radius) {\tiny 0};
  \node[circle, draw, inner sep=2pt, minimum width=2pt] (1) at ({360/\n * (2)}:\radius) {\tiny 1};
  \node[circle, draw, inner sep=2pt, minimum width=2pt] (2) at ({360/\n * (1)}:\radius) {\tiny 2};
  \node[circle, draw, inner sep=2pt, minimum width=2pt] (3) at ({360/\n * (11)}:\radius) {\tiny 3};
  \node[circle, draw, inner sep=2pt, minimum width=2pt] (4) at ({360/\n * (10)}:\radius) {\tiny 4};
  \node[circle, draw, inner sep=2pt, minimum width=2pt] (5) at ({360/\n * (9)}:\radius) {\tiny 5};
  \node[circle, draw, inner sep=2pt, minimum width=2pt] (6) at ({360/\n * (8)}:\radius) {\tiny 6};
  \node[circle, draw, inner sep=2pt, minimum width=2pt] (7) at ({360/\n * (7)}:\radius) {\tiny 7};
  \node[circle, draw, inner sep=2pt, minimum width=2pt] (8) at ({360/\n * (6)}:\radius) {\tiny 8};
  \node[circle, draw, inner sep=1pt, minimum width=1pt] (9) at ({360/\n * (5)}:\radius) {\tiny 9};
  \node[circle, draw, inner sep=1pt, minimum width=1pt] (10) at ({360/\n * (4)}:\radius) {\tiny 10};

\draw[->, Cerulean] (0) -- (7) ;
\draw[->, Cerulean] (1) -- (7) ;
\draw[->,Cerulean] (2) -- (7) ;
\draw[->,Cerulean] (3) -- (7) ;
\draw[->,Goldenrod] (10) -- (7) ;
\draw[->,Goldenrod] (4) -- (7) ;
\draw[->,Mulberry] (0) -- (6) ;
\draw[->,Mulberry] (2) -- (6) ;
\draw[->,Mulberry] (5) -- (6) ;
\draw[->,OrangeRed] (3) -- (10) ;
\draw[->,OrangeRed] (0) -- (10) ;
\draw[->,JungleGreen] (0) -- (8) ;
\draw[->,JungleGreen] (1) -- (8) ;
\draw[->,JungleGreen] (2) -- (8) ;
\draw[->,Salmon] (9) -- (7) ;
\draw[->,Salmon] (8) -- (7) ;
\draw[->,SpringGreen] (8) -- (4) ;
\draw[->,SpringGreen] (10) -- (4) ;
\draw[->,BurntOrange] (5) -- (3) ;
\draw[->,BurntOrange] (6) -- (3) ;
\draw[->,Tan] (9) -- (2) ;
\draw[->,Tan] (0) -- (2) ;
\draw[->,CadetBlue] (9) -- (3) ;
\draw[->,CadetBlue] (8) -- (3) ;
\draw[->,Violet] (6) -- (10) ;
\draw[->,Violet] (8) -- (10) ;
\draw[->,Gray] (2) -- (4) ;
\draw[->,Gray] (1) -- (4) ;
\draw[->,BrickRed] (9) -- (6) ;
\draw[->,BrickRed] (1) -- (6) ;
\draw[->,CarnationPink] (9) -- (1) ;
\draw[->,CarnationPink] (2) -- (1) ;
\draw[->,RedOrange] (3) -- (4) ;
\draw[->,RedOrange] (0) -- (4) ;
\draw[->,YellowGreen] (9) -- (10) ;
\draw[->,YellowGreen] (1) -- (10) ;
\draw[->,Apricot] (2) -- (3) ;
\draw[->,Apricot] (1) -- (3) ;
\draw[->,SeaGreen] (9) -- (5) ;
\draw[->,SeaGreen] (2) -- (5) ;
\draw[->,Periwinkle] (5) -- (1) ;
\draw[->,Periwinkle] (0) -- (1) ;
\draw[->,Orchid] (6) -- (7) ;
\draw[->,Orchid] (5) -- (7) ;
\draw[->,Blue] (5) -- (10) ;
\draw[->,Blue] (2) -- (10) ;
\draw[->,Green] (9) -- (8) ;
\draw[->,Green] (6) -- (8) ;
\draw[->,Red] (6) -- (4) ;
\draw[->,Red] (5) -- (4) ;
\draw[->,Red] (9) -- (4) ;

    \end{tikzpicture}
    \caption{White Wine}
    \label{fig:whitewine}
    \end{subfigure}
    \caption{(a) Results for Red Wine data set. The learned BPEC-DAG has $47$ edges with $20$ color classes. (b) Results for White Wine data set. The learned BPEC-DAG has $51$ edges with $23$ color classes.}
    \label{fig:wine}
\end{figure}

\subsubsection{Protein signaling data}\label{subsubsec:proteindata}
While GECS learned BPEC-DAG models for the wine data sets with relatively rich structure and numerous causal communities, it is also possible that BPEC-DAG models do not provide reasonable models for certain real data problems.
For example, we considered the Sachs et al. data set consisting of 7644 abundance measurements of certain phospholipids and phosphoproteins present in primary human immune system cells \cite{sachs2005causal}.
The data set is purely interventional; however one may extract an observational data set consisting of $1755$ samples as described in \cite{wang2017permutation, squires2020permutation}.

We ran GECS on this observational data set of 1755 samples over the $11$ measured molecules.
The learned BPEC-DAG, depicted in \Cref{fig:sachs}, has 26 edges and 13 colors classes.
We note that each color class contains precisely two edges.
One explanation for this could be that in each color class there exists one edge that exhibited a strong effect according to the data, and, since BPEC-DAGs are properly colored, GECS was forced to also add an additional edge with an optimally similar causal effect.
Hence, the learned parameter values for the color classes in \Cref{fig:sachs} may not be an optimal fit to the data.
(Note that this is less likely to have occurred for the wine data sets as those learned BPEC-DAGs each contain multiple color classes with more than two edges, suggesting that some causal effects were similar enough to justify larger (nonminimal) coloring classes.)
Thus, it may not be reasonable to conclude that the DAG in \Cref{fig:sachs} is a good approximation of the true causal structure of the protein network.
In particular, it is unlikely that any two causal effects in this protein-signaling network are the same.
On the other hand, the adjacency matrices of the two DAGs agree in about $74\%$ of their entries. %

\begin{figure}[t]
    \centering
    \begin{subfigure}[b]{0.45\textwidth}
    \centering
    \begin{tikzpicture}[thick, scale=0.17]

    \def \n {11}
    \def \radius {12cm}
    \def \margin {3} %

  \node[circle, draw, inner sep=2pt, minimum width=2pt] (0) at ({360/\n * (3)}:\radius) {\tiny Raf};
  \node[circle, draw, inner sep=2pt, minimum width=2pt] (1) at ({360/\n * (2)}:\radius) {\tiny Mek};
  \node[circle, draw, inner sep=2pt, minimum width=2pt] (2) at ({360/\n * (1)}:\radius) {\tiny PLCg};
  \node[circle, draw, inner sep=2pt, minimum width=2pt] (3) at ({360/\n * (11)}:\radius) {\tiny PIP2};
  \node[circle, draw, inner sep=2pt, minimum width=2pt] (4) at ({360/\n * (10)}:\radius) {\tiny PIP3};
  \node[circle, draw, inner sep=2pt, minimum width=2pt] (5) at ({360/\n * (9)}:\radius) {\tiny Erk};
  \node[circle, draw, inner sep=2pt, minimum width=2pt] (6) at ({360/\n * (8)}:\radius) {\tiny Akt};
  \node[circle, draw, inner sep=2pt, minimum width=2pt] (7) at ({360/\n * (7)}:\radius) {\tiny PKA};
  \node[circle, draw, inner sep=2pt, minimum width=2pt] (8) at ({360/\n * (6)}:\radius) {\tiny PKC};
  \node[circle, draw, inner sep=2pt, minimum width=1pt] (9) at ({360/\n * (5)}:\radius) {\tiny p38};
  \node[circle, draw, inner sep=2pt, minimum width=1pt] (10) at ({360/\n * (4)}:\radius) {\tiny JNK};

  \draw[->,Cerulean] (7) -- (6) ;
\draw[->,Cerulean] (5) -- (6) ;
\draw[->,Goldenrod] (9) -- (1) ;
\draw[->,Goldenrod] (0) -- (1) ;
\draw[->,Mulberry] (10) -- (9) ;
\draw[->,Mulberry] (8) -- (9) ;
\draw[->,OrangeRed] (4) -- (3) ;
\draw[->,OrangeRed] (2) -- (3) ;
\draw[->,JungleGreen] (5) -- (7) ;
\draw[->,JungleGreen] (1) -- (7) ;
\draw[->,Salmon] (10) -- (1) ;
\draw[->,Salmon] (8) -- (1) ;
\draw[->,SpringGreen] (9) -- (7) ;
\draw[->,SpringGreen] (0) -- (7) ;
\draw[->,BurntOrange] (7) -- (2) ;
\draw[->,BurntOrange] (4) -- (2) ;
\draw[->,Tan] (10) -- (8) ;
\draw[->,Tan] (5) -- (8) ;
\draw[->,CadetBlue] (7) -- (3) ;
\draw[->,CadetBlue] (1) -- (3) ;
\draw[->,Violet] (10) -- (7) ;
\draw[->,Violet] (8) -- (7) ;
\draw[->,Gray] (7) -- (4) ;
\draw[->,Gray] (1) -- (4) ;
\draw[->,BrickRed] (3) -- (6) ;
\draw[->,BrickRed] (1) -- (6) ;
    \end{tikzpicture}
    \caption{}
    \label{fig:sachs}
    \end{subfigure}
    \hfill
\begin{subfigure}[b]{0.45\textwidth}
    \centering
    \begin{tikzpicture}[thick, scale=0.17]

\def\n{11}
\def \radius {12cm}
\def \margin {3} %

  \node[circle, draw, inner sep=2pt, minimum width=1pt] (Raf) at ({360/\n * (3)}:\radius) {\tiny Raf};
  \node[circle, draw, inner sep=2pt, minimum width=1pt] (Mek) at ({360/\n * (2)}:\radius) {\tiny Mek};
  \node[circle, draw, inner sep=2pt, minimum width=1pt] (PLCg) at ({360/\n * (1)}:\radius) {\tiny PLCg};
  \node[circle, draw, inner sep=2pt, minimum width=1pt] (PIP2) at ({360/\n * (11)}:\radius) {\tiny PIP2};
  \node[circle, draw, inner sep=2pt, minimum width=1pt] (PIP3) at ({360/\n * (10)}:\radius) {\tiny PIP3};
  \node[circle, draw, inner sep=2pt, minimum width=1pt] (Erk) at ({360/\n * (9)}:\radius) {\tiny Erk};
  \node[circle, draw, inner sep=2pt, minimum width=1pt] (Akt) at ({360/\n * (8)}:\radius) {\tiny Akt};
  \node[circle, draw, inner sep=2pt, minimum width=1pt] (PKA) at ({360/\n * (7)}:\radius) {\tiny PKA};
  \node[circle, draw, inner sep=2pt, minimum width=1pt] (PKC) at ({360/\n * (6)}:\radius) {\tiny PKC};
  \node[circle, draw, inner sep=2pt, minimum width=1pt] (p38) at ({360/\n * (5)}:\radius) {\tiny p38};
  \node[circle, draw, inner sep=2pt, minimum width=1pt] (JNK) at ({360/\n * (4)}:\radius) {\tiny JNK};

  \draw[->] (Raf) -- (Mek) ;
  \draw[<-] (Raf) -- (Mek) ;
  \draw[->] (PKC) -- (p38) ;
  \draw[->] (PKC) -- (JNK) ;
  \draw[->] (PKC) -- (Raf) ;
  \draw[->] (PKC) -- (Mek) ;
  \draw[->] (PKA) -- (p38) ;
  \draw[->] (PKA) -- (JNK) ;
  \draw[->] (PKA) -- (Raf) ;
  \draw[->] (PKA) -- (Mek) ;
  \draw[->] (PKA) -- (Erk) ;
  \draw[->] (PKA) -- (Akt) ;
  \draw[->] (PIP3) -- (Akt) ;
  \draw[->] (PIP3) -- (PLCg) ;
  \draw[<-] (PIP3) -- (PLCg) ;
  \draw[->] (PIP3) -- (PIP2) ;
  \draw[<-] (PIP3) -- (PIP2) ;
  \draw[->] (PLCg) -- (PIP2) ;
  \draw[<-] (PLCg) -- (PIP2) ;
  \draw[->] (PLCg) -- (PKC) ;
  \draw[<-] (PLCg) -- (PKC) ;
  \draw[->] (Mek) -- (Erk) ;
\end{tikzpicture}
    \caption{}
    \label{fig:sachsgroundtruth}
    \end{subfigure}
    \caption{(a) BPEC-DAG learned by GECS for the Sachs dataset. (b) Conventionally accepted ground-truth network according to the model from \cite{sachs2005causal}.}
    \label{fig:sachsfig}
\end{figure}
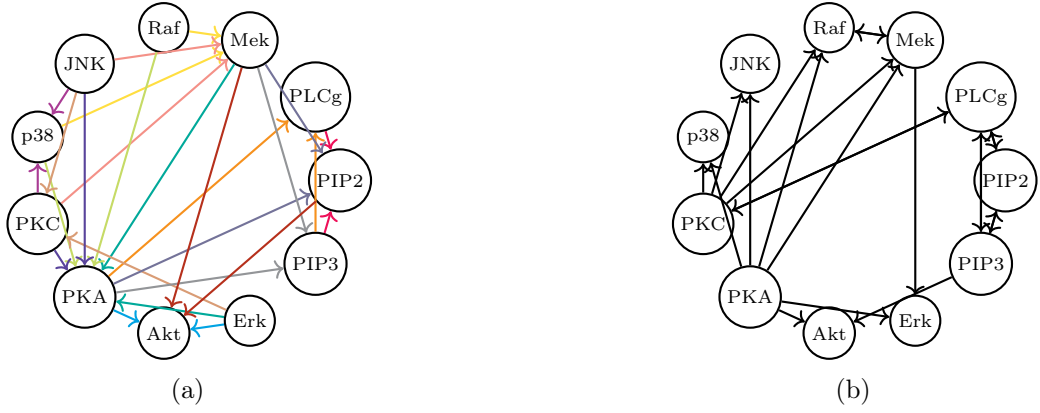

\section{Discussion}
\label{sec:discussion}
In this paper, we considered aspects of Gaussian DAG models under the assumption of partial homogeneity constraints imposed on the vertices and edges of the DAG (i.e., the error variances and structural coefficients, respectively).
Using the graphical representation of such constraints via colored DAGs introduced in \cite{makam2022symmetries}, we investigated fundamental properties of these models for their use in graphical modeling inference tasks.
To embed the colored DAG representation of such models into the theory of DAG models and their many generalizations in the graphical models literature, we described both a local and global Markov property for colored Gaussian DAG models, and showed that their fulfillment is equivalent to containment in the model.

We observed fundamental geometric properties of these models, including smoothness which is useful for likelihood ratio testing \cite{DrtonSmooth}.
We also gave a proof of a conjecture of Sullivant \cite{Sullivant} by showing that the result holds more generally for colored Gaussian DAG models.
Our method of proof relies on a short lemma that applies to other well-studied Gaussian graphical models, including colored (and uncolored) undirected Gaussian graphical models and directed ancestral Gaussian graphical models.
\Cref{lemma:Saturation} is generally applicable to rationally parameterized models satisfying global rational identifiability.
In particular, it provides a tool for identifying a Markov property (e.g.~a set of model-defining constraints) for such models (\Cref{rem:MarkovProperty}).
It would be interesting to study the applications of this result in settings where one wishes to identify testable, model-defining constraints, such as in the evolutionary biology literature \cite{allman2007phylogenetic, casanellas2011relevant}.

We further investigated fundamental properties of the colored DAG models for use in causal discovery, including faithfulness and structural identifiability.
We observed that generic distributions in vertex-colored and edge-colored DAG models are faithful to their defining DAGs and colorings.
However, we also observed that faithfulness may breakdown for general colorings.
As faithfulness is typically used as a first benchmark consistency guarantee for causal discovery algorithms, it would be useful to characterize the colored DAG models admitting faithful distributions.

We additionally gave some first structural identifiability results for models with partially homogeneous structural coefficients in the form of the family of BPEC-DAGs.
These models allow for structure identification, but they also provide additional information on how the direct causes of each node in the system may be clustered into communities of variables having similar causal effects on their target node.
We observed some evidence of this behavior in the real data studies for the GECS causal discovery algorithm, which we introduced as an analogue of GES for learning BPEC-DAGs.
In simulated data, we observed that GECS generally appears to outperform GES at learning causal DAGs, especially for dense models, while offering the additional benefits of causal community identification and structurally identifiable DAG estimates.

We left open the question of consistency of GECS under faithfulness, but provided supporting empirical evidence for consistency via simulations.
It would also be interesting to see if the structural identifiability results generalize to all properly edge-colored DAG models, or even to provide a characterization of model equivalence for general edge-colored DAG models (\Cref{quest: edge-color equivalence}).
Given such results, GECS could be extended to these domains.
Here, characterizations of model equivalence that are both constraint-based (in the style of Verma and Pearl \cite{VermaPearl}) as well as transformational (as considered by Chickering for DAGs \cite{chickering2013transformational}) may be of use in formulating a version of Meek's Conjecture \cite{chickering2002optimal} for colored DAGs whose eventual proof could yield consistency results.
Given a complete characterization of the colored DAGs admitting faithful distributions (\Cref{quest: faithfulness}), one may even consider extensions of these investigations for general colored DAG models (\Cref{quest: equivalence}).

\subsection*{Acknowledgements}
We thank Mathias Drton, Lisa Nicklasson, and Seth Sullivant for helpful discussions.
Part of this research was performed while the authors were visiting the Institute for Mathematical and Statistical Innovation (IMSI), which is supported by the National Science Foundation (Grant No.\ DMS-1929348).
Tobias Boege and Kaie Kubjas were partially supported by the Academy of Finland grant number 323416.
Tobias Boege and Liam Solus were partially supported by the Wallenberg Autonomous Systems and Software Program (WASP) funded by the Knut and Alice Wallenberg Foundation.
Pratik Misra received funding from the Brummers and Partners MathDataLab and the European Research Council (ERC) under the European Union’s Horizon 2020 research and innovation programme (grant agreement No.\ 883818) during the course of this project.
Liam Solus was further supported by a Research Pairs grant from the Digital Futures Lab at KTH, the G\"oran Gustafsson Stiftelse Prize for Young Researchers, and Starting Grant No.\ 2019-05195 from The Swedish Research Council (Vetenskapsr\aa{}det).

\bibliographystyle{tboege}
\bibliography{literature}

\clearpage
\setcounter{page}{1}
\thispagestyle{empty}
\begin{center}\large{\bf COLORED GAUSSIAN DAG MODELS: SUPPLEMENTARY MATERIAL}
\end{center}

\appendix

\section{Proof of \texorpdfstring{\Cref{thm:IdentifyingSets}}{the identifying sets} and \texorpdfstring{\Cref{thm: MP}}{Markov property theorems}}
\label{section:IdentifyingSets}

We first give a self-contained proof of \Cref{lemma:Ident} which is the basis of our parameter identification results.

\begin{proof}[Proof of \Cref{lemma:Ident}]
\begin{paradesc}
\item[\eqref{eq:omega}]
As consequences of the trek rule, we have $\Sigma_{\pa(i),i} = \Sigma_{\pa(i)} \Lambda_{\pa(i),i}$ as well as $\sigma_{ii} = \omega_{i} + \Lambda_{i,\pa(i)} \Sigma_{\pa(i)} \Lambda_{\pa(i),i}$. Using the former to replace $\Lambda_{\pa(i),i}$ in the latter and rearranging terms yields
\[
  \omega_{i} = \sigma_{ii} - \Sigma_{i,\pa(i)} \Sigma_{\pa(i)}^{-1} \Sigma_{\pa(i),i} = \frac{|\Sigma_{\pa(i) \cup i}|}{|\Sigma_{\pa(i)}|},
\]
by Schur's formula.

\item[\eqref{eq:lambda}]
Suppose that $i \not\in \pa(j)$, so $\lambda_{ij} = 0$. In this case, the directed local Markov property of~$G$ dictates that $\CI{i,j|\pa(j)}$ holds for $\Sigma$ and hence the numerator in \eqref{eq:lambda} vanishes as~required. Now~suppose that $i \in \pa(j)$. As above, the trek rule in matrix form reads $\Sigma_{\pa(j),j} = \Sigma_{\pa(j)} \Lambda_{\pa(j),j}$. This implies that $\lambda_{ij}$ is the $i$-th entry of $\Sigma_{\pa(j)}^{-1} \Sigma_{\pa(j),j}$, i.e.,
\begin{align*}
  \lambda_{ij} &= (\Sigma_{\pa(j)}^{-1} \Sigma_{\pa(j),j})_i
  = \frac{(-1)^i}{|\Sigma_{\pa(j)}|} \sum_{s \in \pa(j)} (-1)^s |\Sigma_{is|\pa(j) \setminus is}| \sigma_{sj} \\
  &= \frac{(-1)^i}{|\Sigma_{\pa(j)}|} \left( (-1)^i |\Sigma_{\pa(j) \setminus i}| \sigma_{ij} + \sum_{s \in \pa(j) \setminus i} (-1)^s |\Sigma_{is|\pa(j) \setminus is}| \sigma_{sj} \right).
\end{align*}
The factors $|\Sigma_{is|\pa(j) \setminus is}|$ appearing in every term of the summation are minors with rows $\pa(j) \setminus s$ and columns $\pa(j) \setminus i$. In particular, they all contain row~$i$ since $s \not= i$. In the next step, we write out the Laplace expansions of all of these minors by row~$i$. To avoid confusion about signs, we use the symbol $i[K]$ to denote the position (starting with $1$) of the element $i$ in the ordered set~$K$. Continuing with the expansion,
\begin{align*}
  \lambda_{ij} &= \begin{multlined}[t][.8\linewidth]
    \frac{(-1)^{i[\pa(j)]}}{|\Sigma_{\pa(j)}|} \Bigg( (-1)^{i[\pa(j)]} |\Sigma_{\pa(j) \setminus i}| \sigma_{ij} \\
    {} + \sum_{s \in \pa(j) \setminus i} (-1)^{s[\pa(j)]} \sigma_{sj} \sum_{t \in \pa(j) \setminus i} (-1)^{i[\pa(j) \setminus s] + t[\pa(j) \setminus i]} |\Sigma_{\pa(j)\setminus is, \pa(j)\setminus it}| \sigma_{it} \Bigg),
  \end{multlined}
\end{align*}
and the double sum simplifies to
\begin{align*}
  \sum_{s,t \in \pa(j) \setminus i} (-1)^{s[\pa(j)] + i[\pa(j) \setminus s] + t[\pa(j) \setminus i]} \sigma_{it} |\Sigma_{st|\pa(j) \setminus ist}| \sigma_{sj}.
\end{align*}
The sign $(-1)^{i[\pa(j) \setminus s]}$ equals $(-1)^{i[\pa(j)]}$ if $s > i$ and $(-1)^{i[\pa(j)]+1}$ if $s < i$. The sign $(-1)^{s[\pa(j)]}$ equals $(-1)^{s[\pa(j) \setminus i]}$ if $s < i$ and $(-1)^{s[\pa(j) \setminus i]+1}$ if $s > i$. In both cases, their product equals $(-1)^{i[\pa(j)] + s[\pa(j) \setminus i] + 1}$. Hence we have
\begin{align*}
  \lambda_{ij} &= \begin{multlined}[t][.8\linewidth]
    \frac{(-1)^{i[\pa(j)]}}{|\Sigma_{\pa(j)}|} \Bigg( (-1)^{i[\pa(j)]} |\Sigma_{\pa(j) \setminus i}| \sigma_{ij} \\
    {} - (-1)^{i[\pa(j)]} \sum_{s,t \in \pa(j) \setminus i} (-1)^{s[\pa(j) \setminus i] + t[\pa(j) \setminus i]} \sigma_{it} |\Sigma_{st|\pa(j) \setminus ist}| \sigma_{sj} \Bigg)
  \end{multlined} \\
  &= \frac1{|\Sigma_{\pa(j)}|} \left( |\Sigma_{\pa(j) \setminus i}| \left( \sigma_{ij} - \Sigma_{i,\pa(j) \setminus i} \Sigma_{\pa(j) \setminus i}^{-1} \Sigma_{\pa(j) \setminus i,j} \right) \right),
\end{align*}
which is the claimed expression by Schur's formula.
\qedhere
\end{paradesc}
\end{proof}

We now turn to \Cref{thm:IdentifyingSets}.
Part~\eqref{thm:IdentifyingSets:Vertices} is proved already in \cite[Theorem~3.3]{wu2023partial}. Part~\eqref{thm:IdentifyingSets:NonEdges} concerns identifying $\lambda_{ij} = 0$ if $ij \not\in E$.  The function $\lambda_{ij|A}(\Sigma)$ vanishes if and only if its numerator $|\Sigma_{ij|A \setminus i}|$ does. This is equivalent to the conditional independence $\CI{i,j|A \setminus i}$ which is in turn characterized by d-separation in~$G$ by the global Markov property. This proves claim \eqref{thm:IdentifyingSets:NonEdges}.
The remainder of this section proves part~\eqref{thm:IdentifyingSets:Edges} where we assume $ij \in E$.

\begin{lemma} \label{lemma:EdgeIdentifyBounds}
Assume $ij \in E$. If $A \in \CC A_G(ij)$, then $i \in A \subseteq V \setminus \ol{\de}(j)$.
\end{lemma}

\begin{proof}
Suppose first that $A$ does not contain~$i$. To show that $A$ cannot be identifying, it suffices to find one point $\Sigma = \phi_G(\Omega, \Lambda) \in \CC M(G)$ such that $\lambda_{ij|A}(\Sigma) \not= \lambda_{ij}$. Since $i \not\in A$, the numerator equals $|\Sigma_{ij|A}|$ and $\lambda_{ij|A}(\Sigma)$ is exactly the Schur complement
\[
  \sigma_{ij} - \Sigma_{i,A} \Sigma_A^{-1} \Sigma_{A,j}.
\]
We set all $\Lambda$-parameters except for $\lambda_{ij}$ to zero. This gives a matrix $\Sigma = \phi_G(\Omega, \lambda_{ij})$. Since $\Sigma_{A,j} = 0$ as a result of there being no treks between $A$ and $j$, the expression for $\lambda_{ij|A}$ simplifies to $\sigma_{ij} = \lambda_{ij} \omega_i$ which differs from $\lambda_{ij}$.

If $j \in \ol{\de}(j)$ is in $A$, then the numerator $|\Sigma_{ij|A\setminus i}|$ has a repeated column and must vanish. However, generically $\lambda_{ij} \not= 0$ on the model.
Finally suppose that $A$ contains a descendant $k \in \ol{\de}(j) \setminus j$. We may pick the smallest such $k$ in $A$. Again, it suffices to find one point on which the identification of $\lambda_{ij}$ fails. Let $p$ be a directed path from $j$ to $k$. By setting all $\Lambda$-parameters outside of $p$ except for $\lambda_{ij}$ to zero (which disconnects all vertices not on this path), we see that it is sufficient to prove the statement for the path $i \to j \to k_1 \to \dots \to k_n = k$ with $\Set{i, k} \subseteq A \subseteq \Set{i, k_1, \dots, k_n}$. Since all $k_i$ are descendants of $j$ and $k$ is the smallest descendant of $j$ in $A$, we have $A = \Set{i,k}$. The statement that $A$ is edge-identifying is equivalent to the vanishing of the polynomial expression in $2 \times 2$ determinants:
\begin{align*}
  |\Sigma_{ij|k}| - \lambda_{ij} |\Sigma_{ik}|
  &= \begin{vmatrix} \sigma_{ij} & \sigma_{ik} \\ \sigma_{jk} & \sigma_{kk} \end{vmatrix} - \lambda_{ij} \begin{vmatrix} \sigma_{ii} & \sigma_{ik} \\ \sigma_{ik} & \sigma_{kk} \end{vmatrix}.
\end{align*}
There is a unique trek between $i$ and every other vertex on the path and its top node is~$i$. Between $j$ and its descendant $k$ there are exactly two treks, one with top~$i$ and one with top~$j$. We~have $\sigma_{jk} = \left(\lambda_{ij} + \frac{\omega_j}{\sigma_{ij}}\right) \sigma_{ik}$. Thus
\begin{align*}
  \begin{vmatrix} \sigma_{ij} & \sigma_{ik} \\ \sigma_{jk} & \sigma_{kk} \end{vmatrix} - \lambda_{ij} \begin{vmatrix} \sigma_{ii} & \sigma_{ik} \\ \sigma_{ik} & \sigma_{kk} \end{vmatrix}
  &= \sigma_{ij} \sigma_{kk} - \sigma_{ik}^2\left(\lambda_{ij} + \frac{\omega_j}{\sigma_{ij}}\right) - \lambda_{ij}\left(\omega_i \sigma_{kk} - \sigma_{ik}^2\right) \\
  &= -\frac{\omega_j}{\omega_i \lambda_{ij}} \sigma_{ik}^2.
\end{align*}
This is the negative product of the trek monomials from $i$ to $k$ and from $j$ to $k$. For generic choices of parameters, neither of them is zero.
\end{proof}

Recall from the statement of \Cref{thm:IdentifyingSets} that for a DAG $G = (V,E)$ with $ij \in E$, we denote by $G_{ij}$ the induced subgraph of $G$ in which the edge $i \to j$ and all descendants of~$j$ are removed.

\begin{lemma}
Assume $ij \in E$. Let $A$ be a d-separating set for $i$ and $j$ in $G_{ij}$. Then $A \cup i \in \CC A_G(ij)$.
\end{lemma}

\begin{proof}
Let $\Sigma = \phi_G(\Omega, \Lambda) \in \CC M(G)$ be arbitrary. The graph $G_{ij}$ arises from $G$ by deleting edge $ij$ and descendants of~$j$. This corresponds to setting some $\Omega$ and $\Lambda$ parameters to zero. Taking the leftover parameters through the parametrization $\phi_{G_{ij}}$ results in a matrix $\Sigma' \in \CC M(G_{ij})$. By~assumption $|\Sigma'_{ij|A}| = 0$. The submatrices $\Sigma_{ij|A}$ and $\Sigma'_{ij|A}$ are very similar:
\begin{itemize}
\item Since $j$ is a sink node in the induced subgraph of $G$ on vertices $ijA$, no descendant of $j$ establishes a new trek among any of the vertices in $ijA$.
\item The presence of the edge $ij$ is the only other difference between $G$ and~$G_{ij}$. Again, since $j$ is a sink in $ijA$, the only entries of $\Sigma_{ij|A}$ that differ from $\Sigma'_{ij|A}$ are: $\sigma_{ij} = \sigma'_{ij} + \lambda_{ij} \sigma'_{ii}$ and $\sigma_{jk} = \sigma'_{jk} + \lambda_{ij} \sigma'_{ik}$, for $k \in A$.
\end{itemize}
This shows that the $j$-column of $\Sigma_{ij|A}$ is just $\Sigma_{iA,j} = \Sigma'_{iA,j} + \lambda_{ij} \Sigma'_{iA,i}$. Hence, by multilinearity of the determinant:
\begin{align*}
  |\Sigma_{ij|A}| = \underbrace{|\Sigma'_{ij|A}|}_{= 0} + \lambda_{ij} |\Sigma_{iA}|,
\end{align*}
which establishes identifiability.
\end{proof}

\begin{lemma}
Assume $ij \in E$. Let $A$ be a d-connecting set for $i$ and $j$ in $G_{ij}$. Then $A \cup i \not\in \CC A_G(ij)$.
\end{lemma}

\begin{proof}
As in the proof of \Cref{lemma:EdgeIdentifyBounds}, it suffices to find a matrix~$\Sigma$ in the model (or any submodel) where $\lambda_{ij|A}(\Sigma) \not= \lambda_{ij}$. For any set $A$ which d-connects $i$ and $j$ in $G_{ij}$, we get a supporting undirected path in $G_{ij}$. We set all edge parameters for the model $\CC M(G)$ outside of this path and except for $\lambda_{ij}$ to zero. This gives a sequence of treks $i \ot \dots \ot t_1 \to \dots \to c_1 \ot \dots \ot t_2 \to \dots \to c_2 \ot \dots \to c_m \ot \dots \ot t_{m+1} \to \dots \to j$ which are glued together at colliders $c_1, \dots, c_m$. The set $A$ consisting of these colliders is the only set which d-connects $i$ and $j$ in this graph. Hence, in this special case --- which covers every d-connection statement in $G_{ij}$ --- it suffices to only treat $A = \Set{c_1, \dots, c_m}$. In the subgraph of $G$ which we consider, only the edge $ij$ is added to this gluing of treks.

Suppose first that $m = 0$, so that we have a simple trek $i \ot \dots \ot t \to \dots \to j$ and $A = \emptyset$. Then $|\Sigma_{ij|A}| = \lambda_{ij} \sigma_{ii} + \omega_t \lambda^\tau = \lambda_{ij} |\Sigma_i| + \omega_t \lambda^\tau$, where $\tau$ is the trek from $i$ to $j$ over~$t$. Generically, $\omega_t \lambda^\tau \not= 0$, so $A = \emptyset$ is not identifying.

In case $m > 0$, we can deduce the following facts from the structure of the graph:
\begin{itemize}[noitemsep, itemsep=0.3em]
\item $\Sigma_A$ is a tridiagonal matrix.
\item In $\Sigma_{i,A}$ the only nonzero entry is $\sigma_{i c_1}$.
\item In $\Sigma_{A,j}$ the only nonzero entries are $\sigma_{c_1 j}$ and $\sigma_{c_m j}$.
\item We have $\sigma_{ij} = \lambda_{ij} \sigma_{ii}$ and $\sigma_{c_1 j} = \lambda_{ij} \sigma_{c_1 i}$.
\end{itemize}
Using multilinearity of the determinant in the $j$-column of $|\Sigma_{ij|A}|$ shows
\begin{align*}
  |\Sigma_{ij|A}| &= \lambda_{ij} |\Sigma_{iA}| + \begin{vNiceMatrix}
  0 & \sigma_{i c_1} & 0 & \Cdots & 0 \\
  \Vdots & \Block{3-4}<\Large>{\Sigma_A} &&& \\
  0      & &&& \\
  \sigma_{j c_m} & &&&
  \end{vNiceMatrix}.
\end{align*}
It suffices to prove that the second determinant is nonzero. By the Schur complement formula, the determinant equals $-|\Sigma_A| \sigma_{i c_1} \sigma_{c_m j} (\Sigma_A^{-1})_{c_1 c_m}$. The last factor $(\Sigma_A^{-1})_{c_1 c_m}$ is an entry of an inverse tridiagonal matrix which by \cite{InverseTridiagonal} equals $\sfrac1{|\Sigma_A|} \prod_{k=1}^{m-1} \sigma_{c_k c_{k+1}}$. Hence, the defect to identifiability is the negative product of trek monomials $-\sigma_{i c_1} \cdot \prod_{k=1}^{m-1} \sigma_{c_k c_{k+1}} \cdot \sigma_{c_m j}$ which is generically nonzero.
\end{proof}

\begin{proof}[Proof of \Cref{thm: MP}]
    \eqref{thm: MP:3} clearly implies \eqref{thm: MP:2}.
    Suppose \eqref{thm: MP:2} holds.
    Then $\Sigma\in \CC M(G)$.
    Hence, there exist $(\Omega, \Lambda)\in \PosReal^V\times \Real^E$ such that $\Sigma = \phi_G(\Omega, \Lambda)$.
    Since $\Sigma$ satisfies conditions~\eqref{def: local colored MP:2} and~\eqref{def: local colored MP:3} of the local Markov property, by global rational identifiability, we see that $\omega_i = \omega_j$ and $\lambda_{ij} = \lambda_{k\ell}$ whenever $c(i) = c(j)$ and $c(ij) = c(k\ell)$, respectively.
    Since $\Sigma$ satisfies the local Markov property with respect to $G$, it also satisfies the global Markov property with respect to $G$ by \Cref{thm:uncoloredMP}.
    Hence, $\Sigma\in \CC M(G,c)$, meaning that \eqref{thm: MP:2} implies \eqref{thm: MP:1}.

    It remains to see that \eqref{thm: MP:1} implies \eqref{thm: MP:3}.
    Suppose that $\Sigma \in \CC M(G,c)$.
    Since $\Sigma$ satisfies the global Markov property with respect to $G$, \Cref{thm:uncoloredMP} implies that $\Sigma\in \CC M(G)$.
    Global rational identifiability for $\CC M(G)$ implies that there exist unique parameters $(\Omega, \Lambda)\in \PosReal^V \times \Real^E$ such that $\Sigma = \phi_G(\Omega, \Lambda)$.
    Since $\Sigma \in \CC M(G,c)\subseteq \CC M(G)$, we further have $\omega_i = \omega_j$ whenever $c(i) = c(j)$ and $\lambda_{ij} = \lambda_{k\ell}$ whenever $c(ij) = c(k\ell)$.
    In particular, the formulas for identifying $\omega_i$ and $\lambda_{ij}$ will be equal whenever two vertices, or respectively edges, are in the same color class.
    It therefore follows from \Cref{thm:IdentifyingSets} that $\Sigma$ satisfies the global Markov property with respect to $(G,c)$, which completes the proof.
\end{proof}

\section{Additional Details on Ancestral Graphs}
\label{app: ancestral}

In this section, we collect a few additional details on ancestral graphs which are needed to prove \Cref{thm: all the graphs}~(\ref{cor:ancestralideals}).
We first note that the formula
\[
\omega_{ij} = \lambda_{ij | \pa_A(i)\cup \pa_A(j)}(\Sigma) \mbox{ for $i\leftrightarrow j\in B$}
\]
indeed holds for any ancestral graph $A$. (In what follows we elide the subscript $\pa_A$ since the ancestral graph $A = (V, D, B)$ is arbitrary but fixed.)
The details of this argument have essentially already appeared in \cite[Proof~B of Theorem~7.1]{drton2018algebraic}.
Letting $\Sigma = (\sigma_{ij})_{i,j\in V} = \varphi_A(\Lambda, \Omega)$, we have that
\[
\Sigma_{\pa(i), i} = \Sigma_{\pa(i)}\Lambda_{\pa(i),i}
\]
since any trek from $j\in\pa(i)$ to $i$ must end in an edge $k\rightarrow i$ for some $k\in\pa(i)$.
(Note, here that $\pa(i):=\{j\in V : j\rightarrow i\in D\}$.)
Since $\Sigma_{\pa(i),\pa(i)}$ is a principal submatrix of the positive definite matrix $\Sigma$, it is invertible.
Hence, we have
\[
\Lambda_{\pa(i),i} = \Sigma_{\pa(i)}^{-1}\Sigma_{\pa(i), i}.
\]
Moreover, by the trek rule we have that
\[
\sigma_{ij} = \sum_{T\in\mathcal{T}(i,j)}m_T.
\]
In the case that $i \leftrightarrow j\in B$, this edge is itself a trek between $i$ and $j$.
Moreover, by the ancestral property of $A$, any other trek between $i$ and $j$ must be a trek $T\in\mathcal{T}(k,\ell)$ with $k\in\pa(i)$ and $\ell\in\pa(j)$ plus the edges $k\rightarrow i$ and $\ell\rightarrow j$. (Otherwise, a trek starting from $i$ with an edge $i\rightarrow k$ would imply that the trek is a directed path from $i$ to $j$, which in combination with the edge $i\leftrightarrow j$ violates the ancestral property of $A$.)
Hence, the above formula for $\sigma_{ij}$ can be re-written as
\[
\sigma_{ij} = \omega_{ij} + \Lambda_{\pa(i)\cup\pa(j),i}^T \Sigma_{\pa(i)\cup\pa(j)}\Lambda_{\pa(i)\cup\pa(j), j}.
\]
Making the appropriate substitutions according to the formula for $\Lambda_{\pa(i),i}$ above and simplifying we obtain
\[
\omega_{ij} = \sigma_{ij} - \Sigma_{i,\pa(i)\cup\pa(j)} \Sigma_{\pa(i)\cup\pa(j)}^{-1}\Sigma_{\pa(i)\cup\pa(j), j}.
\]
In other words, $\omega_{ij} = \textrm{Cov}[X_i, X_j \mid X_{\pa(i)\cup\pa(j)}]$, or equivalently (according to Schur's formula),
\[
\omega_{ij} = \frac{|\Sigma_{ij\mid \pa(i)\cup\pa(j)}|}{|\Sigma_{\pa(i)\cup\pa(j)}|},
\]
as desired.

In order to apply \Cref{lemma:Saturation} to deduce \Cref{thm: all the graphs}~(\ref{cor:ancestralideals}), we require that the ancestral graph $A$ can be realized as a subgraph of a \emph{complete} ancestral graph, i.e., an ancestral graph which contains one edge between each pair of vertices.
This fact follows from a result due to Richardson and Spirtes \cite{richardson2002ancestral}.
In \cite[Section~5.2]{richardson2002ancestral} they specify a construction for realizing any maximal ancestral graph as a subgraph of some complete ancestral graph.
Note that our work focuses on \emph{directed} ancestral graphs, i.e., ancestral graphs that do not contain any undirected edges.
Hence, for the graphs considered in this paper, Richardson and Spirtes' notion of anterior is the same as ancestral, i.e., $\mathrm{ant}_A(i) = \an_A(i) =\{j\in V : \textrm{there is a directed path from $j$ to $i$ in $A$}\}$.
In general ancestral graphs, there may be undirected edges $i - j$, in which case it is possible that there are nodes with no arrowheads pointing into them. Such nodes are denoted by $\mathrm{un}_A$ and in our case they are simply the source nodes of the graph.
The \emph{associated complete ancestral graph} of $A = (V, D, B)$ is then the graph $\widetilde{A} =(V, \widetilde{U}, \widetilde{D}, \widetilde{B})$ where
\[
\begin{split}
    &i - j \in \widetilde{U} \qquad \textrm{ if $i,j\in \mathrm{un}_A$},\\
    &i\rightarrow j\in \widetilde{D} \qquad \mbox{if $i\in \an_A(j)$},\\
    &i\leftrightarrow j\in \widetilde{B} \qquad \mbox{otherwise}.
\end{split}
\]
Here, $\widetilde{U}$ is the set of undirected edges $i - j$ in $\widetilde{A}$.
In particular, the associated complete ancestral graph of a directed ancestral graph $A$ will contain undirected edges between the source nodes in $A$, and therefore will not be directed in general.
Richardson and Spirtes then show in \cite[Lemma~5.5]{richardson2002ancestral} that (i) $A$ is a subgraph of $\widetilde{A}$ and (iv) $\widetilde{A}$ is an ancestral graph.
Hence, any ancestral graph is a subgraph of some complete ancestral graph, as desired.

The complete ancestral graph $\widetilde{A}$ may have undirected edges, which are not included in our parametrization. We now explain how this more general setting of Richardson and Spirtes connects to directed ancestral graphs and enables an application of our \Cref{lemma:Saturation}. The parametrization of a Gaussian (general) ancestral graph $A = (V,U,B,D)$ is identical to the parametrization given for directed ancestral models with one additional parameter $\omega_{ij}=\omega_{ji}$ added to the matrix $\Omega$ for every undirected edge $i - j$.
For more detail, see \cite[Section~8.1]{richardson2002ancestral}, where it is discussed how these additional parameters for undirected edges are \emph{concentration} parameters $K = (k_{ij})_{i,j\in\mathrm{un}_{A}}$ of the undirected Gaussian graphical model formed on only the undirected edges of $A$.
Specifically, the error covariance matrix $\Omega$ will have a block-diagonal structure when the rows and columns are ordered with all nodes in $\mathrm{un}_A$ coming first, where
\[
\Omega =
\begin{pmatrix}
    K^{-1} & 0 \\
    0 & \Omega'\\
\end{pmatrix},
\]
where $k_{ij} = 0$ if and only if $i-j\notin U$.
In applying \Cref{lemma:Saturation}, all parameters corresponding to edges added in passing from $A$ to $\widetilde{A}$ must be set to zero. For directed and bidirected edges, the corresponding parameter-identifying functions $\lambda_{ij|\pa_A(i)\cup\pa_A(j)}$ and $\lambda_{ij|\pa_A(j)}$ were already discussed above.
Since the undirected subgraph of $\widetilde{A}$ is complete, the matrix $K^{-1} \in \PD_{\mathrm{un}_A}$ is arbitrary and the matrix $K$ of parameters is also arbitrary. Setting all entries of $K^{-1}$ to zero, corresponding to the absence of undirected edges in $A$, is equivalent to setting all entries of $K$ to zero. By construction, parent sets in $\widetilde{A}$ are ancestor sets in $A$. Hence, a conditional independence ideal suitable for applying \Cref{lemma:Saturation} for a directed ancestral graph $A$ is generated by
\begin{itemize}[noitemsep, itemsep=0.3em]
\item $\sigma_{ij} = 0$ for $i \neq j$ source nodes in $A$ (added undirected edges in $\widetilde{A}$),
\item $|\Sigma_{ij|\an_A(j)\setminus i}| = 0$ for $i \in \an_A(j) \setminus \pa_A(j)$ (added directed edges in $\widetilde{A}$),
\item $|\Sigma_{ij|\an_A(i)\cup\an_A(j)}| = 0$ for $i \notin \an_A(j)$ (added bidirected edges in $\widetilde{A}$).
\end{itemize}

\section{Proof of \texorpdfstring{\Cref{thm:Smoothness}}{the smoothness theorem}, \texorpdfstring{\Cref{thm: all the graphs}}{the vanishing ideal theorem} and \texorpdfstring{\Cref{lemma:Saturation}}{the saturation lemma}}
\label{section:Smoothness}

We start by introducing some algebraic terminology needed in the following appendices. Given a multiplicatively closed set $S \subseteq \BB R[x]$ and the natural inclusion $f\colon \BB R[x] \rightarrow S^{-1}\BB R[x]$ of $\BB R[x]$ into its localization at $S$, the \emph{contraction} of an ideal $I \subseteq S^{-1}\BB R[x]$ is the ideal $f^{-1}(I) \subseteq \BB R[x]$. We may take an ideal $I \subseteq \BB R[x]$ and consider the ideal $I' \subseteq S^{-1} \BB R[x]$ generated by $f(I)$. It is easy to show that its contraction back to $\BB R[x]$ yields precisely the saturation ideal~$I:S$.

Given an ideal $I \subseteq \BB R[x]$, we can define the \emph{quotient ring} $\BB R[x]/I$. Its elements are cosets $f+I$, where $f \in \BB R[x]$, and the addition and multiplication in $\BB R[x]/I$ are defined as $(f+I) + (g+I):=(f+g)+I$ and $(f+I) \cdot (g+I):=(f\cdot g)+I$. There is a \emph{canonical projection} $\pi:\BB R[x] \rightarrow \BB R[x]/I$ defined by $\pi(f):= f+I$. An \emph{integral domain} is a commutative ring with identity in which the product of any two nonzero elements is nonzero. An ideal $I \subseteq \BB R[x]$ is prime if and only if the quotient ring $\BB R[x]/I$ is an integral domain. Given two polynomials $f,g \in \BB R[x]$ and an ideal $I \subseteq \BB R[x]$, we say that $f=g$ \emph{modulo} $I$ if $f-g \in I$.

Recall that the Jacobian of $\phi_{G,c}$ has its columns indexed by the $\binom{p+1}{2}$ image coordinates $\sigma_{ij}$ with $i \le j$ and its rows indexed by the $\vc + \ec$ base parameters $\omega_i$ and $\lambda_{ij}$. The fewer color classes there are, the fewer rows the Jacobian has while the number of columns depends only on the number of vertices $p = |V|$ in the DAG.

The Jacobian exhibits a block-triangular structure when the rows and columns are suitably ordered. We now explain this ordering and how rows and columns are grouped to form blocks. The topological ordering on $V$ extends lexicographically from the right to all the base parameters; this orders the rows of the Jacobian. Note that by this ordering, a base vertex $j$ is identified with the pair $(j,j)$ and comes after all base edges $(i,j)$, $i \in \pa(j)$. A base vertex by itself is a \emph{block of type $\omega$}. All base edges (if any) into a particular vertex $j$ are arranged next to each other and they together form a \emph{block of type $\lambda$}. Let $P_1 < \dots < P_m$ be the ordered collection of all blocks in the rows of the Jacobian.

To order and group the columns, we assign to every block $P_t$ a set of columns $\sigma_{ij}$, as follows:
\begin{itemize}
\item If $P_t$ is of type $\omega$ for base vertex~$j$, then we pair it with $Q_t \defas \Set{ (j,j) }$.
\item If $P_t$ is of type $\lambda$ for base edges into vertex~$j$, then we pair it with $Q_t \defas \Set{ (i,j) : i \in \pa(j) }$, ordered lexicographically from the right.
\end{itemize}
All remaining $\sigma$ variables are ordered arbitrarily (they will contribute to the inside of the triangle and not matter in our proof that the Jacobian has full rank). Notice that the sets $Q_1 < \dots < Q_m$ are indeed in order (lexicographically from the right), because the corresponding $P_t$ were in order.

We prove below that the $\bigcup_{t=1}^m P_t \times \bigcup_{t=1}^m Q_t$-submatrix of the Jacobian is block-triangular. The diagonal blocks are then decisive for the rank of this matrix. We say that a diagonal block $P_t \times Q_t$ is of \emph{type $\omega$ or $\lambda$} based on the type of~$P_t$.

\begin{lemma} \label{lemma:Smoothness:Triangular}
The Jacobian of $\phi_{G,c}$ is block-triangular when the rows and columns are ordered as described above.
\end{lemma}

\begin{proof}
The following observations immediately follow from the trek rule \eqref{eq:TrekRule} and the topological ordering of the vertices of~$G$:
\begin{itemize}
\item Treks from $i$ to $j$ only contain vertex parameters $\omega_k$ with $k < j$ and edge parameters $\lambda_{kl}$ with $k < l \le \max\Set{i,j} = j$. Replacing these $\omega_k$ and $\lambda_{kl}$ by their respective base parameters only lowers the indices.

\item This implies that $\partial_{\omega_k} \sigma_{jj} = 0$ for base vertex $k > j$ and $\partial_{\lambda_{kl}} \sigma_{jj} = 0$ for base edge $(k,l) > (j,j)$ in the lexicographic ordering from the right. This is clear if $l > j$ by the previous point. The case $l = j$ and $k > j$ is absurd since always $k < l$. Moreover, $\partial_{\omega_j} \sigma_{jj} = 1$ for base vertices~$j$.

\item Similarly, we have $\partial_{\omega_k} \sigma_{ij} = 0$ for $k \ge j$ and $\partial_{\lambda_{kl}} \sigma_{ij} = 0$ for $l > j$.
\end{itemize}
This shows that the submatrix of the Jacobian with all rows and with columns $Q_1 \cup \dots \cup Q_t$, as defined before this proof, is block-triangular.
\end{proof}

\begin{lemma} \label{lemma:Smoothness:Local}
The diagonal blocks $P_t \times Q_t$ of the Jacobian have full rank at every parameter vector.
\end{lemma}

\begin{proof}
A diagonal block of type $\omega$ is a $1 \times 1$ matrix whose entry is $\partial_{\omega_j} \sigma_{jj} = 1$. For the rest of the proof we concentrate on a diagonal block $J$ of type $\lambda$. The rows and columns of $J$ all correspond to edges going into a unique vertex; call it $j$. We begin by writing down the entries of~$J$. The trek rule \eqref{eq:TrekRule} implies
\begin{align*}
  \sigma_{ij} &= \sum_{l \in \pa(j)} \sigma_{il} \lambda_{lj}, \; \text{whenever $i \not= j$}.
\end{align*}
Since $l < j$, the parameter $\lambda_{kj}$ cannot appear in the trek polynomial for $\sigma_{il}$. Since $\lambda_{kj}$ is a base parameter, it is the smallest in its color class and therefore no $c(kj)$-colored edge occurs in~$\sigma_{il}$. Hence,
\begin{align}
  \label{eq:Jentry}
  \partial_{\lambda_{kj}} \sigma_{ij}
   = \frac1{\partial \lambda_{kj}}\left( \sum_{l \in \pa(j)} \sigma_{il} \lambda_{lj} \right)
   = \sum_{l \in \pa(j)} \sigma_{il} \frac{\partial \lambda_{lj}}{\partial \lambda_{kj}}
   = \sum_{\substack{l \in \pa(j), \\ c(lj) = c(kj)}} \sigma_{il}.
\end{align}
Note that the entries of $J$ are really polynomials in base parameters $\omega_i$ and $\lambda_{ij}$; however, the above formula shows that we can also write the entries more comfortably in terms of $\sigma$'s (which are in turn polynomials in $\omega$'s and $\lambda$'s). It is sensible to treat these $\sigma$'s as variables at this point and plug in arbitrary symmetric matrices instead of just points in the model (i.e., parameter vectors). This perspective is fruitful because: to show that $J$ has full rank, we will identify a collection of maximal square submatrices such that at every positive definite matrix $\Sigma$, at least one of them is invertible.

Let $R$ be the set of $k \in \pa(j)$ such that $(k,j)$ is a base edge. Note that $R$ indexes the rows of $J$ and so is nonempty whenever this lemma is applied. Consider the sets $C_k = \Set{ l \in \pa(j) : c(lj) = c(kj) }$ for $k \in R$ encoding the relevant color classes that appear in $J$ and let $C = \bigtimes_{k \in R} C_k$. Each element $A \in C$ is an $|R|$-tuple picking one representative edge $(A_k, j)$ from each class~$C_k$. Let $J_A$ denote the $|R| \times |R|$-submatrix of $J$ with all rows and with columns $\sigma_{A_k j}$. The $(kj,ij)$-entry of $J_A$ according to \eqref{eq:Jentry} now reads as
\[
  \sum_{l \in C_k} \sigma_{il}.
\]

We claim that for every positive definite $\Sigma$, at least one of the $J_A$, $A \in C$, is invertible. Suppose not, then for every $A \in C$, $|J_A| = 0$. Every entry of $J_A$ is a sum over representatives of the $C_k$, so by multilinearity of the determinant applied to every column, we obtain
\[
  0 = |J_A| = \sum_{B \in C} |\Sigma_{A,B}|,
\]
where $\Sigma_{A,B}$ is the submatrix of $\Sigma$ with rows and columns indexed, in order, by the tuples $A$ and~$B$. The simultaneous singularity of all $J_A$ is therefore equivalent to the matrix $\CC A_C(\Sigma) \defas (|\Sigma_{A,B}|)_{A,B \in C}$ having the all-ones vector in its kernel.

The last argument in this proof relates $\CC A_C(\Sigma)$ to the $r$-th \emph{compound matrix} $\CC A_r(\Sigma) = (|\Sigma_{K,L}|)_{K,L \in \binom{V}{r}}$ of~$\Sigma$, where $r = |R|$. The convention in the formation of these minors is that submatrices of $\Sigma$ indexed by subsets $K$ and $L$ are taken with respect to the natural order on the rows and columns, which is the topological order on~$V$. In this case, it is a well-known fact in matrix analysis that if $\Sigma$ is positive definite, then so is $\CC A_r(\Sigma)$; see for instance \cite[Section~2.3, Problem~12]{MatrixAnalysis}.

On the other hand, elements $A, B \in C$ are $r$-tuples and the submatrix $\Sigma_{A,B}$ is formed with respect to the ordering in the tuples. Denote by $\Sigma'_{A,B}$ the submatrix of $\Sigma$ where the indices in $A$ and $B$ are permuted first to be ordered with respect to~$V$. The determinants $|\Sigma'_{A,B}|$ are entries in $\CC A_r(\Sigma)$. The reordering of the tuples $A$ and $B$ is achieved by applying a permutation matrix $\pi_A$ on the left and another permutation matrix $\pi_B$ on the right of $\Sigma_{A,B}$ and hence we have $|\Sigma_{A,B}| = s_A s_B |\Sigma'_{A,B}|$, where the signs $s_A, s_B \in \Set{\pm1}$ depend only on the row~$A$ and the column~$B$. Hence, there is a $|C| \times |C|$ diagonal matrix $D = \diag(s_A : A \in C)$ such that $D \CC A_C(\Sigma) D$ is a principal submatrix of $\CC A_r(\Sigma)$ which is positive definite. But this implies that $\CC A_C(\Sigma)$ is positive definite, has a trivial kernel and consequently not all of the $J_A$ can be simultaneously singular. Hence $J$ has full rank on every $\Sigma \in \PD^V$ and in particular on every parameter vector.
\end{proof}

\begin{proof}[Proof of \Cref{thm:Smoothness}]
It was already established that $\phi_{G,c}$ is a homeomorphism. By \cite[\S3]{GuilleminPollack}, it then suffices to show that $\phi_{G,c}$ is a proper immersion. The map is proper because the inverse image of every compact subset of $\PD^V$ under $\phi_{G,c}$ is compact, since $\phi_{G,c}^{-1}$ is given by rational functions without poles in~$\PD^V$. To prove that $\phi_{G,c}$ is an immersion, we have to show that its Jacobian has full rank at every parameter vector. This is accomplished by a series of lemmas in \Cref{section:Smoothness}: \Cref{lemma:Smoothness:Triangular} shows that the Jacobian has a block-triangular structure, and \Cref{lemma:Smoothness:Local} shows that the diagonal blocks all have full rank.
Together, these statements imply that the Jacobian has full rank at every parameter vector.
The statements about dimension and topology of $\CC M(G,c)$ follow from it being diffeomorphic to its parameter space $\PosReal^{\vc} \times \BB R^{\ec}$.
\end{proof}

\begin{proof}[Proof of \Cref{thm: all the graphs}]
We prove the result for \eqref{cor:coloredDAGideals}, noting that the proofs for \eqref{cor:RCONideals} and \eqref{cor:ancestralideals} follow the same template.
In the notation of \Cref{lemma:Saturation}, pick $R = \BB R[\Sigma]$, $R' = \BB R[\Omega, \Lambda]$, $\phi^*$ and $\psi^*$ the trek rule and identification map for the complete DAG, $S = S_V$
and $I' = \ideal{\lambda_{ij} : ij \not\in E} + \ideal{\omega_i - \omega_j : c(i) = c(j)} + \ideal{\lambda_{ij} - \lambda_{kl} : c(ij) = c(kl)}$. The ideal $I'$ is linear and hence prime. This yields $I = {\phi^*}^{-1}(I') = P_{G,c}$ and $J = I_{G,c}$ by the definitions of the ideals on the right-hand sides. From these identities, the relations $J \subseteq I$ and $I \cap S = \emptyset$ are clear. \Cref{lemma:Saturation} then proves the desired result $P_{G,c} = I = J:S = I_{G,c} : S_V$.
\end{proof}

\begin{proof}[Proof of \Cref{lemma:Saturation}]
By assumption $J \subseteq I$ and hence $J : S \subseteq I : S$ but $I$ is prime (as a contraction of the prime ideal $I'$) and disjoint from $S$, so $I:S = I$ which shows the first inclusion $J:S \subseteq I$.

For the reverse inclusion, let $f \in I$ be arbitrary. Since $\phi^*(f) \in I'$, there is a representation $\phi^*(f) = \sum_i k_i' f_i'$, with $k_i' \in R'$, and thus $f = \psi^*(\phi^*(f)) = \sum_i \psi^*(k_i') \sfrac{g_i}{h_i}$. There exists $h' \in S$ which clears all the denominators on the right-hand side and yields $fh' = \sum_i k_i g_i$ with $k_i \in R$, so $fh' \in J$ and $f \in J:S$.
\end{proof}

\section{Proof of the uncolored version of \texorpdfstring{\Cref{thm:Saturation}}{the vanishing ideal theorem}}
\label{section:UncoloredSaturation}

The objective of this section is to give a characterization of the vanishing ideal of $\CC M(G)$ for any (uncolored) DAG~$G$. The method is due to Roozbehani and Polyanskiy \cite{SaturationProof} and succeeds even if their original proof is not completely rigorous. The proof presented below patches the problems in \cite{SaturationProof} and is explicitly written to require saturation only at the parental principal minors.

\begin{theorem} \label{thm:UncoloredSaturation}
Let $G$ be a DAG. The vanishing ideal $P_G$ of the (uncolored) Gaussian DAG model $\CC M(G)$ is the saturation $I_G : S_G$ of the directed local Markov property ideal at the parental principal minors.
\end{theorem}

\begin{proof}
Let $\BBm1_V$ denote the $V \times V$ identity matrix. It is in the graphical model $\CC M(G)$ and satisfies $s(\BBm1_V) = 1$ for every $s \in S_G$. This shows that $P_G \cap S_G = \emptyset$ and then, since $P_G$ is prime, $P_G : S_G = P_G$. It follows that $I_G : S_G \subseteq P_G$. We get equality by showing that $I_G : S_G$ is prime of the same dimension as $P_G$. The primality is proved in \Cref{lemma:UncoloredPrime} and the dimension in \Cref{lemma:UncoloredDimension}.
\end{proof}

\begin{lemma} \label{lemma:UncoloredSub}
For every finite set $f_1, \dots, f_r \in \BB R[\Sigma]$ there exist $s \in S_G$ and $h_1, \dots, h_r \in \BB R[\Sigma_E]$ such that $sf_i - h_i \in I_G$ for $i = 1, \dots, r$.
\end{lemma}

\begin{proof}
Every generator of $I_G$ is a CI~polynomial of the form $|\Sigma_{ij|\pa(j)}|$ for $ij \not\in E$ with $i<j$. Using Schur complement expansion, this determinant rewrites to
\begin{equation*}
  \label{eq:SchurCompl} \tag{$*$}
  |\Sigma_{ij|\pa(j)}| = |\Sigma_{\pa(j)}| \sigma_{ij} - \Sigma_{i,\pa(j)} \adj(\Sigma_{\pa(j)}) \, \Sigma_{\pa(j),j}.
\end{equation*}
Hence the equality $|\Sigma_{\pa(j)}| \sigma_{ij} = \Sigma_{i,\pa(j)} \adj(\Sigma_{\pa(j)}) \, \Sigma_{\pa(j),j}$ holds modulo $I_G$ and its right-hand side is a polynomial all of whose nonedge variables are of the forms $ik$ or $kl$ with $k,l \in \pa(j)$. These nonedge variables are all less than $ij$ in the lexicographic order from the right.

Now let $f_1, \dots, f_r \in \BB R[\Sigma]$ be arbitrary and let $\sigma_{ij}$ be the largest nonedge variable appearing in any of the~$f_i$. Let $r$ be the largest exponent with which $\sigma_{ij}$ appears. Repeated use of \eqref{eq:SchurCompl} removes all occurences of $\sigma_{ij}$ from $|\Sigma_{\pa(j)}|^r f_i$ resulting in a polynomial $f'_i$ which is equivalent to $|\Sigma_{\pa(j)}|^r f_i$ modulo~$I_G$ but in which only nonedge variables strictly below $\sigma_{ij}$ occur. Proceeding recursively with the $f'_i$ proves the claim.
\end{proof}

\begin{lemma} \label{lemma:UncoloredPrime}
The saturation ideal $I_G : S_G$ is prime.
\end{lemma}

\begin{proof}
By the prime ideal correspondence for localizations \cite[Theorem~6.5]{Kemper}, $I_G : S_G = \BB R[\Sigma] \cap S_G^{-1} I_G$ is prime if and only if $S_G^{-1} I_G$ is prime as an ideal in the localized ring $S_G^{-1} \BB R[\Sigma]$. To~prove this, we show that $(S_G^{-1} \BB R[\Sigma]) / (S_G^{-1} I_G)$ is isomorphic to an integral domain which is explicitly constructed below as a localization of $\BB R[\Sigma_E]$.

Consider the generator $g_{ij} = |\Sigma_{\pa(j)}| \sigma_{ij} - \Sigma_{i,\pa(j)} \adj(\Sigma_{\pa(j)}) \, \Sigma_{\pa(j),j}$ of~$I_G$. Applying \Cref{lemma:UncoloredSub} to the two polynomials $|\Sigma_{\pa(j)}|$ and $\Sigma_{i,\pa(j)} \adj(\Sigma_{\pa(j)}) \, \Sigma_{\pa(j),j}$ yields $s \in S_G$ and $u_{ij}, h_{ij} \in \BB R[\Sigma_E]$ such that $s |\Sigma_{\pa(j)}| = u_{ij}$ and $s (\Sigma_{i,\pa(j)} \adj(\Sigma_{\pa(j)}) \, \Sigma_{\pa(j),j}) = h_{ij}$ modulo~$I_G$.  Then $s g_{ij} = u_{ij} \sigma_{ij} - h_{ij}$ modulo~$I_G$. Because the $g_{ij}$ generate $I_G$ and $s$ is a unit in $S_G^{-1} \BB R[\Sigma]$, the polynomials $u_{ij} \sigma_{ij} - h_{ij}$ generate $S_G^{-1} I_G$.
The equality $s |\Sigma_{\pa(j)}| = u_{ij}$ holds modulo $I_G$, so it also holds in the ring of functions on $\CC M(G)$, and hence $u_{ij}$ vanishes nowhere on $\CC M(G)$ and in particular $u_{ij} \not= 0$. Thus, as functions on $\CC M(G)$, the nonedge variable $\sigma_{ij}$ equals $\sfrac{h_{ij}}{u_{ij}}$.

Let $U$ denote the multiplicatively closed set generated by the $u_{ij}$ in $\BB R[\Sigma_E]$ and consider the ring homomorphism $\alpha: \BB R[\Sigma] \to U^{-1} \BB R[\Sigma_E]$ which maps
\[
  \sigma_{ij} \mapsto \begin{cases}
    \sigma_{ij}, & \text{$i = j$ or $ij \in E$}, \\
    \sfrac{h_{ij}}{u_{ij}}, & ij \not\in E.
  \end{cases}
\]
This map is constructed so that $\alpha(\sigma_{ij}) = \sigma_{ij}$ as functions on $\CC M(G)$ for every entry of $\Sigma$. It follows that for every $f \in \BB R[\Sigma]$, we have $f = \alpha(f)$ as functions on $\CC M(G)$. In particular, $\alpha(t)$ is nowhere zero on $\CC M(G)$ for every $t \in S_G$.

For the final step, let $\ol{U}$ be the multiplicatively closed set generated by $U$ and $\alpha(S_G)$. We~want to extend $\alpha$ to a map $\ol{\alpha}: S_G^{-1} \BB R[\Sigma] \to \ol{U}^{-1} \BB R[\Sigma_E]$. The construction is summarized in the following diagram:
\begin{center}
\begin{tikzcd}
\BB R[\Sigma] \arrow[r, "\alpha"] \arrow[d, hook] & U^{-1} \BB R[\Sigma_E] \arrow[r, hook] & \ol{U}^{-1} \BB R[\Sigma_E] \\
S_G^{-1} \BB R[\Sigma] \arrow[rru, "\ol{\alpha}", dashed] &  &
\end{tikzcd}
\end{center}
Since $U^{-1} \BB R[\Sigma_E]$ is an integral domain and $0 \not\in \alpha(S_G)$, $U^{-1} \BB R[\Sigma_E]$ is canonically embedded into $\ol{U}^{-1} \BB R[\Sigma_E]$ and we may extend $\alpha$ into the latter localization. The existence (and uniqueness) of the dashed arrow $\ol{\alpha}$ is guaranteed by the universal property of localizations \cite[Proposition~6.3~(e)]{Kemper}: since $\alpha(S_G)$ are units in $\ol{U}^{-1} \BB R[\Sigma_E]$, we may simply set $\ol{\alpha}(\sfrac{f}{t}) \defas \sfrac{\alpha(f)}{\alpha(t)}$.

The codomain of $\ol{\alpha}$ is an integral domain, so $\ker (\ol{\alpha})$ is prime. We claim that this kernel is $S_G^{-1} I_G$. By~construction of $\alpha$, the kernel contains all the generators $u_{ij} \sigma_{ij} - h_{ij}$ of the ideal $S_G^{-1} I_G$ and thus it contains the entire ideal. For the converse, let any $\sfrac{f}{t} \in \ker (\ol{\alpha})$ be given, where $f \in \BB R[\Sigma]$ and $t \in S_G$. By~\Cref{lemma:UncoloredSub} there exist $s \in S_G$ and $h \in \BB R[\Sigma_E]$ such that $sf - h \in I_G$. Thus $\sfrac{f}{t} = \sfrac{h}{st} + S_G^{-1} I_G$ and $0 = \ol{\alpha}(\sfrac{f}{t}) = \sfrac{\alpha(h)}{\alpha(st)}$ shows $\alpha(h) = 0$. But since $h \in \BB R[\Sigma_E]$ and $\alpha|_{\BB R[\Sigma_E]} = \id$, it follows that $h = 0$ and hence $\sfrac{f}{t} \in S_G^{-1} I_G$ as required.
\end{proof}

\begin{lemma} \label{lemma:UncoloredDimension}
The conditional independence ideal $I_G$ of a DAG $G = (V,E)$ has dimension $|V| + |E|$.
\end{lemma}

\begin{proof}
Global rational identifiability is used in \cite[Proposition~2.5]{AlgGeoOfGBN} to show that $\dim(P_G) = |V|+|E|$. By inclusion of ideals, this yields a lower bound $|V|+|E| = \dim(P_G) \le \dim(I_G)$. To prove that this is indeed an equality, we use that $I_G$ lies in a polynomial ring with $|V|+\binom{|V|}{2}$ variables and find a lower bound of $\binom{|V|}{2}-|E|$ on its codimension. By a well-known lemma (see for example \cite[Lemma~10]{DoubleMarkov}) this can be accomplished by finding a point on $V(I_G)$ on which the Jacobian of the $\binom{|V|}{2}-|E|$ generators $|\Sigma_{ij|\pa(j)}|$, $ij \not\in E$, of $I_G$ has full rank.

Using Jacobi's formula, $\partial_{\sigma_{i'j'}} |\Sigma_{ij|\pa(j)}|$ is either zero (if $i'$ and $j'$ do not both belong to the set $\Set{i,j} \cup \pa(j)$) or it is a cofactor of $\Sigma_{ij|\pa(j)}$. Evaluating the partial derivative at $\BBm1_V \in \CC M(G)$ simplifies the computation substantially as all the diagonal cofactors are~$1$ and the nondiagonal ones are~$0$:
\[
  \frac{\partial |\Sigma_{ij|\pa(j)}|}{\partial \sigma_{i'j'}} = \begin{cases}
    |\Sigma_{\pa(j)}| = 1 & \text{if $i = i'$ and $j = j'$}, \\
    0 & \text{otherwise}.
  \end{cases}
\]
This shows that there is a full-rank identity matrix in the Jacobian, implying the desired lower bound on the codimension and hence $\dim(I_G) = |V|+|E|$.
\end{proof}

\begin{proof}[Proof of \Cref{thm:Saturation}]
The proof is done in two stages. First, in \Cref{thm:UncoloredSaturation}, we show that in the uncolored case $I_G : S_G = P_G$. The method is due to Roozbehani--Polyanskiy \cite{SaturationProof}.
Given the uncolored result, we apply \Cref{lemma:Saturation} modulo~$P_G$. The situation is summarized in the following commutative diagram.
\begin{center}
\begin{tikzcd}
I_{G,c}:S_G \arrow[d, equal] \arrow[r, two heads] & J:S \arrow[d, equal] \arrow[r, two heads] & J(S^{-1} R) \arrow[dd, hook] \\
P_{G,c} \arrow[dd, hook] \arrow[r, two heads] & I \arrow[dd, hook] \arrow[rr, two heads, crossing over] & & I' \arrow[dd, hook] \\
& & S^{-1} R \\
\BB R[\Sigma] \arrow[r, two heads, "\pi"] & R = \BB R[\CC M(G)] \arrow[ru, hook] \arrow[rr, "\phi^*"] & & R' = \BB R[\Omega, \Lambda_G] \arrow[lu, "\psi^*"]
\end{tikzcd}
\end{center}

We now set up the application of \Cref{lemma:Saturation} using the same notation as in its statement:
\begin{itemize}[leftmargin=2em, rightmargin=1em]
\item Let $R = \BB R[\CC M(G)] = \BB R[\Sigma] / P_G$ and $R' = \BB R[\Omega, \Lambda_G] = \BB R[\Omega, \Lambda] / \ideal{\lambda_{ij} : ij \not\in E}$ be quotient rings and $\pi: \BB R[\Sigma] \onto R$ the canonical projection. Let $S = \pi(S_G)$ be the image in $R$ of the multiplicatively closed set generated by the parental principal minors. It is naturally multiplicatively closed.
Note that $S_G \cap P_G = \emptyset$ (as witnessed by the identity matrix in $\CC M(G)$), so $S$ does not contain the zero element of $R$ and hence $\pi$ lifts to a well-defined homomorphism $\pi': S_G^{-1} \BB R[\Sigma] \to S^{-1} R$ sending $\sfrac{f}{s} \mapsto \sfrac{\pi(f)}{\pi(s)}$.

\item The trek rule map $\phi_G^*: \BB R[\Sigma] \to R'$ factors through $R = \BB R[\Sigma] / \ker (\phi_G^*)$ by the homomorphism theorem: there exists a map $\phi^*: R \to R'$ such that $\phi_G^* = \phi^* \circ \pi$. Furthermore, the composition $\psi^* = \pi' \circ \psi_G^*$ is a well-defined homomorphism $R' \to S^{-1} R$. Since $\psi_G^* \circ \phi_G^* = \id_{\BB R[\Sigma]}$, we have in particular $\psi^* \circ \phi^* = \id_R$.

\item As~the prime ideal $I'$ we choose the linear ideal generated by the coloring relations on $\Omega$ and $\Lambda_G$ parameters. Then the numerators of these generators under $\psi^*$ generate the ideal~$J = \pi(I_c)$ (by definition of $I_c$) and the denominators belong to~$S = \pi(S_G)$.

\item It remains to compute the ideal $I$ in our setup and check that $J \subseteq I$ and $I \cap S = \emptyset$. Denote by $\phi_{G,c}^*: \BB R[\Sigma] \to \BB R[\Omega_c, \Lambda_{G,c}]$ the pullback of the trek rule parametrization with respect to the colored graph; note that $\Omega_c$ and $\Lambda_{G,c}$ only contain base parameters. The vanishing ideal $P_{G,c}$ is the kernel of $\phi_{G,c}^*$. The pullback factors as $\phi_{G,c}^* = \tau \circ \phi_G^* = \tau \circ \phi^* \circ \pi$, where $\tau: \BB R[\Omega, \Lambda_G] \onto \BB R[\Omega_c, \Lambda_{G,c}]$ is the canonical projection with kernel~$I'$. Hence $P_{G,c} = \ker(\tau \circ \phi^* \circ  \pi) = \pi^{-1}({\phi^*}^{-1}(\ker(\tau))) = \pi^{-1}({\phi^*}^{-1}(I')) = \pi^{-1}(I)$ and $I = \pi(P_{G,c})$.

\item The relation $I_c \subseteq P_{G,c}$ implies $J = \pi(I_c) \subseteq \pi(P_{G,c}) = I$. To show $I \cap S = \emptyset$, we have to show that $\pi(P_{G,c}) \cap \pi(S_G) = \emptyset$. Recall that any element of $\pi(P_{G,c})$ is represented by a sum $f + g$ in the quotient ring~$R$, where $f \in P_{G,c}$ and $g \in P_G$. But since $P_G \subseteq P_{G,c}$, we have $(f+g)(\Sigma) = f(\Sigma) + g(\Sigma) = 0 + 0 = 0$ for every distribution $\Sigma \in \CC M(G,c)$. Similarly, every element in $\pi(S_G)$ is represented as $s + g$ with $s \in S_G$ and $g \in P_G$. It evaluates to $(s+g)(\Sigma) = s(\Sigma) + 0 \not= 0$ for, say, $\Sigma = \BBm1_V \in \CC M(G,c)$. Thus $I \cap S = \emptyset$.
\end{itemize}

Hence, \Cref{lemma:Saturation} applies and yields $I = J:S$, which translates to $\pi(P_{G,c}) = \pi(I_c) : \pi(S_G)$ in $\BB R[\CC M(G)]$. This equality contracts to an equality $P_{G,c} = \pi^{-1}(\pi(I_c) : \pi(S_G))$ in $\BB R[\Sigma]$. It is an easy exercise in commutative algebra to show that $\pi^{-1}(\pi(I_c) : \pi(S_G)) \subseteq (I_c + \ker (\pi)) : S_G = (I_c + P_G) : S_G$. \Cref{thm:UncoloredSaturation} shows $P_G = I_G : S_G$ and hence we have the chain
\begin{align*}
  P_{G,c} &= \pi^{-1}(\pi(I_c) : \pi(S_G)) \\
  &\subseteq (I_c + (I_G : S_G)) : S_G \\
  &\subseteq ((I_c : S_G) + (I_G : S_G)) : S_G \\
  &\subseteq (I_c + I_G) : S_G = I_{G,c} : S_G \subseteq P_{G,c},
\end{align*}
which shows that there is equality throughout.
\end{proof}

\section{Proof of Results\texorpdfstring{ in \Cref{sec: model equivalence}}{ on model equivalence}}
\label{section: faithfulness and MP proofs}

\begin{proof}[Proof of~\Cref{prop: faithful to c}]
If $c(ij) = c(k\ell)$ then since $A$ and $B$ are edge-identifying sets for $ij$ and $k\ell$, respectively, it follows that $\ecr_c(ij, k\ell; A, B)$ evaluates to $0$ on all points $\Sigma$ in the model $\CC M(G,c)$.
Conversely, if $c(ij) \neq c(k\ell)$, consider $\Sigma = \phi_{G,c}(\Omega, \Lambda)$ where all base parameters are distinct (e.g., a generic choice of parameters).
Then $\lambda_{ij} \neq \lambda_{k\ell}$.
Since the base parameters are uniquely recoverable from $\Sigma$ via the formulas $\lambda_{ij} = \lambda_{ij | A}(\Sigma)$ and $\lambda_{k\ell} = \lambda_{k\ell | B}(\Sigma)$, it is clear that $\ecr_c(ij, k\ell; A, B)$ does not evaluate to $0$ on generic $\Sigma$ in the model $\CC M(G,c)$.
\end{proof}

\begin{proof}[Proof of~\Cref{prop:vertex coloring faithful to G}]
Wu and Drton~\cite[Theorem~2.2]{wu2023partial} show that a conditional independence $\CI{i,j|K}$ holds for all points in the model $\CC M(G,c)$ if and only if the d-separation $\Dsep{i,j|K}$ holds in~$G$. Hence if the d-separation $\Dsep{i,j|K}$ does not hold in~$G$, there exists a matrix in $\CC M(G,c)$ which does not satisfy $\CI{i,j|K}$. Since the model $\CC M(G,c)$ is irreducible, the set of covariance matrices which do satisfy $\CI{i,j|K}$ is a proper subvariety of codimension at~least~one. Hence, a dense subset of matrices in $\CC M(G,c)$ does not satisfy $\CI{i,j|K}$. The intersection of finitely many dense subsets, one for each d-connection statement in $G$, is still dense and so the generic matrix in $\CC M(G,c)$ is faithful to~$G$.
\end{proof}

\begin{proof}[Proof of \Cref{prop: edge faithful}]
Sullivant, Talaska and Draisma \cite[Section~3]{TrekSeparation} give a construction and algebraic proof that a generic matrix in the uncolored model $\CC M(G)$ is faithful to~$G$. The proof consists of showing that the polynomial $|\Sigma_{ij|K}|$ corresponding to a d-connection statement $\Dconn{i,j|K}$ in $G$ is not the zero polynomial in $\BB R[\Omega, \Lambda]$. Together with irreducibility of the model, this proves faithfulness similarly to the proof above. The crucial step in \cite[Lemma~3.2]{TrekSeparation} uses the fact that generically all $\Omega$ variables are distinct. This condition is maintained in edge-colored DAGs and the rest of the proof goes through verbatim.
\end{proof}

\begin{proof}[Proof of~\Cref{thm:MarkovEqv}]
By \Cref{prop: edge faithful}, generic $\Sigma$ in $\CC M(G,c)$ are faithful to $G$, and analogously for $(H,c)$.  Hence, we can find $\Sigma\in \CC M(G,c)$ which is faithful to both $G$ and $H$. It follows that all CI relations defining the model correspond to the d-separations in $G$ and the d-separations in $H$. In other words, $G$ and $H$ are Markov equivalent DAGs. By the classical characterization of Markov equivalent DAGs due to Verma and Pearl, we have that $G$ and $H$ have the same skeleton and v-structures.
\end{proof}

\begin{proof}[Proof of \Cref{lem: no_rev_cov}]
Let $\Sigma\in M(G,c)$.
Since $i\rightarrow j$ is the only edge reversed in $G_{i\leftarrow j}$ relative to $G$ then
\[
\pa_G(\ell) = \pa_{G_{i\leftarrow j}}(\ell) \qquad \mbox{ for all $\ell\neq i,j$}.
\]
By assumption, we have that there is an edge $k\ell\in c(ij)$ such that $\ell \neq i, j$.
Since $\pa_G(\ell)$ is an edge-identifying set for $k\ell$, it follows from an application of the inverse map of the trek rule for each graph to $\Sigma$ that
\begin{equation}
    \label{eqn: equalparams}
    \lambda_{k\ell}^G = \lambda_{k\ell}^{G_{i\leftarrow j}}. %
\end{equation}
Since we have assumed that $c_{i\leftarrow j}(ji) = c(ij)$, it must also be that $\lambda_{ji}^{G_{i\leftarrow j}} = \lambda_{ij}^G$.
To see this, we use the observation in~\eqref{eqn: equalparams} and the fact that $c(ij)
= c(k\ell)$ and $c_{i\leftarrow j}(ij) = c_{i\leftarrow j}(k\ell)$: %
\[
\lambda_{ji}^{G_{i\leftarrow j}} = \lambda_{k\ell}^{G_{i\leftarrow j}} = \lambda_{k\ell}^G = \lambda_{ij}^G.
\]

Note next that by expanding as a block matrix and applying the Schur complement we obtain
\[
|\Sigma_{\pa_{G_{i\leftarrow j}}(i)}| = |\Sigma_{\pa_G(i) \cup \{j\}}| = |\Sigma_{\pa_G(i)}|(\sigma_{jj} - \Sigma_{j,\pa_G(i)}\Sigma_{\pa_G(i)}^{-1}\Sigma_{\pa_G(i),j}).
\]
We also observe that $|\Sigma_{ij | \pa_G(i)}| = |\Sigma_{ji | \pa_G(i)}|$ since $\Sigma$ is symmetric.
Since $\lambda_{ij}^{G}  = \lambda_{ji}^{G_{i\leftarrow j}}$ and the parent sets are edge-identifying sets, we then have that
\begin{equation*}
    \begin{split}
        \lambda_{ij}^G &= \lambda_{ji}^{G_{i\leftarrow j}},\\
        \frac{|\Sigma_{ij | \pa_G(j) \setminus i}|}{|\Sigma_{\pa_G(j)}|} &= \frac{|\Sigma_{ji | \pa_{G_{i\leftarrow j}}(i)\setminus j}|}{|\Sigma_{\pa_{G_{i\leftarrow j}}(i)}|},\\
        \frac{|\Sigma_{ij | \pa_G(i)}|}{|\Sigma_{\pa_G(i)\cup i}|} &= \frac{|\Sigma_{ij | \pa_G(i)}|}{|\Sigma_{\pa_G(i)\cup j}|},\\
        |\Sigma_{ij | \pa_G(i)}||\Sigma_{\pa_G(i)\cup j}| &= |\Sigma_{ij | \pa_G(i)}||\Sigma_{\pa_G(i)\cup i}|.
    \end{split}
\end{equation*}
However, the final equation does not hold for every point in the model
since the minors $|\Sigma_{\pa_G(i)\cup j}|$ and $|\Sigma_{\pa_G(i)\cup i}|$ need not be equal for generic parameter values.
This leads to a contradiction. We therefore conclude that $\CC M(G_{i\leftarrow j}, c^\prime) \neq \CC M(G,c)$, which completes the proof.
\end{proof}

\begin{lemma}
    \label{lem:forPratik}
    Suppose that $G$ and $H$ are Markov-equivalent DAGs and that $j$ is a sink node in $G$ that is not a sink node in $H$.
    Suppose also that there exist at least two edges in $H$, including an edge $j\ell$.
    If for all edges $ik$ not equal to the edge $j\ell$
    in $H$ we have $j\in\pa_H(k)$ then $j$ is a source node in a complete connected component of $H$ and all other connected components of $H$ are isolated vertices.
\end{lemma}

\begin{proof}
    We note first that $j$ must be a source node in $H$.
    This is due to the assumption that every edge $i \rightarrow k$ in $H$ is such that $j\in \pa_H(k)$, which implies that there can be no edge in $H$ pointing into $j$, as $H$ is a simple, loopless graph.
    Hence, $j$ must be a source node in $H$.

    Now let $i\rightarrow k$ be an edge in $H$ other than $j\rightarrow \ell$.
    By assumption, we have that $j\rightarrow k$ is also an edge in $H$.
    Note that if $i$ and $j$ are not adjacent in $H$ then we would have a v-structure $i\rightarrow k \leftarrow j$ in $H$.
    Since $G$ and $H$ are Markov equivalent, then $G$ and $H$ have the same skeleton.
    Since $j$ is a sink node in $G$, it follows that $k\rightarrow j$ is an edge in $G$.
    Hence, we cannot have the v-structure $i\rightarrow k\leftarrow j$ in $G$, contradicting the assumption that $G$ and $H$ are Markov equivalent.
    Thus, $i$ and $j$ must be adjacent in $H$.
    Moreover, since $j$ is a source node in $H$, we have that $j\rightarrow i$ is an edge in $H$.

    To see that any connected components of $H$ not containing $j$ are isolated vertices, note that if a connected component contains at least two vertices, then it contains an edge between two vertices and by the above argument, must contain $j$.
    So all connected components in $H$ not containing $j$ are isolated vertices.

    Now suppose that $k$ and $i$ are vertices in the connected component of $H$ containing $j$.
    From above, we have that $j\rightarrow k$ and $j\rightarrow i$ are edges in $H$.
    Since $G$ and $H$ are Markov equivalent, they have the same skeleton.
    So since $j$ is a sink node in $G$, we have the edges, $j\leftarrow k$ and $j\leftarrow i$ in $G$.
    However, $G$ and $H$ must have the same v-structures, so it follows that $k$ and $i$ are adjacent in $H$.
    This completes the proof.
\end{proof}

\begin{proof}[Proof of ~\Cref{thm: ident single edge color}]
Since we restrict ourselves to the collection of edge-colored DAGs with only one color class, we denote the coloring of each DAG in $\mathcal{E}_1$ by $c: E \to k$ where $E$ is the edge set of any DAG on $p$ nodes and $k$ is a constant.
    Let $(G,c), (H,c) \in \mathcal{E}_{1}$, suppose that $G \neq H$, and assume for the sake of contradiction that $\CC M(G,c) = \CC M(H,c)$.
    By \Cref{thm:MarkovEqv}, $G$ and $H$ have the same skeleton and v-structures.

\noindent\textbf{The case of a connected graph.}
Suppose first that $G$ is a connected graph containing at least two edges.
We now pick a sink node $j$ in $G$.
If $j$ is also a sink node in $H$, then we marginalize $G$ and $H$ with respect to $j$.
By \cite[Proposition 3.22]{lauritzen1996graphical}, we know that marginalizing sink nodes in a DAG produces a marginal distribution that is Markov to the DAG with the marginalized nodes removed.
Moreover, if $\Sigma$ is faithful to $G$, then the resulting marginal distribution is also faithful to $G\setminus \{j\}$, and $\CC M(G\setminus \{j\},c) = \CC M(H\setminus \{j\},c)$ if $\CC M(G,c) = \CC M(H,c)$.
Thus, from our first assumption that $\CC M(G,c) = \CC M(H,c)$ and $\Sigma$ is faithful to $G$, we can identify a distribution $\Sigma_j\in \CC M(G\setminus \{j\}, c) = \CC M(H\setminus \{j\},c)$ that is faithful to $G\setminus \{j\}$ and $H\setminus \{j\}$.
Therefore, we either have a vertex $j$ which is a sink node in $G$ but not in $H$, or we iteratively marginalize sink nodes from $G$ and $H$ until we reach subgraphs $G'$ and $H'$ where the set of sink nodes of $G'$ and $H'$ are different.
Note that this process terminates with graphs $G'$ and $H'$ each having at least one edge, as otherwise the colored DAGs $(G,c)$ and $(H,c)$ would be identical.

Consider first the situation where the resulting graphs $G'$ and $H'$ following this marginalization process have exactly one edge.
Since $G \neq H$, we have $H'$ is $\ell\leftarrow j$ and $G'$ is $\ell \rightarrow j$ for a pair of nodes $\ell$ and $j$.
Since $G$ and $H$ each had at least two edges, they must have had at least one sink in common, and in fact the graphs had $\pa_G(k) = \pa_H(k)$ for all $k$ other than $k = \ell, j$.
(Here, we are using the fact that $G$ and $H$ are Markov equivalent DAGs.)
Moreover, we have in $G$ that $\pa_{G}(j) = \{\ell\}$ and $\pa_{G}(\ell) = \emptyset$.
Similarly, in $H$, we have that $\pa_{H}(\ell) = \{j\}$ and $\pa_{H}(j) = \emptyset$.
In other words, $G$ and $H$ differ by a single covered edge reversal, where the edge is $\ell \leftarrow j$ in $H$.
Hence, we are in the situation of \Cref{lem: no_rev_cov}, and we conclude that $\CC M(G,c) \neq \CC M(H,c)$.

In the remainder of the proof, we assume that the DAGs $G'$ and $H'$ resulting from the iterative marignalization process above each have at least two edges.
For simplicity we set $G:= G'$ and $H:=H'$.
For such $G$ and $H$, we may always pick a vertex $j$ in $G$ which is a sink in $G$ but not a sink in $H$.
In this case, we know that there exists an edge $\ell\rightarrow j$ in $G$ which is reversed in $H$.
As every edge in $H$ has the same color, we know that the polynomial
$\ecr_c(j\ell,ik; \pa_H(\ell), \pa_H(k))$
lies in the vanishing ideal of $(H,c)$ for any edge $i\rightarrow k$ in $H$.
We then consider two cases:
\begin{enumerate}[label=(\roman*)]
    \item \label{thm: ident single edge color:1} $G$ and $H$ have at least two edges and there exists an edge $i\rightarrow k$ in $H$ other than specified edge $j \rightarrow \ell$ such that $j\notin \pa_H(k)$, and
    \item \label{thm: ident single edge color:2} $G$ and $H$ have at least two edges and all edges $i\rightarrow k$ in $H$ other than the specified edge $j \rightarrow \ell$ satisfy $j\in \pa_H(k)$.
\end{enumerate}

For case~\ref{thm: ident single edge color:1}, if we select an edge $i\rightarrow k$ such that $j\notin \pa_H(k)$, we get the following polynomial lying in the global colored conditional independence ideal of $(H,c)$ (see \Cref{subsec:conj}), and hence in the vanishing ideal of $(H,c)$:
\begin{eqnarray}
    \ecr_c(j\ell,ik; \pa_H(\ell), \pa_H(k)) = |\Sigma_{j\ell|\pa_H(\ell)\setminus j}||\Sigma_{\pa_H(k)}|-|\Sigma_{\pa_H(\ell)}||\Sigma_{ik|\pa_H(k)\setminus i}|.
\end{eqnarray}
As $j\in \pa_H(\ell)$, it is clear that $\sigma_{jj}$ divides at least one of the terms of the above polynomial. Expanding the determinant $|\Sigma_{\pa_H(\ell)}|$ in the polynomial using the Laplace expansion along the row containing the $(j,j)$ entry, we get
\begin{eqnarray}\label{eqnarray:edge identifier}
   \ecr_c(j\ell,ik; \pa_H(\ell), \pa_H(k))= \sigma_{jj}|\Sigma_{\pa_H(\ell)\setminus j}||\Sigma_{ik|\pa_H(k)\setminus i}|+F(\Sigma\setminus \sigma_{jj}),
\end{eqnarray}
where $F(\Sigma\setminus \sigma_{jj})$ is a polynomial in the ring $\BB R[\Sigma\setminus\sigma_{jj}]$.
Using this expansion, we first show that there cannot exist a minimal generating set for $\ker (\phi^*_{H,c})$ which does not involve $\sigma_{jj}$. After this we will show that there cannot exist any irreducible polynomial in the generating set of $\ker (\phi^*_{G,c})$ where one of the terms is divisible by $\sigma_{jj}$, contradicting the original assumption that $\CC M(G,c)= \CC M(H,c)$. Suppose there exists a generating set $\{f_1,f_2,\ldots,f_m\}$ for $\ker (\phi^*_{H,c})$ which completely lies in $\BB R[\Sigma\setminus \sigma_{jj}]$. So, we have
\[
\ecr_c(j\ell,ik; \pa_H(\ell), \pa_H(k))=h_1f_1+h_2f_2+\ldots h_mf_m,
\]
for some some polynomials $h_t\in \BB R[\Sigma]$.
As $\sigma_{jj}$ appears in $\ecr_c(j\ell,ik; \pa_H(\ell), \pa_H(k))$, we further collect the monomials of $h_t$ which involve $\sigma_{jj}$, i.e.,
\[
h_t=\sigma_{jj}h_t'+ h_t'',
\]
for all $t=1,2,\ldots, m$, where $h_t'$ is the sum of terms in $h_t$ divisible by $\sigma_{jj}$ and $h_t''$ is the sum of terms in $h_t$ not divisible by $\sigma_{jj}$
(note that $h_t'$ and $h_t''$ could also be zero for some values of $t$).
This gives us
\begin{eqnarray*}
\sigma_{jj}|\Sigma_{\pa_H(\ell)\setminus j}||\Sigma_{ik|\pa_H(k)\setminus i}|+F(\Sigma\setminus \sigma_{jj})&=&
\sum_{t=1}^m h_t f_t \\
&=&\sigma_{jj}\sum_{t=1}^m h_t'f_t + \sum_{t=1}^m h_t''f_t.
\end{eqnarray*}
This means that $|\Sigma_{\pa_H(\ell)\setminus j}||\Sigma_{ik|\pa_H(k)\setminus i}|$ has to be equal to $h_1'f_1+\ldots h_m'f_m$, implying that $|\Sigma_{\pa_H(\ell)\setminus j}||\Sigma_{ik|\pa_H(k)\setminus i}|$ lies in $\ker (\phi^*_{H,c})$.

Note, however, that there exist distributions $\Sigma\in \CC M(H,c)$ that are faithful to $H$.
For such distributions, the minor $|\Sigma_{\pa_H(\ell)\setminus j}|$ cannot vanish since $\Sigma$ is positive definite, and the minor $|\Sigma_{ik|\pa_H(k)\setminus i}|$ cannot vanish since since $i\rightarrow k$ is an edge in $H$, and hence there is no CI relation $\CI{X_i,X_k|X_C}$ for any set $C$ satisfied by the faithful distribution $\Sigma$.
Since neither of these minors vanish on faithful $\Sigma$, the polynomial $|\Sigma_{\pa_H(\ell)\setminus j}||\Sigma_{ik|\pa_H(k)\setminus i}|$ cannot be in $\ker (\phi^*_{H,c})$, contrary to the above conclusion.

Thus, in order to complete the proof, we show that there cannot exist any irreducible polynomial in the generating set of vanishing ideal of $(G,c)$ where one of the terms is divisible by $\sigma_{jj}$, contradicting the original assumption that $\CC M(G,c)= \CC M(H,c)$.
To prove this, we first look at the image of $\sigma_{jj}$ under the colored trek rule of $(G,c)$.
We have
\[
\phi^*_{G,c}(\sigma_{jj})=\sum_{p\in \pa_G(j)}\sum_{q\in \pa_G(j)}(\lambda^G)^2\sigma_{pq}+\omega^G_j.
\]
As $j$ is a sink node of $G$, it is clear that $\omega^G_j$ does not appear in the image of any other $\sigma_{ab}$.
Now, let $f=f_1+f_2+\ldots +f_m$ be an irreducible polynomial in a generating set of $\ker (\phi^*_{G,c})$, where each $f_i$ is a monomial.
If $\sigma_{jj}$ is a factor of each $f_i$, then it would contradict the fact that $f$ is irreducible.
So, let us assume that $\sigma_{jj}$ is a factor of $f_1,f_2,\ldots,f_k$ but not a factor of $f_{k+1},\ldots, f_m$.
This gives us
\[
    f=\sigma_{jj}(g_1+g_2+\ldots +g_k)+f_{k+1}+\ldots +f_m,
\]
where $\sigma_{jj}g_i=f_i$ for $i=1,2,\ldots,k$.
As $\phi^*_{G,c}$ is a homomorphism and $\phi^*_{G,c}(f)$ is $0$, we have
\[
\phi^*_{G,c}(\sigma_{jj})(\phi^*_{G,c}(g_1)+\ldots + \phi^*_{G,c}(g_k))=-(\phi^*_{G,c}(f_{k+1})+\ldots + \phi^*_{G,c}(f_m)).
\]
As the right hand side of the equality does not have $\omega^G_j$, we can conclude that $\phi^*_{G,c}(g_1+\ldots +g_k)= 0$ (as that is the polynomial coefficient of $\omega^G_j$ on the left) and consequently $\phi^*_{G,c}(f_{k+1}+\ldots +f_m)=0$. This allows us to replace $f$ in the generating set with $\overline{g}=g_1+\ldots + g_k$ and $\overline{f}=f_{k+1}+\ldots +f_k$, where $\overline{f}$ does not involve $\sigma_{jj}$. If $\overline{g}$ again involves $\sigma_{jj}$, we continue this process recursively by replacing $f$ with $\overline{g}$ until we have obtained two generators which do not involve $\sigma_{jj}$.
Thus, we can construct a generating set for $\ker (\phi^*_{G,c})$ where none of the generators involve the variable $\sigma_{jj}$.
Since we have observed that $\ker(\phi^\ast_{H,c})$ necessarily contains generators involving $\sigma_{jj}$, we conclude that $\ker(\phi^\ast_{G,c}) \neq \ker(\phi^\ast_{H,c})$, which implies $\CC M(G,c) \neq \CC M(H,c)$.
This completes the proof in this~case.

For case~\ref{thm: ident single edge color:2}, we are in the situation of \Cref{lem:forPratik}.
Namely, by \Cref{lem:forPratik}, we have that $H$ is composed of a collection of isolated vertices together with a complete connected DAG containing at least two edges where $j$ is a source node.
    Without loss of generality, we assume $H$ is simply the connected component containing $j$.
    Since $H$ is then a complete DAG, it has a single topological ordering $\pi_1\cdots\pi_p$.
    Since $j$ is the source node in $H$, we have that $\pi_1 = j$.
    Let $\pi_2 = \ell$ and $\pi_3 = k$.
    We then have that $\pa_H(\ell) = \{j\}$ and $\pa_H(k) = \{\ell,j\}$.
    Since $c$ is the constant edge coloring, we have that $j\ell$ and $jk$ have the same edge parameter.
    It then follows from \Cref{thm:IdentifyingSets} that
    \[
    \lambda_{j\ell| \pa_H(\ell)}(\Sigma) = \lambda^H = \lambda_{jk | \pa_H(k)}(\Sigma).
    \]
    Therefore,
    \[
    \ecr_c(j\ell,jk; \pa_H(\ell), \pa_H(k)) =\sigma_{\ell j}(\sigma_{\ell\ell}\sigma_{jj} - \sigma_{\ell j}^2) - \sigma_{jj}(\sigma_{jk}\sigma_{\ell \ell} - \sigma_{k\ell}\sigma_{j\ell})
    \]
    is a polynomial in $\ker(\phi^\ast_{H,c})$.
    Equivalently,
    \[
    \sigma_{jj}(\sigma_{j\ell }\sigma_{\ell \ell} - \sigma_{jk}\sigma_{\ell\ell} + \sigma_{k\ell}\sigma_{j\ell}) - \sigma_{j\ell}^3
    \]
    is in $\ker(\phi^\ast_{H,c})$. However, from the initial assumption that $\CC M(G,c)= \CC M(H,c)$, the above polynomial must also lie in $\ker (\phi^*_{G,c})$. Following the same argument as in Case~\ref{thm: ident single edge color:1}, we can conclude
    that $\sigma_{j\ell}^3$ must be in $\ker(\phi^\ast_{G,c})$, which is not true for generic parameter choices as we do not have $\CI{X_j,X_\ell}$ in this model. This completes the proof in case~\ref{thm: ident single edge color:2} by the same reasoning as in case~\ref{thm: ident single edge color:1}.

\noindent\textbf{The case of multiple connected components.}
    We have observed that, for any connected graph $G$ containing at least two edges with constant edge-coloring $c$, $(G,c)$ is not model equivalent to any other DAG in $\mathcal{E}_1$.
    Consider now $G$ with multiple connected components and constant edge-coloring $c$.
    By the above argument, any $(H,c)$ to which $(G,c)$ is model equivalent must have $G$ and $H$ model equivalent; i.e., $G$ and $H$ must have the same skeleton and v-structures.
    Hence, each connected component of $G$ and $H$ must have the same skeleton and v-structures.
    The above argument shows that the constant edge-coloring $c$ forces any connected components of $G$ containing at least two edges to have the exact same edges in $H$ if $(G,c)$ and $(H,c)$ are model equivalent.
    Hence, we need only consider the connected components of $G$ that contain a single edge $i\rightarrow j$.

    Suppose first that $G$ contains a connected component with at least two edges and a connected component consisting of a single edge $i\rightarrow j$.
    Let $H$ be the graph in which all edges are identical to $G$ with the exception that $i\rightarrow j$ is replaced with $i\leftarrow j$, and suppose that $\CC M(G,c) = \CC M(H, c)$.
    Since $H$ contains the same connected component with two edges as $G$, then it follows from global identifiability of the model parameters that for any $\Sigma \in\CC M(G,c) = \CC M(H,c)$, the single edge parameter $\lambda^G$ used to parametrize $\Sigma$ according to $(G,c)$ is equal to the parameter $\lambda^H$ used to parametrize $\Sigma$ according to $(H,c)$.
    In particular, the marginal distribution with covariance matrix $\Sigma_{ij,ij}$ for the edge $i\rightarrow j$ in $G$ must satisfy
    \[
    \Sigma_{ij,ij} =
    \begin{pmatrix}
        \omega_i & \lambda^G\omega_i\\
        \lambda^G\omega_i & \omega_j + (\lambda^G)^2\omega_i
    \end{pmatrix}
    =
    \begin{pmatrix}
        \omega_i^\prime + (\lambda^G)^2\omega^\prime_j & \lambda^G\omega^\prime_j\\
        \lambda^G\omega_j^\prime & \omega_j^\prime
    \end{pmatrix}
    \]
    Here, the left matrix is the parametrization of the covariance matrix of the marginal distribution for the edge $i\rightarrow j$ in $G$ and the right matrix is the parametrization according to the edge $i\leftarrow j$ in $H$.
    Now choose $\Sigma\in \CC M(G,c)$, satisfying $\lambda^G = 1$ and $\omega_i,\omega_j >0$.
    Since $\CC M(G,c) = \CC M(H,c)$, we should be able to identify $\omega_i^\prime, \omega_j^\prime>0$.
    Note that the off-diagonal entries in the above matrix equality imply that $\omega_j^\prime = \omega_i>0$.
    Combining this observation with the equality of the first entries on the diagonal, we find that $\omega_i^\prime = \omega_i(1 - (\lambda^G)^2) = \omega_i \cdot 0 = 0$.
    Hence, $\Sigma\in \CC M(G,c)$ but $\Sigma\notin \CC M(H,c)$, a contradiction.
    Thus, we conclude that $(G,c)$ is in a model equivalence class of size one whenever it contains a connected component with at least two edges.

    Finally, suppose that $G$ is a DAG with at least two edges and every connected component is a single edge.
    If $H$ is any graph identical to $G$ where at least one edge has the same orientation as $G$,
    it follows by the same argument in the preceding paragraph that $\CC M(G, c) \neq \CC M(H,c)$.
    On the other hand, if $H$ is the graph in which all edges are reversed in $G$ then we no longer necessarily have that $\lambda^G$ is equal to $\lambda^H$ for any given $\Sigma$. Now, let $i\rightarrow j$ and $k\rightarrow l$ be two edges in $G$.
    In particular, $\CC M(G,c) = \CC M(H,c)$ if and only if for any $\Sigma \in \CC M(G,c)$ parameterized by $\lambda^G$ and $\omega_i > 0$ there exists $\lambda^H \in \mathbb{R}$ and $\omega_i^\prime>0$ solving the system of equations induced by the matrix equality on the marginal model $\Sigma_{i,j,k,\ell}$:
    \[
    \begin{pmatrix}
    \omega_i & \lambda^G\omega_i & 0 & 0\\
    \lambda^G\omega_i & \omega_j + (\lambda^G)^2\omega_i & 0 & 0 \\
    0 & 0 & \omega_k & \lambda^G\omega_k \\
    0 & 0 & \lambda^G\omega_k & \omega_\ell + (\lambda^G)^2\omega_k \\
    \end{pmatrix}
    =
    \begin{pmatrix}
        \omega_i^\prime + (\lambda^H)^2\omega^\prime_j & \lambda^H\omega^\prime_j & 0 & 0 \\
        \lambda^H\omega_j^\prime & \omega_j^\prime & 0 & 0 \\
        0 & 0 & \omega_k^\prime + (\lambda^H)^2\omega^\prime_\ell & \lambda^H\omega^\prime_\ell\\
        0 & 0 & \lambda^H\omega_\ell^\prime & \omega_\ell^\prime\\
    \end{pmatrix}.
    \]
    The equations in the upper block imply that
    \[
    \lambda^H = \frac{\lambda^G\omega_i}{\omega_j + (\lambda^G)^2\omega_i},
    \]
    and similarly the equations in the lower block imply
    \[
    \lambda^H = \frac{\lambda^G\omega_k}{\omega_\ell + (\lambda^G)^2\omega_k}.
    \]
    However, if for instance, we choose $\Sigma\in\CC M(G,c)$ with $\lambda^G = \omega_i =\omega_j = 1$ and $\omega_k, \omega_\ell > 0$ arbitrary, it is easy to see that these two rational functions need not be equal.
    Hence, we conclude that $\CC M(G,c) \neq \CC M(H,c)$, completing the proof.
\end{proof}

\begin{proof}[Proof of \Cref{theorem:properlyBlocked}]
First consider the case when $(G, c), (H, c')\in\mathcal{BP}$ and $G = H$.
In this case, we have that the edge-identifying sets $\mathcal{A}(ij)$ associated to the edges of $G$ and $H$, respectively, are equal for all edges $ij$ (as both graphs have the same edge set).
Suppose that $c \neq c'$.
Then there exists a pair of edges $ij$ and $k\ell$ such that $c(ij) = c(k\ell)$ but $c'(ij) \neq c'(k\ell)$.
Let $A\in\mathcal{A}(ij)$ and $B\in\mathcal{A}(k\ell)$.
By \Cref{prop: faithful to c}, there exists some $\Sigma\in \CC M(H,c')$ such that $\ecr_c(ij, k\ell; A, B)$ does not evaluate to $0$ on $\Sigma$.
Hence, $\Sigma$ is not Markov to $(G,c)$; i.e., $\Sigma\notin \CC M(G,c)$.
It follows that $\CC M (G, c) = \CC M(H, c')$ if and only if $c = c'$; i.e., $(G, c) = (H, c')$.

Suppose now that $(G, c), (H, c')\in\mathcal{BP}$ with $G \neq H$ and assume for the sake of contradiction that $\CC M(G, c) = \CC M(H, c')$.
By \Cref{thm:MarkovEqv}, we know that $G$ and $H$ have the same skeleton and v-structures.

We begin by following the same procedure as seen in the proof of \Cref{thm: ident single edge color}. We pick a sink node $j$ in $G$ and marginalize out $j$ in both $G$ and $H$ if $j$ is also a sink node in $H$.
We repeat this process until we have either marginalized out all nodes (in which case it must be that $G = H$, a contradiction) or until we obtain a sink node $j$ in the subgraph of $G$ that is not a sink node in the corresponding subgraph of $H$.
Thus, we end in a situation where we can assume that we can pick a sink node $j$ in $G$ that is not a sink in $H$.
Since $G \neq H$, but $G$ and $H$ have the same skeleton, it follows that there is an edge $\ell\rightarrow j$ in $G$ that is reversed in $H$.
By the definition of BPEC-DAGs, we know that there must exist another edge $i\rightarrow \ell$ in $H$ with $\lambda_{j\ell}^H=\lambda_{i\ell}^H$.
In order to complete the proof, we construct a polynomial in $\ker (\phi^*_{H, c'})$ (in analogy to the polynomial used in the proof of \Cref{thm: ident single edge color}), and then use it to argue that there cannot exist a minimal generating set for $\ker (\phi^*_{H,c'})$ which does not involve $\sigma_{jj}$.
As $\lambda_{j\ell}^H=\lambda_{i\ell}^H$, we know that the following polynomial lies in the vanishing ideal of $\CC M(H,c')$:
\begin{eqnarray*}
    \ecr_{c'}(j\ell,i\ell; \pa_H(\ell),\pa_H(\ell))&=&|\Sigma_{j\ell|\pa_H(\ell)\setminus j}||\Sigma_{\pa_H(\ell)}|-|\Sigma_{\pa_H(\ell)}||\Sigma_{i\ell|\pa_H(\ell)\setminus i}| \\
    &=& |\Sigma_{\pa_H(\ell)}|(|\Sigma_{j\ell|\pa_H(\ell)\setminus j}|-|\Sigma_{i\ell|\pa_H(\ell)\setminus i}|).
\end{eqnarray*}
As $|\Sigma_{\pa_H(\ell)}|$ is a principal minor (and hence cannot vanish on the model), we know that the factor $|\Sigma_{j\ell|\pa_H(\ell)\setminus j}|-|\Sigma_{i\ell|\pa_H(\ell)\setminus i}|$ has to lie in the vanishing ideal.
Observe that as $\sigma_{jj}$ appears only in $|\Sigma_{i\ell|\pa_H(\ell)\setminus i}|$ but not in $|\Sigma_{i\ell|\pa_H(\ell)\setminus j}|$, expanding the determinant $|\Sigma_{i\ell|\pa_H(\ell)\setminus i}|$ with respect to the $(j,j)$ entry gives us
\begin{eqnarray*}
    |\Sigma_{j\ell|\pa_H(\ell)\setminus j}|-|\Sigma_{i\ell|\pa_H(\ell)\setminus i}|= \sigma_{jj}|\Sigma_{i\ell|\pa(\ell)\setminus \{i,j\}}| + F(\Sigma\setminus \sigma_{jj}),
\end{eqnarray*}
where $F(\Sigma\setminus \sigma_{jj})$ is some polynomial in $\BB R[\Sigma \setminus \sigma_{jj}]$.
Now, if there exists a generating set $\{f_1,f_2,\ldots, f_m\}$ for $\ker (\phi^*_{H,c'})$ which lies completely in $\BB R[\Sigma\setminus \sigma_{jj}]$, we would have
\[
\sigma_{jj}|\Sigma_{i\ell|\pa(\ell)\setminus \{i,j\}}| + F(\Sigma\setminus \sigma_{jj})=\sigma_{jj}\sum_{t=1}^m h'_tf_t+\sum_{t=1}^m h''_t f_t,
\]
for some $h'_i$ and $h''_i$ in $\BB R[\Sigma]$.
This means that $|\Sigma_{i\ell|\pa(\ell)\setminus \{i,j\}}|$ is equal to $\sum_{t=1}^m h'_tf_t$, implying that $|\Sigma_{i\ell|\pa(\ell)\setminus \{i,j\}}|$ lies in $\ker (\phi^*_{H,c'})$.
However, since $i \rightarrow l$ is an edge in $H$, there is no CI relation of the form $\CI{X_i,X_\ell|X_C}$ for any set $C$ that is satisfied by any distribution faithful to $H$.
Since, there exist distributions in $\CC M(H, c')$ faithful to $H$ by \Cref{prop: edge faithful}, this brings us to a contradiction of the assumption that there exists a generating set of $\ker (\phi^*_{H,c'})$ which does not involve $\sigma_{jj}$.
Hence, any generating set of $\ker (\phi^*_{H,c})$ must contain a polynomial having some term divisible by $\sigma_{jj}$.

Now, as $j$ is a sink node in $G$, we have already shown in the proof of \Cref{thm: ident single edge color} that there cannot exist any irreducible polynomial in the generating set of $\ker (\phi^*_{G,c})$ where one of the terms is divisible by $\sigma_{jj}$. Thus, we can conclude that $\ker (\phi^*_{G,c})\neq \ker (\phi^*_{H,c'})$, implying that $\CC M(G,c)\neq \CC M(H,c')$.
\end{proof}

\section{Additional Details on BPEC DAGs}
\label{app: BPEC additional details}

BPEC DAGs are examples of the so-called compatibly colored DAGs introduced by \cite{makam2022symmetries}.
A colored DAG $(G,c)$ is called \emph{compatibly colored} if whenever $c(ij) = c(k\ell)$ for edges $ij,k\ell\in E$, we have that $c(j) = c(\ell)$.
In words, a compatibly colored DAG is a colored DAG in which any two edges of the same color point towards nodes of the same color.
Since BPEC DAGs are defined by the property that two edges of the same color must point to the same vertex, then they are trivially compatibly colored.

Compatibly colored DAGs are compatible with standard, efficient methods for computing the maximum likelihood estimate of a (uncolored) Gaussian DAG model.
Specifically, let $\mathbf{x} = (x_{i,j})\in \mathbb{R}^{V\times n}$ be a data matrix in which the columns form a random sample from a joint multivariate Gaussian distribution $P$ on $X = (X_i)_{i\in V}$ with density $f$.
If $P$ belongs to the Gaussian DAG model $\CC M(G) = \varphi_G(\Lambda, \Omega)$, we have that the likelihood function for $(\Lambda, \Omega)$ satisfies
\[
\begin{split}
L(\Lambda, \Omega \mid \mathbf{x})
&= \prod_{i=1}^n f(\mathbf{x}_{:i} \mid \Lambda, \Omega),\\ %
&= \prod_{i=1}^n\prod_{k\in V} f(\mathbf{x}_{k,i} \mid \mathbf{x}_{\pa_G(k), i},\Lambda, \Omega),\\ %
&= \prod_{i=1}^n\prod_{k\in V} f(\mathbf{x}_{k,i} \mid \mathbf{x}_{\pa_G(k), i}, \Lambda_{\pa_G(k), k}, \omega_k),\\ %
&= \prod_{i=1}^n\prod_{k\in V}\frac{1}{\sqrt{2\pi\omega_k}} \exp\left(-\frac{1}{2\omega_k}\left(\mathbf{x}_{k,i} - \Lambda_{\pa_G(k), k}^T\mathbf{x}_{\pa_G(k), i}\right)^2\right), %
\end{split}
\]
where in the last line we use the fact that $X_k\mid X_{\pa_G(k)}=\mathbf{x}_{\pa_G(k),i} \sim \textrm{N}\left(\sum_{j\in\pa_G(k)}\lambda_{jk}x_j, \omega_k\right)$.
This follows from the structural equation $X_k = \sum_{j\in\pa_G(k)}\lambda_{jk}X_j + \varepsilon_k$ with $\varepsilon_k\sim \textrm{N}(0,\omega_k)$.
Flipping the order of the products, we obtain
\[
\begin{split}
   L(\Lambda, \Omega \mid \mathbf{x})
&= \prod_{k\in V}\frac{1}{(\sqrt{2\pi\omega_k})^n}\exp\left(-\frac{1}{2\omega_k}\left|\left|\mathbf{x}_{k:} - \Lambda_{\pa_G(k), k}^T\mathbf{x}_{\pa_G(k):}\right|\right|^2\right),
\end{split}
\]
which reveals the classic approach for computing the MLE of the (uncolored) Gaussian DAG models parameters, specifically solving $|V|$ least squares problems to recover $\hat\omega_k$ and $\hat{\lambda}_{jk}$ for all $j\in\pa_G(k)$, for all $k \in V$.

Note now that if we restrict to finding the MLE over the colored DAG model $\CC M(G,c)$ we need to adjust the above likelihood function to account for the parameter homogeneities $\omega_k = \omega_\ell$ whenever $c(k) = c(\ell)$ and $\lambda_{jk} = \lambda_{t\ell}$ whenever $c(jk) = c(t\ell)$.
In the case that $c$ is a compatible coloring, we will always have $\omega_k = \omega_\ell$ whenever $\lambda_{jk} = \lambda_{t\ell}$.  This is significant due to the factor of $\frac{1}{\omega_k}$ in front of the summand in each exponential in the above product.
Namely, the compatible coloring means that we can group factors in the above likelihood function and still formulate the MLE problem as the solution to a collection of disjoint least squares problems.
Specifically, when $c$ is a compatible coloring we have
\[
\begin{split}
   L(\Lambda, \Omega \mid \mathbf{x})
&= \prod_{c \in c(V)}\frac{1}{(\sqrt{2\pi\omega_c})^{n|c^{-1}(c)|}}\exp\left(-\frac{1}{2\omega_c}\sum_{k\in c^{-1}(c)}\left|\left|\mathbf{x}_{k:} - \Lambda_{\pa_G(k), k}^T\mathbf{x}_{\pa_G(k):}\right|\right|^2\right).
\end{split}
\]
By taking the logarithm and solving one least squares problem for each vertex color $c\in c(V)$ we obtain the MLE of the model parameters.
Note specifically, that we would not be able to solve such a collection of least squares problem if some $jk$ was colored the same as some $t\ell$ where $k$ and $\ell$ were colored differently.
In this way, compatible colorings are exactly the colorings that are compatible with classic solutions to the likelihood equations for Gaussian DAG models.

Recall that the Bayesian Information Criterion (BIC) of a statistical model $M$ with parameters $\theta$ for the random sample $\mathbf{x}$ is
\[
\textrm{BIC}(M; \theta, \mathbf{x}) = k\ln(n) - 2\log L(\hat\theta \mid \mathbf{x}), %
\]
where $k$ denotes the number of free parameters (typically taken as $\theta = (\theta_1,\ldots, \theta_d)$) and $\hat\theta$ denotes the maximum likelihood estimate of the model parameters given the data $\mathbf{x}$.
In graphical models, it is typical to work with the associated penalized maximum likelihood function, which is a scalar multiple of the BIC:
\[
\textrm{score}(M; \theta, \mathbf{x}) = -\frac{1}{2}\textrm{BIC}(M; \theta, \mathbf{x}) = \log L(\hat\theta \mid \mathbf{x}) - \frac{\ln(n)k}{2}.
\]
Hence, the optimal model will \emph{maximize} $\textrm{score}(M;\theta, \mathbf{x})$.

For the colored DAG model $\CC M(G,c)$ we have $|c(V)| + |c(E)|$ free parameters, one for each vertex color and one for each edge color.
When $(G,c)$ is compatibly colored, the MLE $(\hat\Lambda, \hat\Omega)$ is obtained by solving the $|c(V)|$ least squares problems as outlined above.
Specifically, letting $\pa_{(G,c)}(c) = \{ e\in c(E) : e = c(jk) \textrm{ for some } jk\in E \textrm{ with } k \in c^{-1}(c)\}$ for the vertex color $c\in c(V)$, we have
\[
\begin{split}
    (\hat\lambda_{e})_{e \in \pa_{(G,c)}(c)} &= \textrm{argmin}_{\beta\in\mathbb{R}^{|\pa_{(G,c)}(c)|}}\sum_{k\in c^{-1}(c)}\left|\left|\mathbf{x}_{k:} - \Lambda_{\pa_G(k), k}^T\mathbf{x}_{\pa_G(k):}\right|\right|^2, \\
    \hat\omega_c &= \frac{1}{n|c^{-1}(c)|}\sum_{k\in c^{-1}(c)}\left|\left|\mathbf{x}_{k:} - \hat\Lambda_{\pa_G(k), k}^T\mathbf{x}_{\pa_G(k):}\right|\right|^2
\end{split}
\]
for $c\in c(V)$.
In the case of BPEC-DAGs, the vertex coloring is the identity map; i.e., $c(V) = V$.
Hence, this formula reduces to
\[
\begin{split}
    (\hat\lambda_{e})_{e \in \pa_{(G,c)}(k)} &= \textrm{argmin}_{\beta\in\mathbb{R}^{|\pa_{(G,c)}(k)|}}\left|\left|\mathbf{x}_{k:} - \Lambda_{\pa_G(k), k}^T\mathbf{x}_{\pa_G(k):}\right|\right|^2, \\
    \hat\omega_k &= \frac{1}{n}\left|\left|\mathbf{x}_{k:} - \hat\Lambda_{\pa_G(k), k}^T\mathbf{x}_{\pa_G(k):}\right|\right|^2
\end{split}
\]
for $k\in V$.
Thus, given a random sample $\mathbf{x}$ and a BPEC-DAG $(G,c)$ specifying the model parameters $(\Lambda, \Omega)$, the GECS algorithm is computing the score
\[
\textrm{score}(G,c; \mathbb{D}) = \log L(\hat\Lambda, \hat\Omega \mid \mathbb{D}) - \frac{\ln(n)(|c(V)| + |c(E)|)}{2}.
\]

\section{Pseudocode and additional experiments for the GECS algorithm}\label{section:pseudocode}

\begin{algorithm}
  \scriptsize
  \label{alg:ecDAGmodify}
  \raggedright
  \hspace*{\algorithmicindent} \textbf{Input:} A BPEC-DAG $(G,c)$ with $G = (V,E)$.\\
  \hspace*{\algorithmicindent} \textbf{Input:} A random sample $\mathbb{D}$ of size $n$ from a distribution of over $|V|$ variables.\\
  \hspace*{\algorithmicindent} \textbf{Input:} A list of functions $\TT f_1, \dots, \TT f_k$ modifying a BPEC-DAG.\\
  \hspace*{\algorithmicindent} \textbf{Output:} A BPEC-DAG $(G, c)$.
  \begin{algorithmic}[1]
    \Loop
      \State{$\RM{score} \gets \RM{BIC}(G,c; \BB D)$}
      \For{$\TT{f}$ \textbf{in} $\TT f_1, \dots, \TT f_k$} \Comment{Keep modifying the BPEC-DAG}
        \State{$G', c' \gets \TT f(G, c; \BB D)$}
        \If{$\RM{score} < \RM{BIC}(G', c'; \BB D)$}
          \State{$G, c \gets G', c'$}
        \EndIf
      \EndFor
      \If{$\RM{score} = \RM{BIC}(G, c; \BB D)$} \Comment{Stop if local maximum is reached}
        \State \textbf{break}
      \EndIf
    \EndLoop
    \State \Return $(G, c)$
  \end{algorithmic}
  \caption{The $\TT{edDAGmodify}$ routine used to implement the three phases of GECS.}
\end{algorithm}

\begin{algorithm}
  \scriptsize
  \caption{\texttt{addColor}}
  \label{alg:addColor}
  \raggedright
  \hspace*{\algorithmicindent} \textbf{Input:} A BPEC-DAG $(G,c)$ with $G=(V,E)$. \Comment{$c$ is dictionary object with keys being colors}\\
  \hspace*{\algorithmicindent} \textbf{Input:} A random sample $\mathbb{D}$ of size $n$ from a distribution of over $|V|$ variables.\\
  \hspace*{\algorithmicindent} \textbf{Output:} A BPEC-DAG $(G, c)$.
  \begin{algorithmic}[1]
    \State{\texttt{possible-additions} $\gets \emptyset$}
    \For{$i$ \textbf{in} $V$}
        \State{\texttt{non-adjacencies} $\gets \{j\in V : ij\notin E \textrm{ and } ji\notin E$\}}
        \For{ $p_1,p_2$ \textbf{in} \texttt{non-adjacencies}}
            \State{$E' \gets E \cup \{p_1i, p_2i\}$}
            \State{$G' \gets (V, E')$}
            \State{$c' \gets c + \{\textrm{new color} : \{p_1i,p_2i\}\}$}
            \If{$G'$ is acyclic}
                \State{\texttt{possible-additions} $\gets \texttt{possible-additions}\cup\{(G',c')\}$}
            \EndIf
        \EndFor
    \EndFor
    \State{$(G^*, c^*) \gets \arg\max_{(G',c')\in \texttt{possible-additions}}\texttt{score}(G',c'; \BB D)$}
    \If{$\texttt{score}(G',c'; \BB D) > \texttt{score}(G,c; \BB D)$}
        \State{$(G,c) \gets (G',c')$}
    \EndIf
    \State \Return{$(G, c)$}
  \end{algorithmic}
\end{algorithm}

\begin{algorithm}
  \scriptsize
  \caption{\texttt{splitColor}}
  \label{alg:splitColor}
  \raggedright
  \hspace*{\algorithmicindent} \textbf{Input:} A BPEC-DAG $(G,c)$ with $G = (V,E)$. \Comment{$c$ is dictionary object with keys being colors}\\
  \hspace*{\algorithmicindent} \textbf{Input:} A random sample $\mathbb{D}$ of size $n$ from a distribution of over $|V|$ variables.\\
  \hspace*{\algorithmicindent} \textbf{Output:} A BPEC-DAG $(G, c)$.
  \begin{algorithmic}[1]
    \State{$\texttt{possible-splits} \gets \emptyset$}
    \State{$\texttt{splittable-colors} \gets \{\textrm{color} \in c.\textrm{keys} : \textrm{ $c[\textrm{color}]$ has size at least $4$}\}$}
    \For{$\textrm{color}\in\texttt{splittable-colors}$}
        \For{$ij,kj \in c[\textrm{color}]$}
            \State{$E' \gets E$}
            \State{$G'\gets (V,E')$}
            \State{$c' \gets c \textrm{ with $c[\textrm{color}] \gets c[\textrm{color}]\setminus\{ij, kj\}$ and $c[\textrm{new-color}] \gets \{ij,kj\}$}$}
            \State{$\texttt{possible-splits} \gets \texttt{possible-splits}\cup\{(G',c')\}$}
        \EndFor
    \EndFor
    \State{$(G^*, c^*) \gets \arg\max_{(G',c')\in \texttt{possible-splits}}\texttt{score}(G',c'; \BB D)$}
    \If{$\texttt{score}(G',c'; \BB D) > \texttt{score}(G,c; \BB D)$}
        \State{$(G,c) \gets (G',c')$}
    \EndIf
    \State \Return{$(G, c)$}
  \end{algorithmic}
\end{algorithm}

\begin{algorithm}
  \scriptsize
  \caption{\texttt{addEdge}}
  \label{alg:addEdge}
  \raggedright
  \hspace*{\algorithmicindent} \textbf{Input:} A BPEC-DAG $(G,c)$ with $G = (V,E)$. \Comment{$c$ is dictionary object with keys being colors}\\
  \hspace*{\algorithmicindent} \textbf{Input:} A random sample $\mathbb{D}$ of size $n$ from a distribution of over $|V|$ variables.\\
  \hspace*{\algorithmicindent} \textbf{Output:} A BPEC-DAG $(G, c)$.
  \begin{algorithmic}[1]
    \State{\texttt{possible-edge-additions} $\gets \emptyset$}
    \State{$\texttt{nonedges} \gets \{ij: ij\notin E\}$}
    \For{$ij\in \texttt{nonedges}$}
        \State{$\texttt{family-colors} \gets \{\textrm{color}\in \textrm{c.keys} : kj\in c[\textrm{color}] \textrm{ for some $k\in V$}\}$}
        \For{$\textrm{color}\in \texttt{family-colors}$}
            \State{$E' \gets E\cup\{ij\}$}
            \State{$G' \gets (V, E')$}
            \State{$c' \gets c \textrm{ with c[\textrm{color}]+=\{ij\}}$}
            \If{$G'$ is acyclic}
                \State{\texttt{possible-edge-additions} $\gets \texttt{possible-edge-additions}\cup\{(G',c')\}$}
            \EndIf
        \EndFor
    \EndFor
    \State{$(G^*, c^*) \gets \arg\max_{(G',c')\in \texttt{possible-edge-additions}}\texttt{score}(G',c'; \BB D)$}
    \If{$\texttt{score}(G',c'; \BB D) > \texttt{score}(G,c; \BB D)$}
        \State{$(G,c) \gets (G',c')$}
    \EndIf
    \State \Return{$(G, c)$}
  \end{algorithmic}
\end{algorithm}

\begin{algorithm}
  \scriptsize
  \caption{\texttt{moveEdge}}
  \label{alg:moveEdge}
  \raggedright
  \hspace*{\algorithmicindent} \textbf{Input:} A BPEC-DAG $(G,c)$ with $G=(V,E)$. \Comment{$c$ is dictionary object with keys being colors}\\
  \hspace*{\algorithmicindent} \textbf{Input:} A random sample $\mathbb{D}$ of size $n$ from a distribution of over $|V|$ variables.\\
  \hspace*{\algorithmicindent} \textbf{Output:} A BPEC-DAG $(G, c)$.
  \begin{algorithmic}[1]
    \State{\texttt{possible-moves} $\gets \emptyset$}
    \For{$i\in V$}
        \State{$\texttt{family-colors} \gets \{\textrm{color}\in \textrm{c.keys} : ki\in c[\textrm{color}] \textrm{ for some $k\in V$}\}$}
        \For{$\textrm{color1, color2}\in \texttt{family-colors}$ with $|c[\textrm{color1}]| >2$}
            \For{$ki \in c[\textrm{color1}]$}
                \State{$G' \gets (V, E)$}
                \State{$c' \gets c \textrm{ with $c[\textrm{color1}]\setminus\{ki\}$ and $c[\textrm{color2}]+=\{ki\}$}$}
                \State{\texttt{possible-moves} $\gets \texttt{possible-moves}\cup\{(G',c')\}$}
            \EndFor
        \EndFor
    \EndFor
    \State{$(G^*, c^*) \gets \arg\max_{(G',c')\in \texttt{possible-moves}}\texttt{score}(G',c'; \BB D)$}
    \If{$\texttt{score}(G',c'; \BB D) > \texttt{score}(G,c; \BB D)$}
        \State{$(G,c) \gets (G',c')$}
    \EndIf
    \State \Return{$(G, c)$}
  \end{algorithmic}
\end{algorithm}

\begin{algorithm}
  \scriptsize
  \caption{\texttt{reverseEdge}}
  \label{alg:reverseEdge}
  \raggedright
  \hspace*{\algorithmicindent} \textbf{Input:} A BPEC-DAG $(G,c)$ with $G = (V,E)$. \Comment{$c$ is dictionary object with keys being colors}\\
  \hspace*{\algorithmicindent} \textbf{Input:} A random sample $\mathbb{D}$ of size $n$ from a distribution of over $|V|$ variables.\\
  \hspace*{\algorithmicindent} \textbf{Output:} A BPEC-DAG $(G, c)$.
  \begin{algorithmic}[1]
    \State{$\texttt{possible-reverses} \gets \emptyset$}
    \State{$\texttt{reversable-edges} \gets \{ij \in E : \textrm{ color class of $ij$ is size at least $2$ and reversing $ij$ preserves acyclicity}\}$}
    \For{$ij\in\texttt{reversable-edges}$}
        \State{$\texttt{family-colors} \gets \{\textrm{color}\in \textrm{c.keys} : ki\in c[\textrm{color}] \textrm{ for some $k\in V$}\}$}
        \For{$\textrm{color} \in \texttt{family-colors}$}
            \State{$E' \gets E\setminus\{ij\}\cup\{ji\}$}
            \State{$G'\gets (V,E')$}
            \State{$c' \gets c \textrm{ with $c[\textrm{color}] += \{ji\}$ and $ij$ removed from its color class}$}
            \State{$\texttt{possible-reverses} \gets \texttt{possible-reverses}\cup\{(G',c')\}$}
        \EndFor
    \EndFor
    \State{$(G^*, c^*) \gets \arg\max_{(G',c')\in \texttt{possible-reverses}}\texttt{score}(G',c'; \BB D)$}
    \If{$\texttt{score}(G',c'; \BB D) > \texttt{score}(G,c; \BB D)$}
        \State{$(G,c) \gets (G',c')$}
    \EndIf
    \State \Return{$(G, c)$}
  \end{algorithmic}
\end{algorithm}

\begin{algorithm}
  \scriptsize
  \caption{\texttt{removeEdge}}
  \label{alg:removeEdge}
  \raggedright
  \hspace*{\algorithmicindent} \textbf{Input:} A BPEC-DAG $(G,c)$ with $G = (V,E)$. \Comment{$c$ is dictionary object with keys being colors}\\
  \hspace*{\algorithmicindent} \textbf{Input:} A random sample $\mathbb{D}$ of size $n$ from a distribution of over $|V|$ variables.\\
  \hspace*{\algorithmicindent} \textbf{Output:} A BPEC-DAG $(G, c)$.
  \begin{algorithmic}[1]
    \State{$\texttt{possible-removals} \gets \emptyset$}
    \State{$\texttt{removable-edges} \gets \{ij \in E : \textrm{ color class of $ij$ is size at least $3$}\}$}
    \For{$ij\in\texttt{removable-edges}$}
        \State{$E' \gets E\setminus\{ij\}$}
        \State{$G'\gets (V,E')$}
        \State{$\textrm{color} \gets \textrm{color of edge $ij$}$}
        \State{$c' \gets c \textrm{ with $c[\textrm{color}] \gets c[\textrm{color}]\setminus\{ij\}$}$}
        \State{$\texttt{possible-removals} \gets \texttt{possible-removals}\cup\{(G',c')\}$}
    \EndFor
    \State{$(G^*, c^*) \gets \arg\max_{(G',c')\in \texttt{possible-removals}}\texttt{score}(G',c'; \BB D)$}
    \If{$\texttt{score}(G',c'; \BB D) > \texttt{score}(G,c; \BB D)$}
        \State{$(G,c) \gets (G',c')$}
    \EndIf
    \State \Return{$(G, c)$}
  \end{algorithmic}
\end{algorithm}

\begin{algorithm}
  \scriptsize
  \caption{\texttt{mergeColors}}
  \label{alg:mergeColors}
  \raggedright
  \hspace*{\algorithmicindent} \textbf{Input:} A BPEC-DAG $(G,c)$ with $G = (V,E)$. \Comment{$c$ is dictionary object with keys being colors}\\
  \hspace*{\algorithmicindent} \textbf{Input:} A random sample $\mathbb{D}$ of size $n$ from a distribution of over $|V|$ variables.\\
  \hspace*{\algorithmicindent} \textbf{Output:} A BPEC-DAG $(G, c)$.
  \begin{algorithmic}[1]
    \State{$\texttt{possible-merges} \gets \emptyset$}
    \For{$i \in V$}
        \State{$\texttt{family-colors} \gets \{\textrm{color}\in \textrm{c.keys} : ki\in c[\textrm{color}] \textrm{ for some $k\in V$}\}$}
        \If{\texttt{family-colors} has size at least $2$}
            \For{$\textrm{color1, color2}\in \texttt{family-colors}$}
                \State{$G'\gets (V,E)$}
                \State{$c' \gets c \textrm{ with $c[\textrm{color1}] \gets c[\textrm{color1}]\cup c[\textrm{color2}]$ and $c[\textrm{color2}] \gets \emptyset$}$}
                \State{$\texttt{possible-merges} \gets \texttt{possible-merges}\cup\{(G',c')\}$}
            \EndFor
        \EndIf
    \EndFor
    \State{$(G^*, c^*) \gets \arg\max_{(G',c')\in \texttt{possible-merges}}\texttt{score}(G',c'; \BB D)$}
    \If{$\texttt{score}(G',c'; \BB D) > \texttt{score}(G,c; \BB D)$}
        \State{$(G,c) \gets (G',c')$}
    \EndIf
    \State \Return{$(G, c)$}
  \end{algorithmic}
\end{algorithm}

\begin{algorithm}
  \scriptsize
  \caption{\texttt{removeColor}}
  \label{alg:removeColor}
  \raggedright
  \hspace*{\algorithmicindent} \textbf{Input:} A BPEC-DAG $(G,c)$ with $G = (V,E)$. \Comment{$c$ is dictionary object with keys being colors}\\
  \hspace*{\algorithmicindent} \textbf{Input:} A random sample $\mathbb{D}$ of size $n$ from a distribution of over $|V|$ variables.\\
  \hspace*{\algorithmicindent} \textbf{Output:} A BPEC-DAG $(G, c)$.
  \begin{algorithmic}[1]
    \State{$\texttt{possible-removals} \gets \emptyset$}
    \For{$\textrm{color}\in c.\textrm{keys}$}
        \State{$E' \gets E\setminus c[\textrm{color}]$}
        \State{$G'\gets (V,E')$}
        \State{$c' \gets c \textrm{ with $c[\textrm{color}] \gets \emptyset$}$}
        \State{$\texttt{possible-removals} \gets \texttt{possible-removals}\cup\{(G',c')\}$}
    \EndFor
    \State{$(G^*, c^*) \gets \arg\max_{(G',c')\in \texttt{possible-removals}}\texttt{score}(G',c'; \BB D)$}
    \If{$\texttt{score}(G',c'; \BB D) > \texttt{score}(G,c; \BB D)$}
        \State{$(G,c) \gets (G',c')$}
    \EndIf
    \State \Return{$(G, c)$}
  \end{algorithmic}
\end{algorithm}

\clearpage

\subsection{Additional experimental results}
To see the performance on smaller sample size, we ran the same simulations as in~\Cref{subsec:simulations} for $p \in\{6, 10\}$ nodes with $n = 250$ samples,  $\texttt{nc}\in\{2,\ldots, p - 1\}$, and $\rho\in\{0.2,0.3,0.4,0.5,0.6,0.7,0.8,0.9\}$. The SHD and coloring sensitivity results for $p = 6$ are presented in \Cref{fig:p6n250}.
The SHD and coloring sensitivity results for $p = 10$ are presented in \Cref{fig:p10n250} and \Cref{fig:p10n250cTPR}, respectively. \Cref{fig:p10n1000cTPR} gives the coloring sensitivity results for the experiments on $p = 10$ nodes with $n = 1000$ samples.

\begin{figure}[h]
    \begin{subfigure}[b]{0.24\textwidth}
    \centering
    \includegraphics[width=\textwidth]{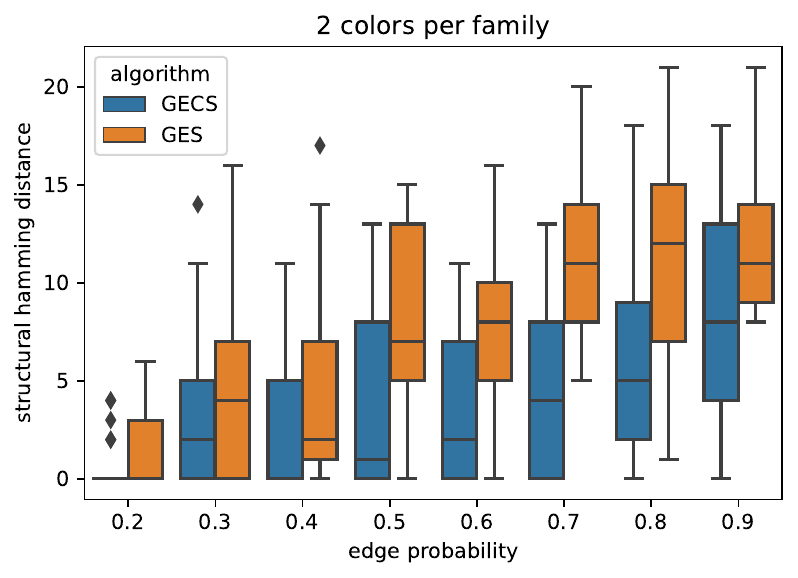}
    \caption{$\texttt{nc} = 2$}
    \label{fig:p6n250c2}
    \end{subfigure}
    \hfill
    \begin{subfigure}[b]{0.24\textwidth}
    \centering
    \includegraphics[width=\textwidth]{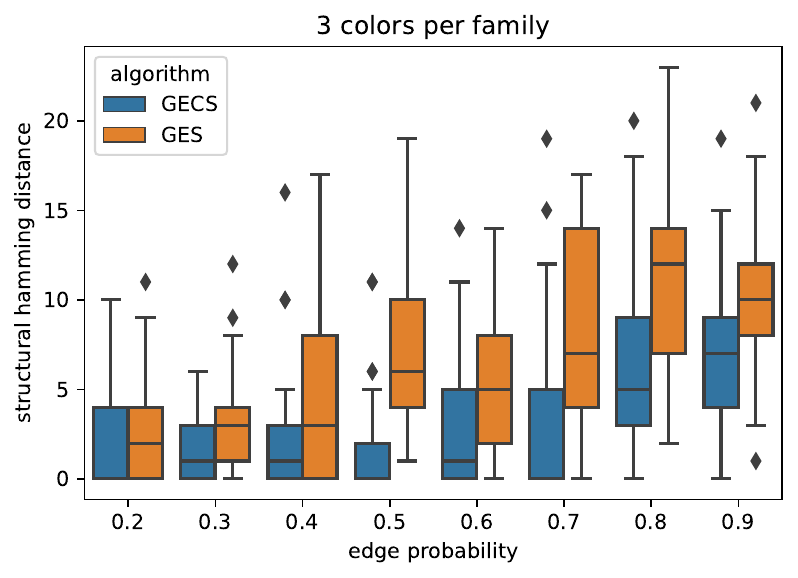}
    \caption{$\texttt{nc} = 3$}
    \label{fig:p6n250c3}
    \end{subfigure}
    \hfill
    \begin{subfigure}[b]{0.24\textwidth}
    \centering
    \includegraphics[width=\textwidth]{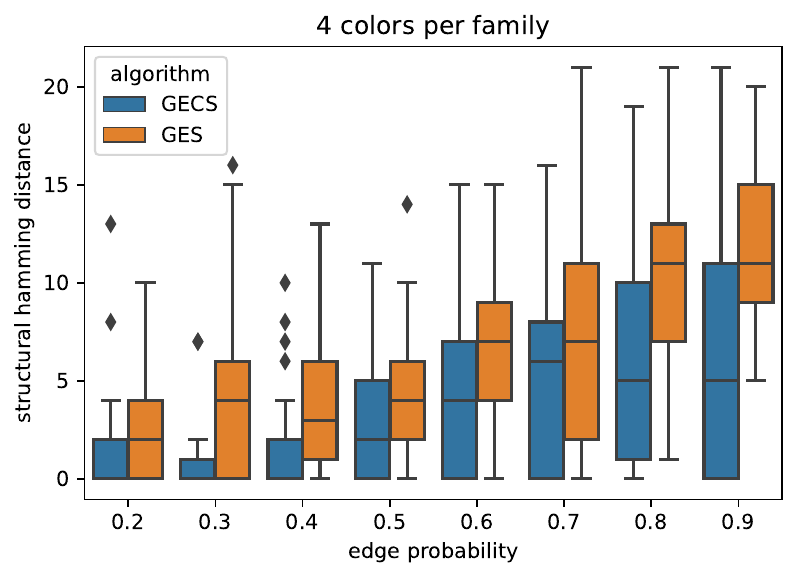}
    \caption{$\texttt{nc} = 4$}
    \label{fig:p6n250c4}
    \end{subfigure}
    \hfill
    \begin{subfigure}[b]{0.24\textwidth}
    \centering
    \includegraphics[width=\textwidth]{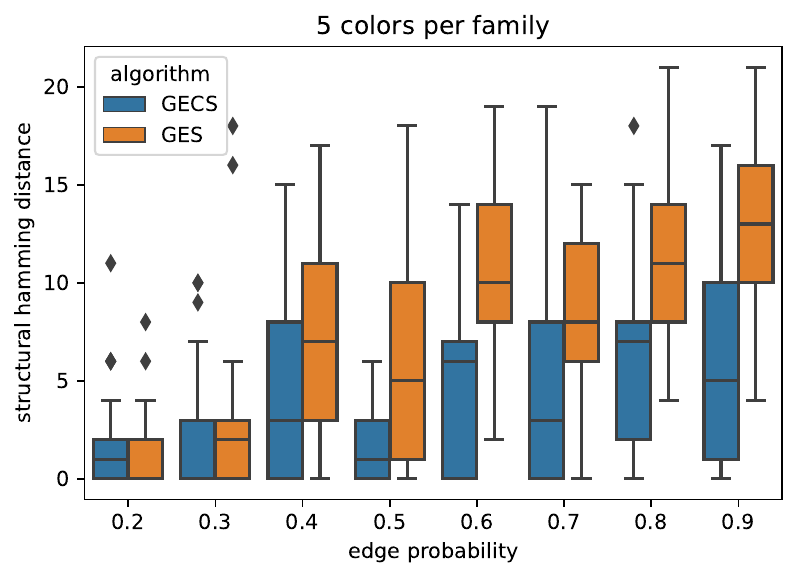}
    \caption{$\texttt{nc} = 5$}
    \label{fig:p6n250c5}
    \end{subfigure}

    \begin{subfigure}[b]{0.24\textwidth}
    \centering
    \includegraphics[width=\textwidth]{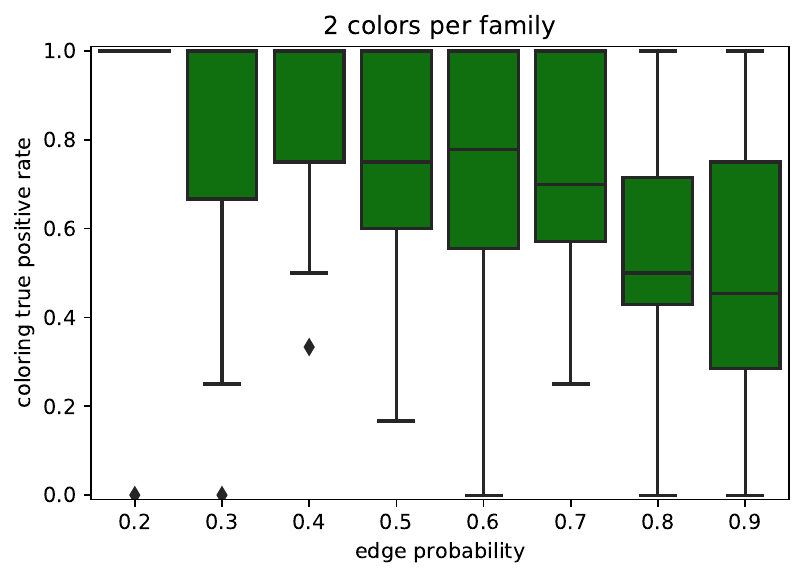}
    \caption{$\texttt{nc} = 2$}
    \label{fig:p6n250c2cTPR}
    \end{subfigure}
    \hfill
    \begin{subfigure}[b]{0.24\textwidth}
    \centering
    \includegraphics[width=\textwidth]{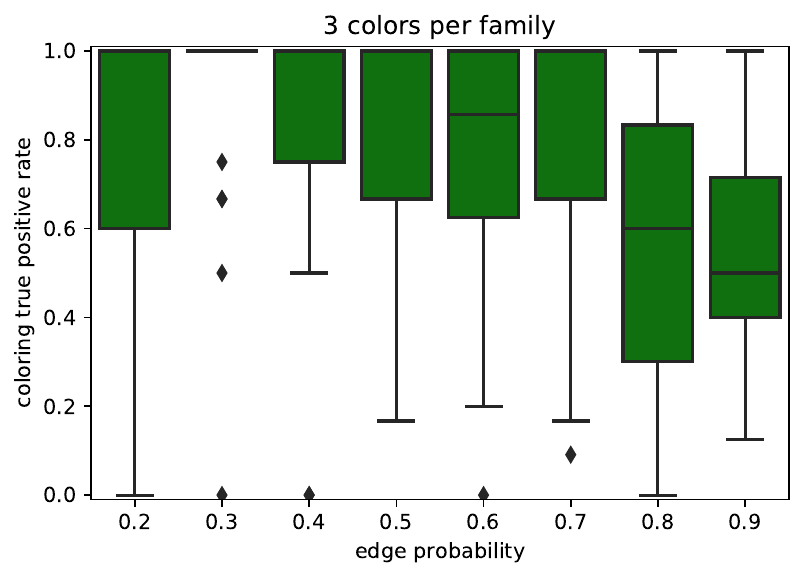}
    \caption{$\texttt{nc} = 3$}
    \label{fig:p6n250c3cTPR}
    \end{subfigure}
    \hfill
    \begin{subfigure}[b]{0.24\textwidth}
    \centering
    \includegraphics[width=\textwidth]{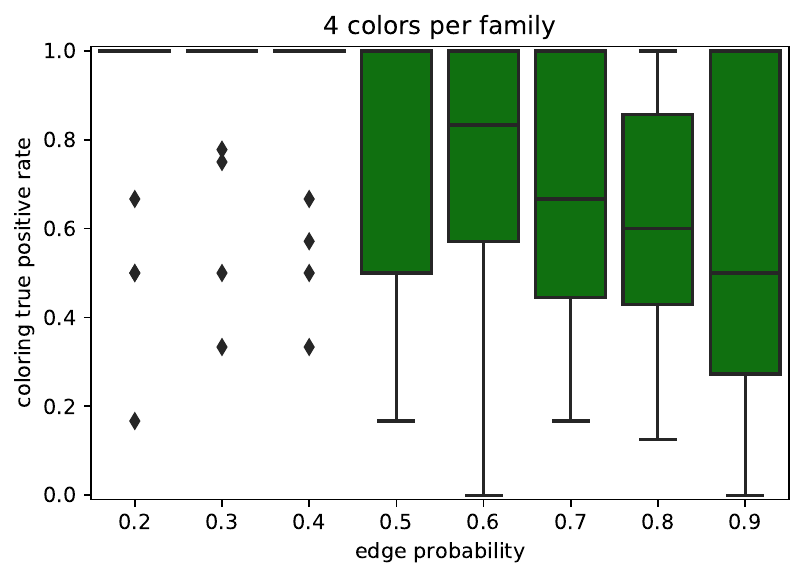}
    \caption{$\texttt{nc} = 4$}
    \label{fig:p6n250c4cTPR}
    \end{subfigure}
    \hfill
    \begin{subfigure}[b]{0.24\textwidth}
    \centering
    \includegraphics[width=\textwidth]{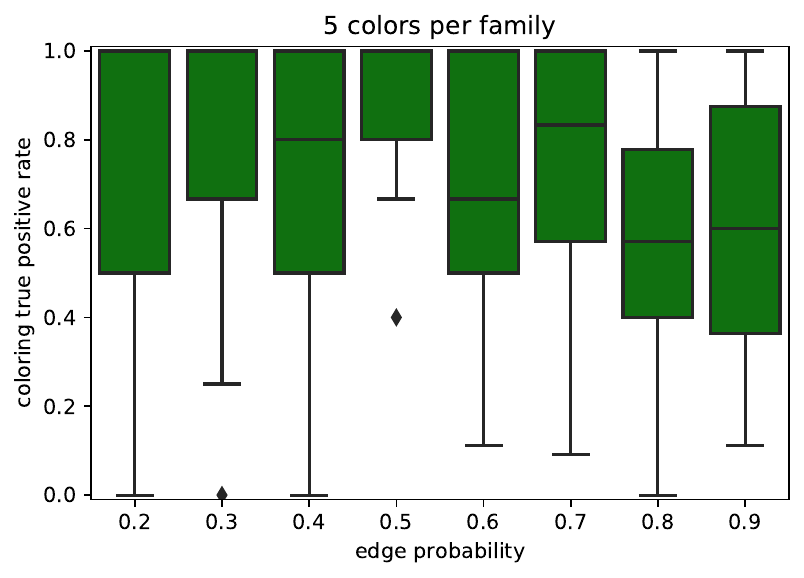}
    \caption{$\texttt{nc} = 5$}
    \label{fig:p6n250c5cTPR}
    \end{subfigure}
\caption{(a)--(d): Structural Hamming distance results for $p = 6$ and sample size $n = 250$. (e)--(h): True positive rates for learned colorings.}
\label{fig:p6n250}
\end{figure}

\begin{figure}
    \begin{subfigure}[b]{0.24\textwidth}
    \centering
    \includegraphics[width=\textwidth]{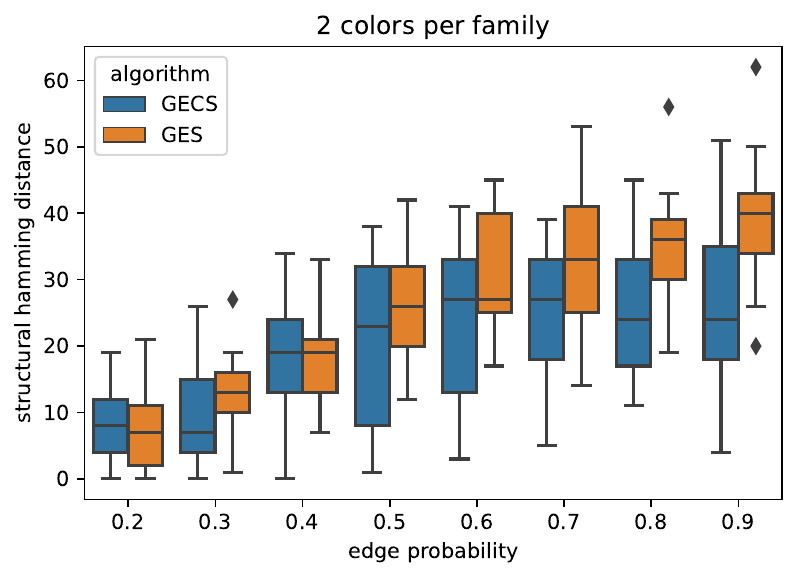}
    \caption{$\texttt{nc} = 2$}
    \label{fig:p10n250c2}
    \end{subfigure}
    \hfill
    \begin{subfigure}[b]{0.24\textwidth}
    \centering
    \includegraphics[width=\textwidth]{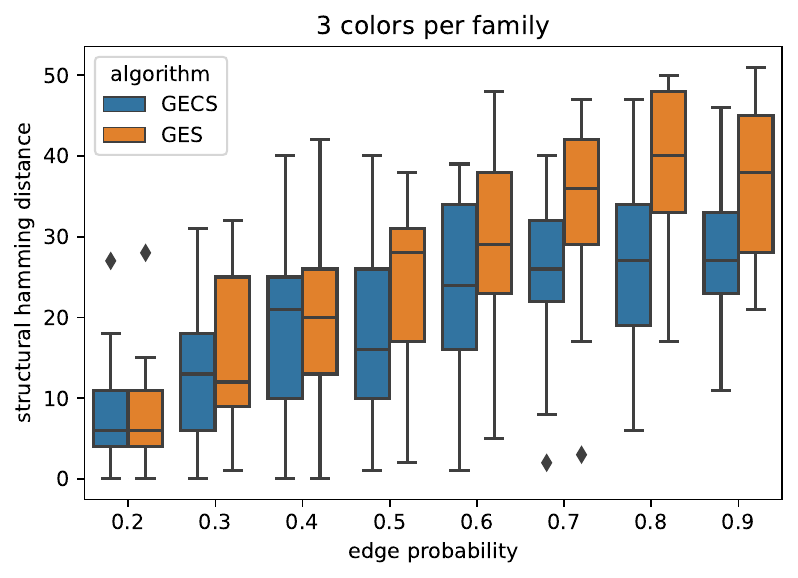}
    \caption{$\texttt{nc} = 3$}
    \label{fig:p10n250c3}
    \end{subfigure}
    \hfill
    \begin{subfigure}[b]{0.24\textwidth}
    \centering
    \includegraphics[width=\textwidth]{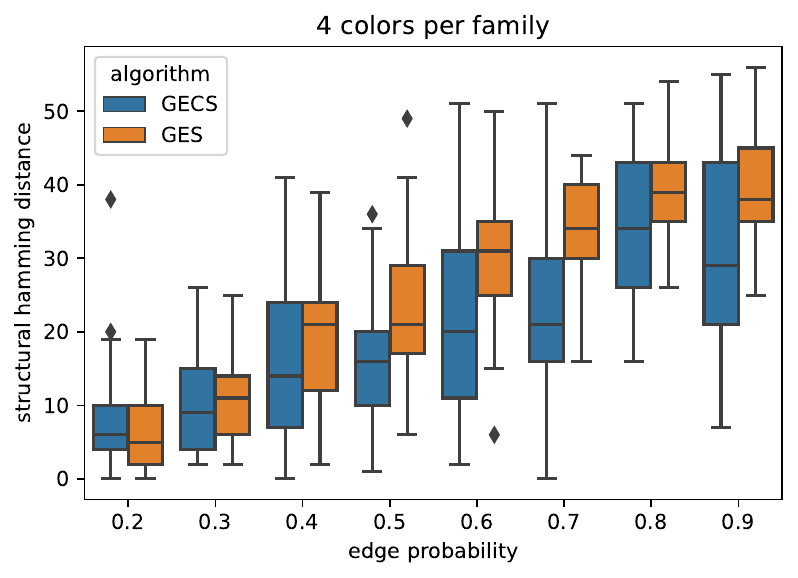}
    \caption{$\texttt{nc} = 4$}
    \label{fig:p10n250c4}
    \end{subfigure}
    \hfill
    \begin{subfigure}[b]{0.24\textwidth}
    \centering
    \includegraphics[width=\textwidth]{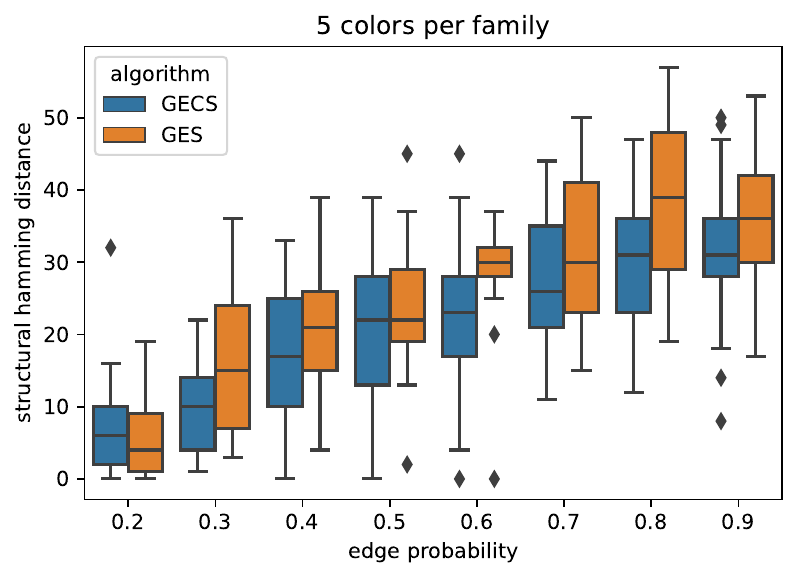}
    \caption{$\texttt{nc} = 5$}
    \label{fig:p10n250c5}
    \end{subfigure}

    \begin{subfigure}[b]{0.24\textwidth}
    \centering
    \includegraphics[width=\textwidth]{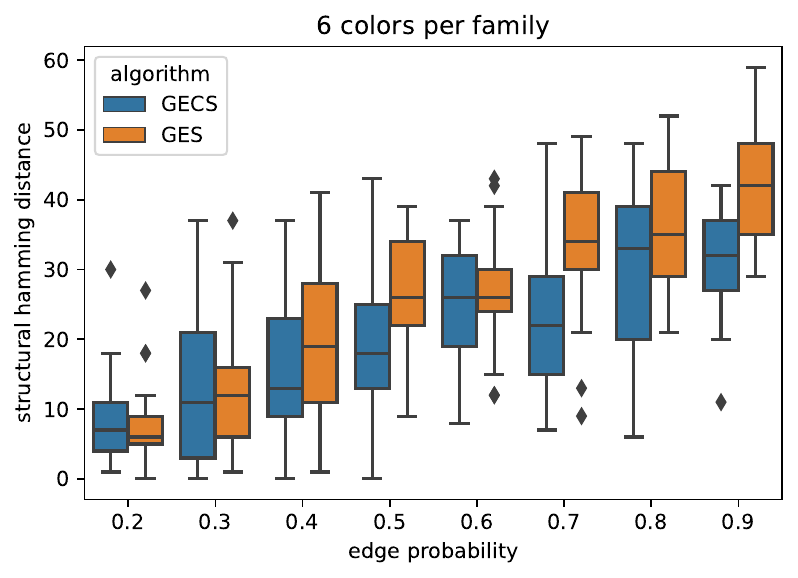}
    \caption{$\texttt{nc} = 6$}
    \label{fig:p10n250c6}
    \end{subfigure}
    \hfill
    \begin{subfigure}[b]{0.24\textwidth}
    \centering
    \includegraphics[width=\textwidth]{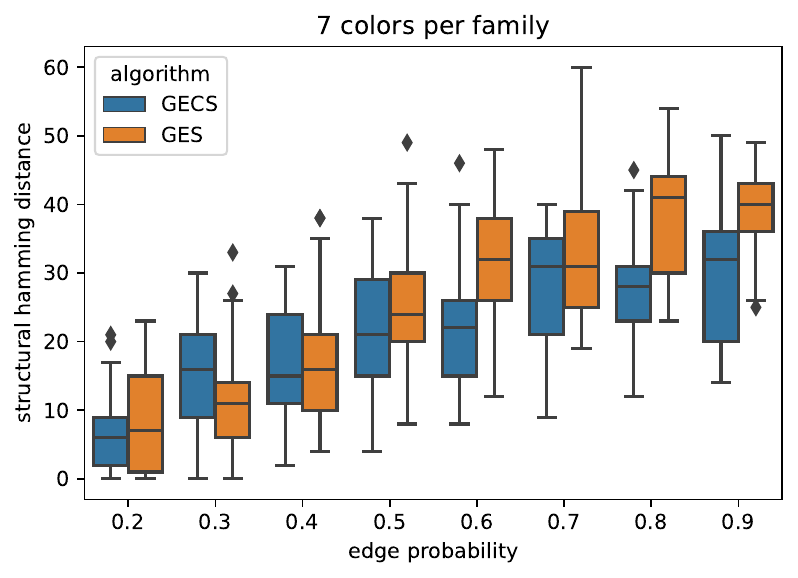}
    \caption{$\texttt{nc} = 7$}
    \label{fig:p10n250c7}
    \end{subfigure}
    \hfill
    \begin{subfigure}[b]{0.24\textwidth}
    \centering
    \includegraphics[width=\textwidth]{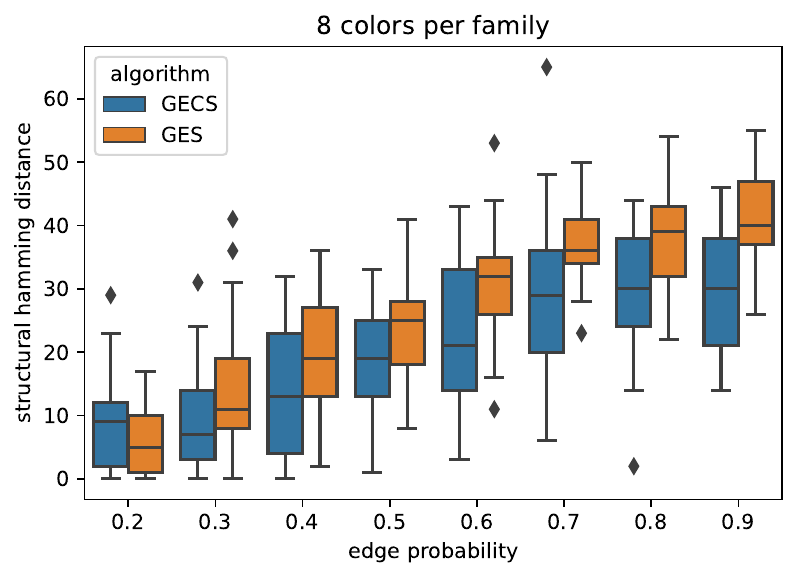}
    \caption{$\texttt{nc} = 8$}
    \label{fig:p10n250c8}
    \end{subfigure}
    \hfill
    \begin{subfigure}[b]{0.24\textwidth}
    \centering
    \includegraphics[width=\textwidth]{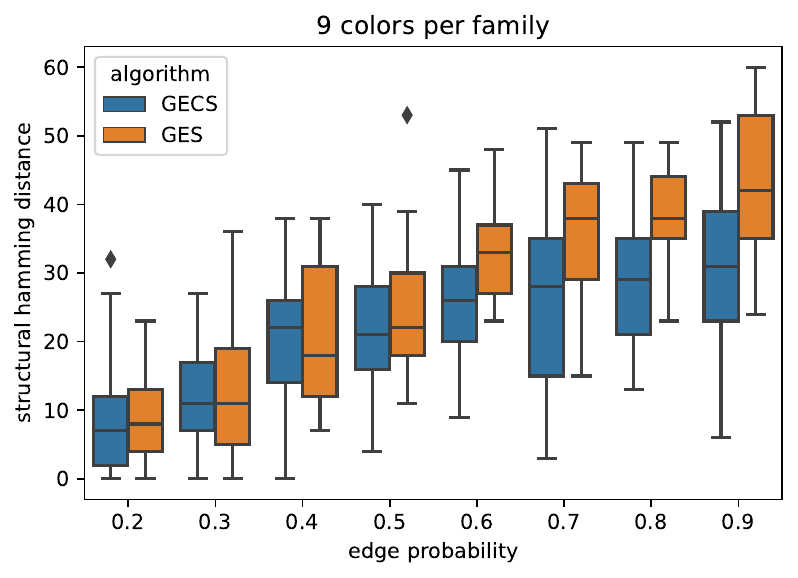}
    \caption{$\texttt{nc} = 9$}
    \label{fig:p10n250c9}
    \end{subfigure}
\caption{Structural Hamming distance results for $p = 10$ and $n = 250$ samples. $\texttt{nc}$ is the pre-specified number of colors per family for the data-generating models.}
\label{fig:p10n250}
\end{figure}

\begin{figure}
    \begin{subfigure}[b]{0.24\textwidth}
    \centering
    \includegraphics[width=\textwidth]{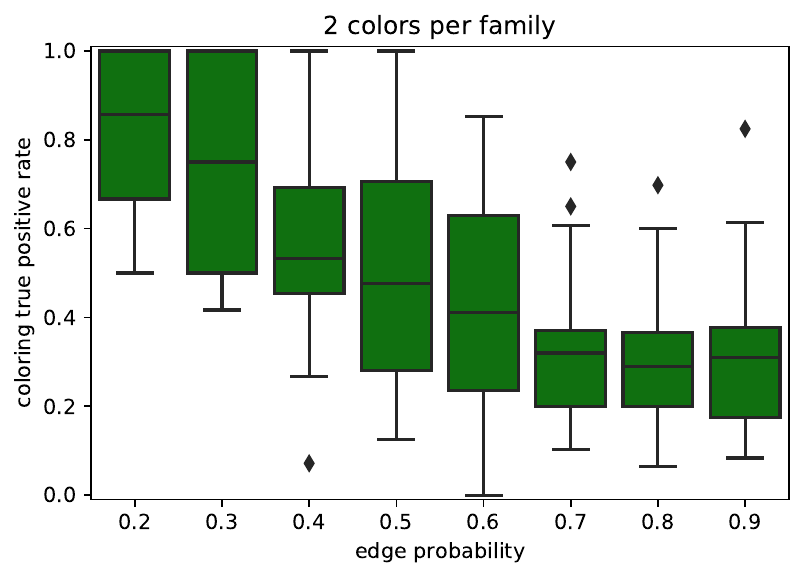}
    \caption{$\texttt{nc} = 2$}
    \label{fig:p10n250c2cTPR}
    \end{subfigure}
    \hfill
    \begin{subfigure}[b]{0.24\textwidth}
    \centering
    \includegraphics[width=\textwidth]{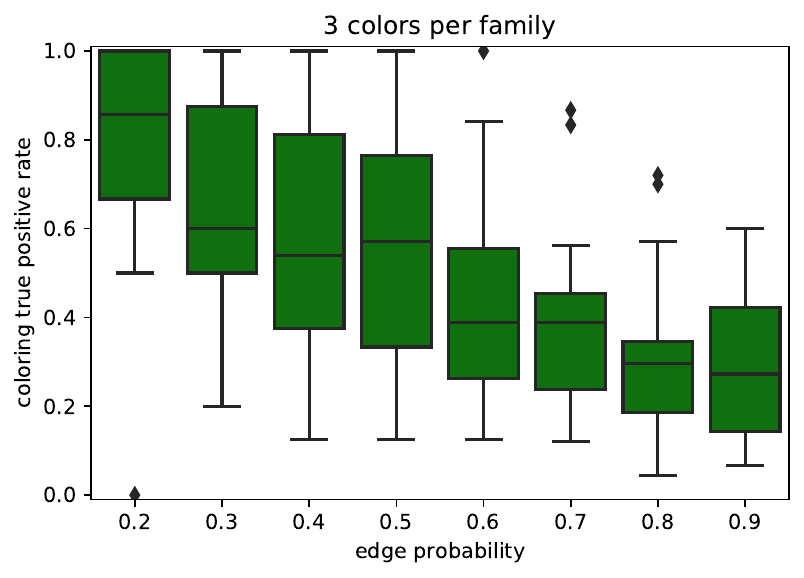}
    \caption{$\texttt{nc} = 3$}
    \label{fig:p10n250c3cTPR}
    \end{subfigure}
    \hfill
    \begin{subfigure}[b]{0.24\textwidth}
    \centering
    \includegraphics[width=\textwidth]{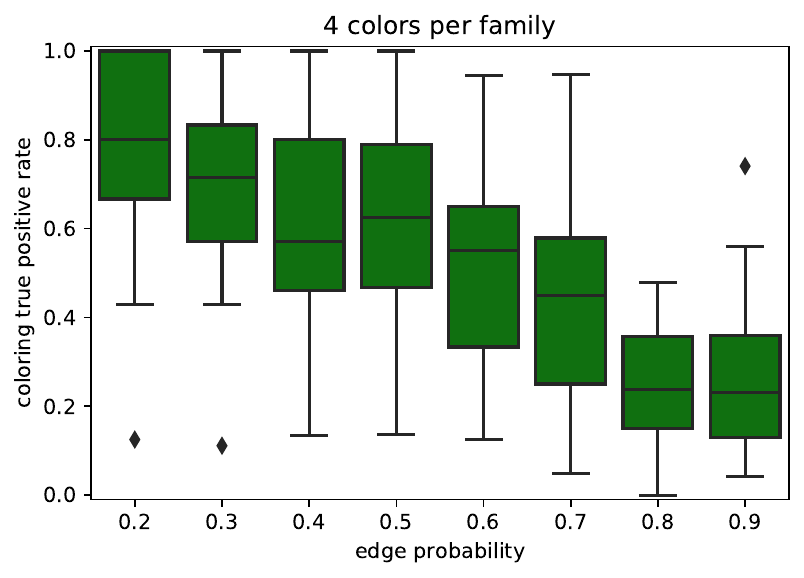}
    \caption{$\texttt{nc} = 4$}
    \label{fig:p10n250c4cTPR}
    \end{subfigure}
    \hfill
    \begin{subfigure}[b]{0.24\textwidth}
    \centering
    \includegraphics[width=\textwidth]{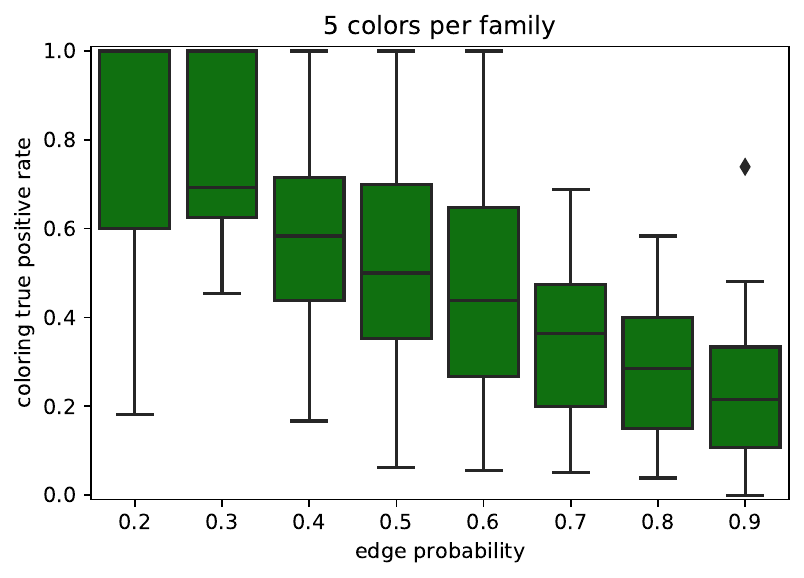}
    \caption{$\texttt{nc} = 5$}
    \label{fig:p10n250c5cTPR}
    \end{subfigure}

    \begin{subfigure}[b]{0.24\textwidth}
    \centering
    \includegraphics[width=\textwidth]{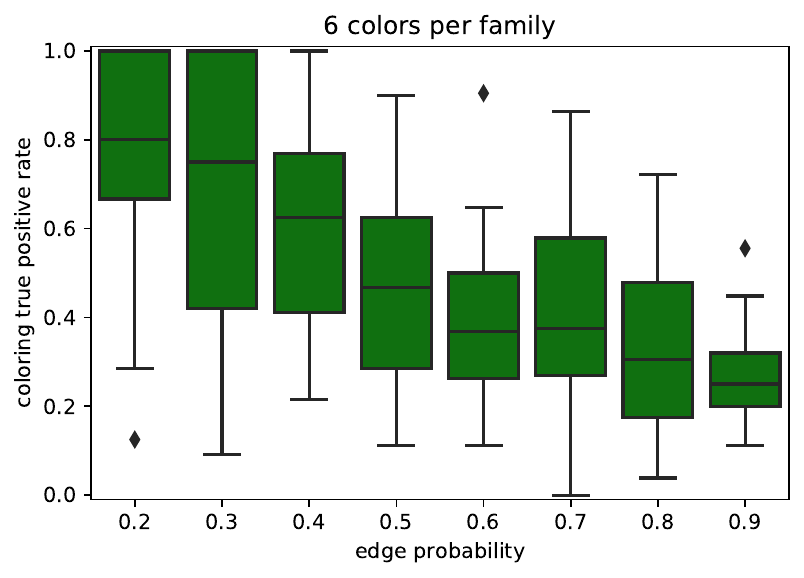}
    \caption{$\texttt{nc} = 6$}
    \label{fig:p10n250c6cTPR}
    \end{subfigure}
    \hfill
    \begin{subfigure}[b]{0.24\textwidth}
    \centering
    \includegraphics[width=\textwidth]{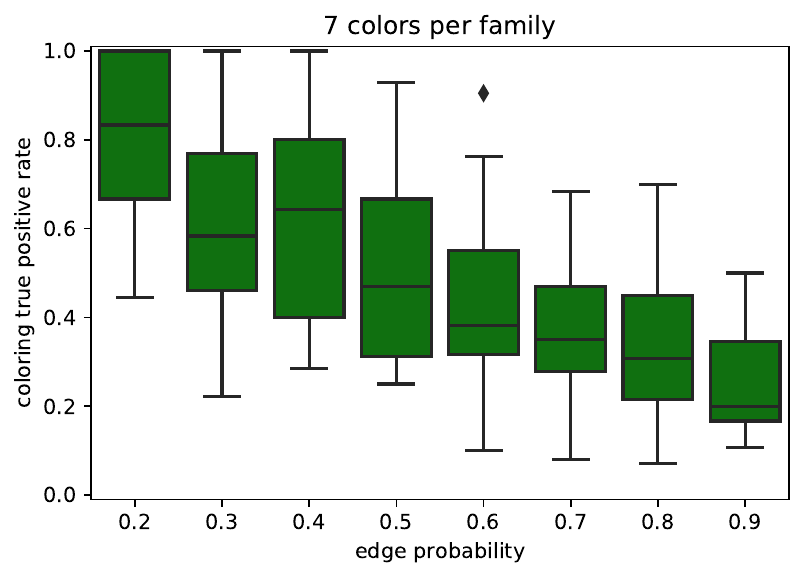}
    \caption{$\texttt{nc} = 7$}
    \label{fig:p10n250c7cTPR}
    \end{subfigure}
    \hfill
    \begin{subfigure}[b]{0.24\textwidth}
    \centering
    \includegraphics[width=\textwidth]{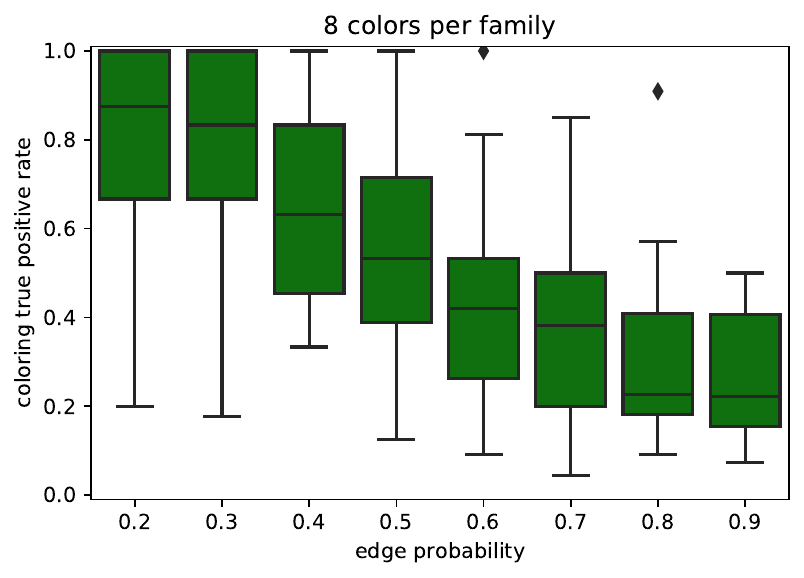}
    \caption{$\texttt{nc} = 8$}
    \label{fig:p10n250c8cTPR}
    \end{subfigure}
    \hfill
    \begin{subfigure}[b]{0.24\textwidth}
    \centering
    \includegraphics[width=\textwidth]{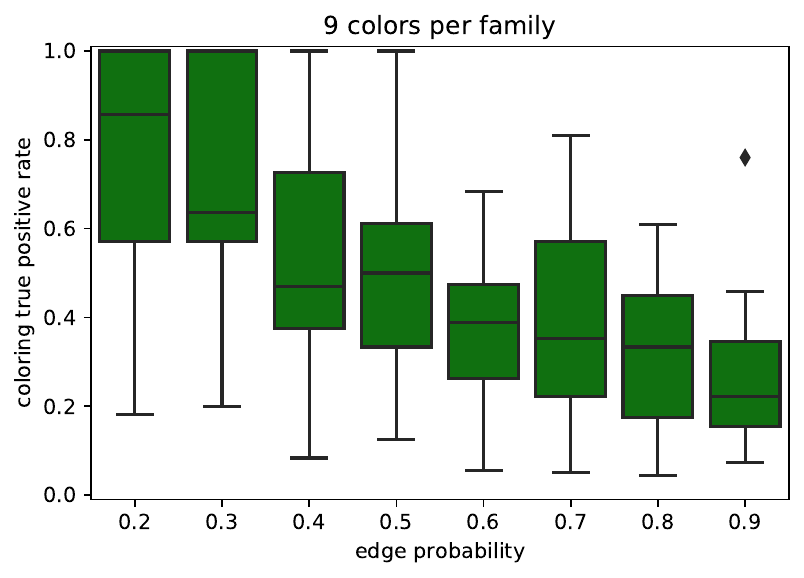}
    \caption{$\texttt{nc} = 9$}
    \label{fig:p10n250c9cTPR}
    \end{subfigure}
\caption{Color sensitivity results for $p = 10$ and $n = 250$ samples. $\texttt{nc}$ is the pre-specified number of colors per family for the data-generating models.}
\label{fig:p10n250cTPR}
\end{figure}

\begin{figure}
    \begin{subfigure}[b]{0.24\textwidth}
    \centering
    \includegraphics[width=\textwidth]{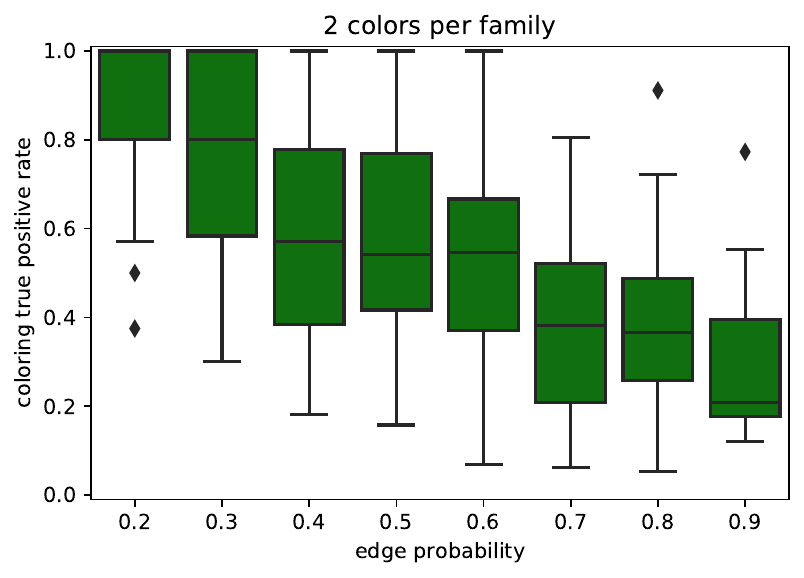}
    \caption{$\texttt{nc} = 2$}
    \label{fig:p10n1000c2cTPR}
    \end{subfigure}
    \hfill
    \begin{subfigure}[b]{0.24\textwidth}
    \centering
    \includegraphics[width=\textwidth]{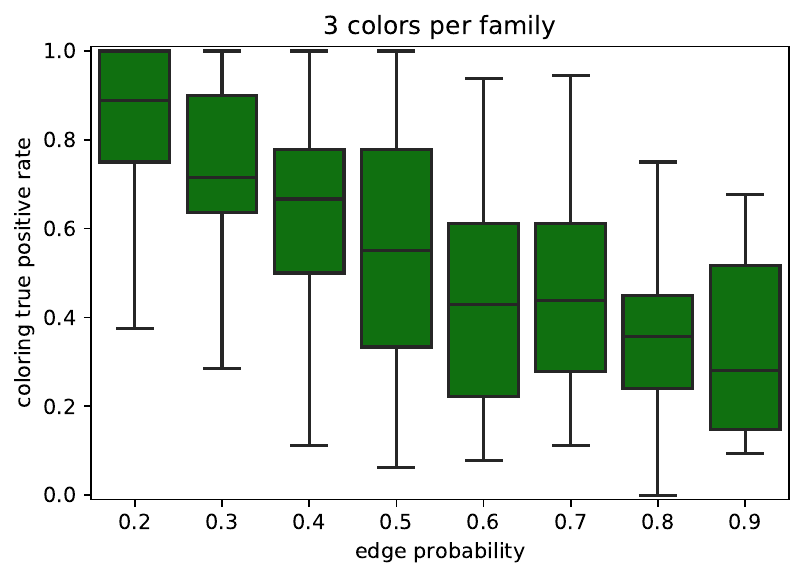}
    \caption{$\texttt{nc} = 3$}
    \label{fig:p10n1000c3cTPR}
    \end{subfigure}
    \hfill
    \begin{subfigure}[b]{0.24\textwidth}
    \centering
    \includegraphics[width=\textwidth]{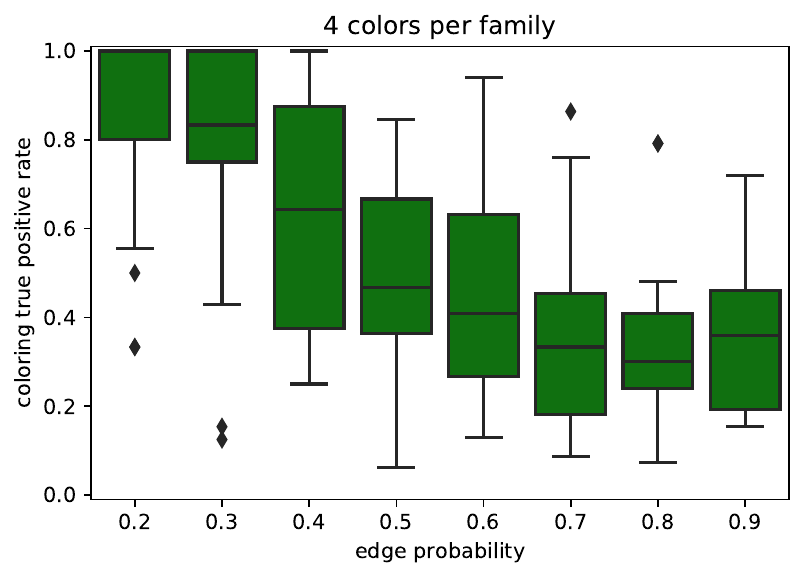}
    \caption{$\texttt{nc} = 4$}
    \label{fig:p10n1000c4cTPR}
    \end{subfigure}
    \hfill
    \begin{subfigure}[b]{0.24\textwidth}
    \centering
    \includegraphics[width=\textwidth]{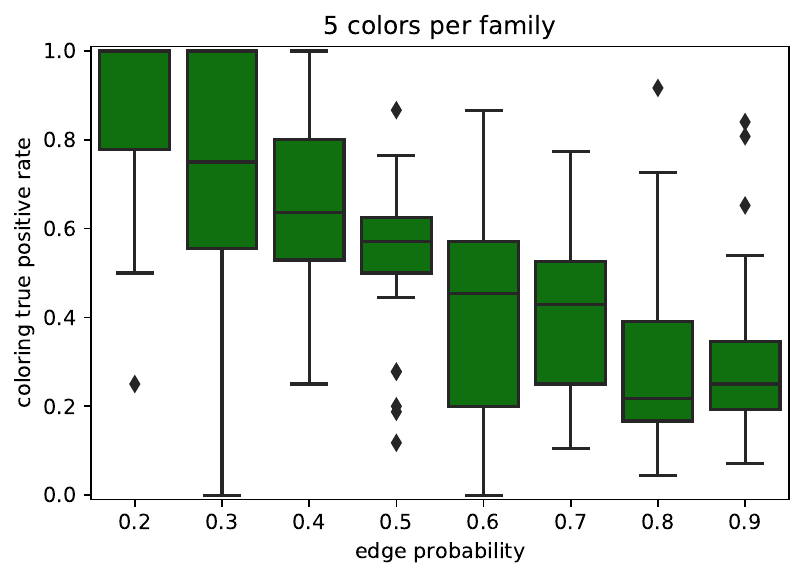}
    \caption{$\texttt{nc} = 5$}
    \label{fig:p10n1000c5cTPR}
    \end{subfigure}

    \begin{subfigure}[b]{0.24\textwidth}
    \centering
    \includegraphics[width=\textwidth]{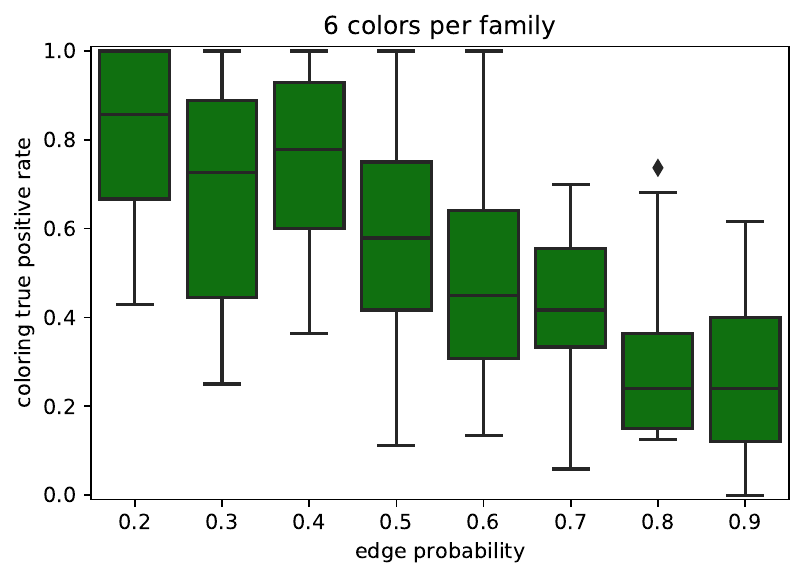}
    \caption{$\texttt{nc} = 6$}
    \label{fig:p10n1000c6cTPR}
    \end{subfigure}
    \hfill
    \begin{subfigure}[b]{0.24\textwidth}
    \centering
    \includegraphics[width=\textwidth]{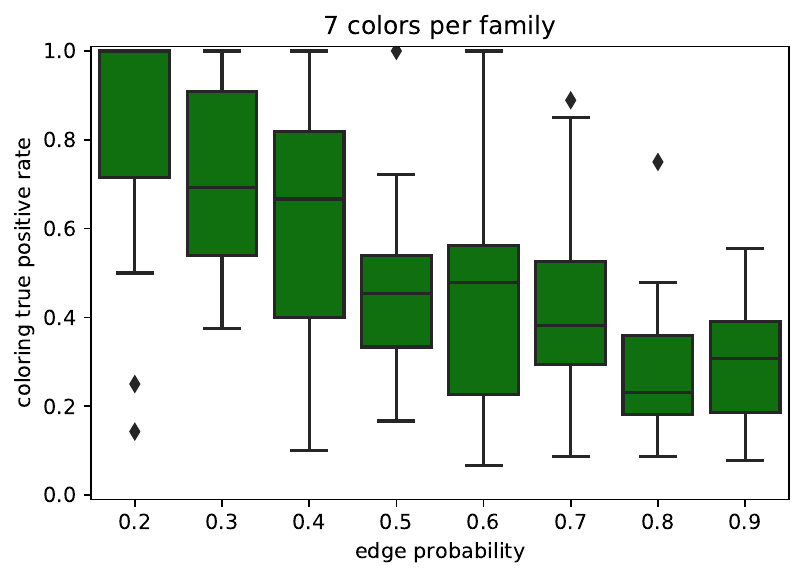}
    \caption{$\texttt{nc} = 7$}
    \label{fig:p10n1000c7cTPR}
    \end{subfigure}
    \hfill
    \begin{subfigure}[b]{0.24\textwidth}
    \centering
    \includegraphics[width=\textwidth]{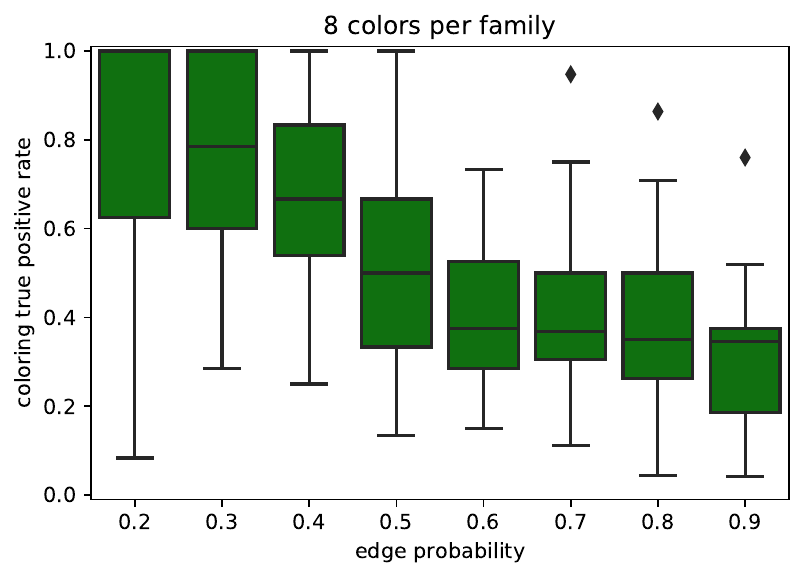}
    \caption{$\texttt{nc} = 8$}
    \label{fig:p10n1000c8cTPR}
    \end{subfigure}
    \hfill
    \begin{subfigure}[b]{0.24\textwidth}
    \centering
    \includegraphics[width=\textwidth]{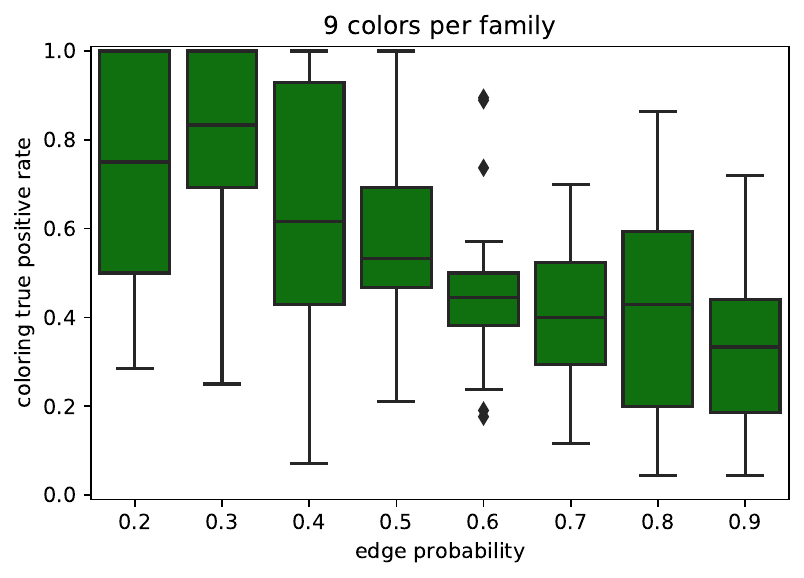}
    \caption{$\texttt{nc} = 9$}
    \label{fig:p10n1000c9cTPR}
    \end{subfigure}
\caption{Color sensitivity results for $p = 10$ and $n = 1000$ samples. $\texttt{nc}$ is the pre-specified number of colors per family for the data-generating models.}
\label{fig:p10n1000cTPR}
\end{figure}

\clearpage

\begin{figure}
    \begin{subfigure}[b]{0.45\textwidth}
    \centering
    \includegraphics[width=\textwidth]{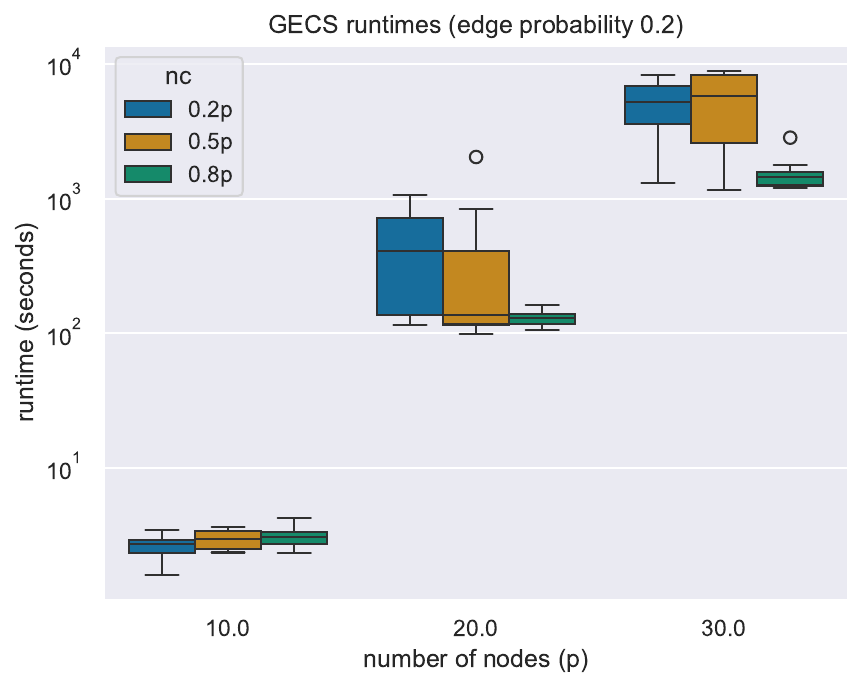}
    \caption{Edge probability $\rho =0.2$}
    \label{fig:GECS_runtimes02}
    \end{subfigure}
    \hfill
    \begin{subfigure}[b]{0.45\textwidth}
    \centering
    \includegraphics[width=\textwidth]{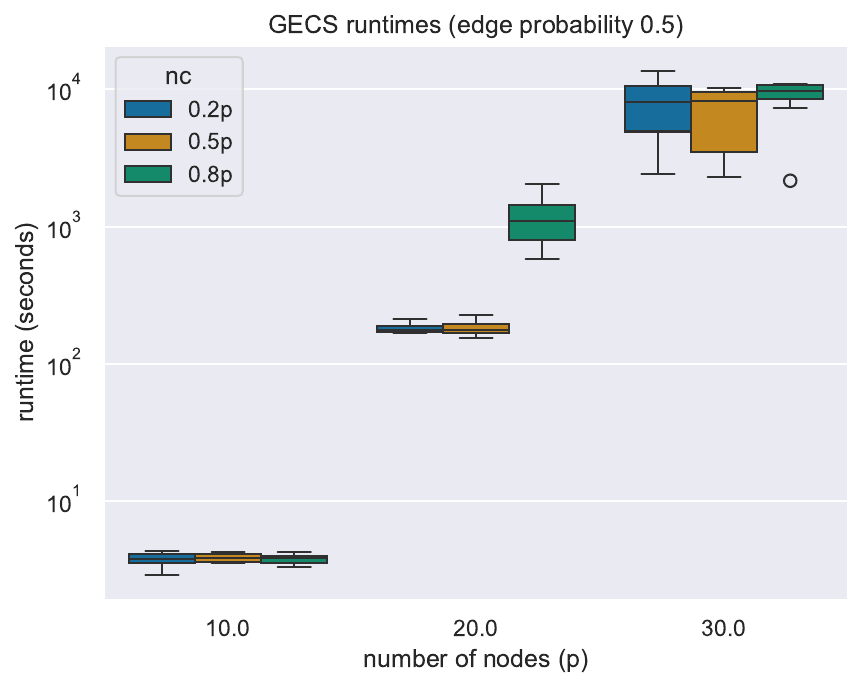}
    \caption{Edge probability $\rho =0.5$}
    \label{fig:GECS_runtimes05}
    \end{subfigure}

\caption{
Empirical runtime estimates of the implementation of GECS used in these experiments. $10$ random BPEC-DAGs were generated for each $\rho\in\{0.2,0.5\}$, $p\in\{10,20,30\}$ and $\texttt{nc}\in\{0.2p, 0.5p,0.8p\}$.  In particular, the number of randomly generated color classes is proportional to the number of nodes in each experiment.
The runtime of GECS (in seconds) for each model using $n=1000$ samples was then recorded. The predominant trend is that the implementation of GECS used in this paper tends to take several seconds for graphs on 10 nodes, several minutes for graphs on 20 nodes and several hours for graphs on 30 nodes.  This exponential growth is to be expected given that the size of the search space of colored DAGs on $p$ nodes is much larger than the size of the space of all DAGs, which is already super-exponential in $p$. Interestingly, while there is no definitive trend in this data regarding the impact of sparsity and the $\texttt{nc}$ coloring parameter on runtime, it does appear that certain pairs $(\rho, \texttt{nc})$ yield better complexity of larger $p$, e.g., $(0.2, 0.8p)$ and $(0.5, 0.2p)$.  This suggests the ratio $\rho/\texttt{nc}$ should have some impact on runtime.  Deducing this effect would require a careful complexity analysis of the submethods in Algorithms~\ref{alg:addColor}-\ref{alg:removeColor}.
}
\label{fig:GECS_runtimes}
\end{figure}

\begin{table}
    \begin{tabular}{ | l | l || l | l | }\hline
    {\bf Node label}   &   {\bf Variable} & {\bf Node label}   &   {\bf Variable} \\\hline
    0   &  fixed acidity & 6   &  total sulfur dioxide\\\hline
    1   &  volatile acidity & 7   &  density\\\hline
    2   &  citric acid & 8   &  pH\\\hline
    3   &  residual sugar& 9   &  sulphates\\\hline
    4   &  chlorides& 10   & alcohol\\\hline
    5   &  free sulfur dioxide & \multicolumn{2}{|l|}{} \\\hline
    \end{tabular}
    \caption{Variable names for the Wine Quality data set BPEC-DAGs (see \Cref{fig:wine}).}
    \label{tab:WineLabels}
\end{table}

\end{document}